\documentclass[onecolumn]{IEEEtran}
\usepackage{soul} 
\usepackage[utf8]{inputenc} 
\usepackage[T1]{fontenc}
\usepackage{url}              % provides \url{...}
\usepackage{cite}             % improves presentation of citations

\usepackage{amsmath}  % Use the [cmex10] option to ensure complicance
\usepackage{mleftright}       % fix to wrong spacing of \left-,
\mleftright                   % \middle- \right-commands 

\usepackage{graphicx}         % provides \includegraphics{...} to
\usepackage{booktabs}         % fixes poor spacing in tables and
\usepackage{amsmath,amsthm,amssymb}
\usepackage{lineno}
\usepackage{geometry}
\usepackage{color}
\usepackage{listings}
\usepackage{algorithm}
\usepackage{algorithmic}
\usepackage{setspace}
\usepackage{caption}  

\usepackage{tikz}  
\usetikzlibrary{arrows.meta}%画箭头用的包
\usepackage{curves}
\usepackage{scalefnt}
\usepackage{ulem}

\newtheorem{theorem}{Theorem}[section]
\newtheorem{lemma}{Lemma}[section]
\newtheorem{cor}{Corollary}[section]
\newtheorem{proposition}{Proposition}[section]
\newtheorem{define}{Definition}[section]
\newtheorem{remark}{Remark}[section]

\newtheorem{example}{Example}[section]

\lstdefinelanguage{Maple}{
	keywords={if, while, do, else, end, for, from, to,then},
	keywordstyle=\color{blue}\bfseries,
	ndkeywords={class, export, boolean, throw, implements, import, this},
	ndkeywordstyle=\color{darkgray}\bfseries,
	identifierstyle=\color{black},
	sensitive=false,
	comment=[l]{//},
	morecomment=[s]{/*}{*/},
	commentstyle=\color{purple}\ttfamily,
	stringstyle=\color{red}\ttfamily,
	morestring=[b]',
	morestring=[b]"
}

\lstdefinelanguage{SOStools}{
	keywords={syms,sosprogram,monomials,sosineq,sossetobj,sossolve,sosgetsol,sospolyvar},
	keywordstyle=\color{blue}\bfseries,
	ndkeywords={syms,sosprogram,monomials,sosineq,sossetobj,sossolve,sosgetsol},
	ndkeywordstyle=\color{blue}\bfseries,
	identifierstyle=\color{black},
	sensitive=false,
	comment=[l]{//},
	morecomment=[s]{/*}{*/},
	commentstyle=\color{purple}\ttfamily,
	stringstyle=\color{red}\ttfamily,
	morestring=[b]',
	morestring=[b]"
}

\makeatletter
\newenvironment{breakablealgorithm}
{% \begin{breakablealgorithm}
		\begin{center}
			\refstepcounter{algorithm}% New algorithm
			\hrule height.8pt depth0pt \kern2pt% \@fs@pre for \@fs@ruled »­Ïß
			\renewcommand{\caption}[2][\relax]{% Make a new \caption
				{\raggedright\textbf{\ALG@name~\thealgorithm} ##2\par}%
				\ifx\relax##1\relax % #1 is \relax
				\addcontentsline{loa}{algorithm}{\protect\numberline{\thealgorithm}##2}%
				\else % #1 is not \relax
				\addcontentsline{loa}{algorithm}{\protect\numberline{\thealgorithm}##1}%
				\fi
				\kern2pt\hrule\kern2pt
			}
		}{% \end{breakablealgorithm}
		\kern2pt\hrule\relax% \@fs@post for \@fs@ruled »­Ïß
	\end{center}
}
\makeatother

\begin{document}

\title{Characterizations of Conditional Mutual Independence: Equivalence and Implication} 

%%%%%%
\author{
	Laigang~Guo,~%\IEEEmembership{Member,~IEEE,}
	Raymond~W.~Yeung,~and Tao~Guo%\IEEEmembership{Fellow,~IEEE,}
%	Xiao-Shan~Gao,~%,~\IEEEmembership{Member,~IEEE}% <-this % stops a space
	\thanks{L. Guo is with the School of Mathematical Sciences, Beijing Normal University, Beijing. Email: lgguo@bnu.edu.cn}% <-this % stops a space
    \thanks{R. W. Yeung is with the Institute of Network Coding and Department of Information Engineering, The Chinese University of Hong Kong, Hong Kong. Email: whyeung@ie.cuhk.edu.hk}
	\thanks{T. Guo is with the School of Cyber Science and Engineering, Southeast University, Nanjing. (Corresponding author) Email: taoguo@seu.edu.cn}
	
%	\thanks{X.-S. Gao is with the Key Laboratory of Mathematics Mechanization, Chinese Academy of Sciences, Beijing. e-mail: (xgao@mmrc.iss.ac.cn).}
%	%\thanks{Manuscript received April 19, 2005; revised September 17, 2014.}
}

\maketitle

%%%%%
%% Abstract: 
%% If your paper is eligible for the student paper award, please add
%% the comment "THIS PAPER IS ELIGIBLE FOR THE STUDENT PAPER
%% AWARD." as a first line in the abstract. 
%% For the final version of the accepted paper, please do not forget
%% to remove this comment!
%%

\doublespacing

\begin{abstract}
Conditional independence, and more generally conditional mutual independence, are central notions in probability theory.
In their general forms, they include functional dependence as a special case. In this paper, we tackle two fundamental problems related to conditional mutual independence.
Let $K$ and $K'$ be two conditional mutual independncies (CMIs) defined on a finite set of discrete random variables. We have obtained a necessary and sufficient condition for i) $K$ is equivalent to $K'$; ii) $K$ implies $K'$.
These characterizations are in terms of a canonical form introduced for conditional mutual independence.
%The characterization of the conditional mutual independence (CMI) and the implication between the CMIs are basic problems in probability theory.
%In this paper, the new algebraic characterization of CMIs is obtained. Based on this new characterization, we proved the necessary and sufficient conditions for the equivalence of CMIs  and the necessary and sufficient conditions for the implication of CMIs.
\end{abstract}


\section{Introduction}
In this paper, all random variables are assumed to be discrete.
Let $X$ be a random variable taking values in an alphabet~$\cal X$.
The probability distribution for $X$ is denoted by 
$\{ p_X(x), x \in {\cal X} \}$, with
$p_X(x) = {\rm Pr}\{X=x\}$.
When there is no ambiguity, $p_X(x)$ will be 
abbreviated as $p(x)$, and $\{p(x)\}$ will be abbreviated as
$p(x)$.  The support of $X$, or equivalently the support of $p(x)$, denoted by $\cal S_X$,
is the set of all $x \in {\cal X}$ such that $p(x) > 0$.
For two random variables $X$ and $Y$, the conditional probability
${\rm Pr}\{ X=x | Y = y \}$ is denoted by $p_{X|Y}(x|y)$, where 
$y \in {\cal S}_Y$. Again, when there is no ambiguity, $p_{X|Y}(x|y)$ will be abbreviated as $p(x|y)$. 

Conditional independence of random variables is a central notion in probability theory \cite{Feller1950}. We begin our discussion by giving the definition for conditional independence of two random variables.

\begin{define}[Conditional Independence]
	For random variables $X_1, X_2$, and $Y$, $X_1$ is independent
	of $X_2$ conditioning on $Y$, denoted by $X_1 \perp X_2 | Y$, if for all $x_1, x_2$, and $y$,
	\begin{equation}
			p(x_1, x_2, y) = \left\{ \begin{array}{ll}
					\frac{p(x_1,y) p(x_2,y)}{p(y)} = p(x_1|y) p(x_2|y) p(y) & \mbox{if $y \in {\cal S}_Y$} \\
					0 & \mbox{otherwise}.
				\end{array} \right.
		\label{qtiphr}
		\end{equation}
%	or equivalently,
%	\begin{equation}
	%	p(x_1, x_1, y) p(y) = p(x_1, y) p(x_2, y) 
	%	\label{qtiphr23}
	%	\end{equation}
	\label{CI_defn}
\end{define}

In the above definition, if $Y$ is a degenerate random variable, i.e., it takes a constant value with probability 1, then we simply say that $X_1$ and $X_2$ are independent. Therefore, conditional independence include independence as a special case.

Note that Definition~\ref{CI_defn} continues to apply when some of the random variables represent a group of random variables instead of a single random variable, where these groups of random varaibles may overlap. For example, if $X_1 = (Z_1, Z_2)$ and $X_2 = (Z_1, Z_3)$, then $Z_1$ is a random variable common to $X_1$ and $X_2$. When a group of random variables is empty, it is taken to be a degenerate random variable.

When three or more random variables are conditionally independent, we need to distinguish  two types of conditional independence.

\begin{define}[Conditional Mutual Independence]
	For random variables $X_1, X_2, \ldots, X_n$, and $Y$, where $n \ge 3$, $X_1, X_2,$ $\ldots, X_n$ are mutually independent conditioning on $Y$, denoted by $\perp \hspace{-1mm} (X_1, X_2, \ldots X_n) | Y$, if 
	for all $x_1, x_2, \ldots x_n$, and $y$,
	\begin{equation}
		p(x_1, x_2, \ldots, x_n, y) = \left\{ \begin{array}{ll}
			\frac{p(x_1,y) p(x_2,y) \cdots p(x_n,y)}{p(y)} = p(x_1|y) p(x_2|y) \ldots p(x_n|y) p(y) & \mbox{if $y \in {\cal S}_Y$} \\
			0 & \mbox{otherwise.}
		\end{array} \right.
	\end{equation}
	%	or equivalently,
	%	\begin{equation}
		%	p(x_1, x_2, \ldots, x_n, y) p(y) = p(x_1,y) p(x_2,y) \cdots p(x_n,y).
		%	\end{equation}
	\label{CMI_defn}
\end{define}

\begin{define}[Conditional Pairwise Independence]
	For random variables $X_1, X_2, \ldots, X_n$, and $Y$, where $n \ge 3$, $X_1, X_2, \ldots, X_n$ are pairwise independent conditioning on $Y$ if 
	$X_i \perp X_j |Y$ for all $1 \le i < j \le n$.
	\label{CPI_defn}
\end{define}

Note that conditional mutual independence implies conditional pairwise independence, but not vice versa. In Definitions~\ref{CMI_defn} 
and~\ref{CPI_defn}, if we relax the requirement $n \ge 3$ to $n \ge 2$, then the two definitions coincide when $n=2$, and so we do not have to distinguish between mutual independence and pairwise independence between two random variables (i.e., $X_1$ and $X_2$).

It can be shown that a conditional mutual independency (CMI) can be expressed as a collection of conditional independencies (CIs). See Appendix \ref{app:A} for a proof. Thus, a collection of CMIs can be expressed as a collection of CIs, and vice versa.
As an example, the CMI
$\perp \hspace{-1mm} (X_1, X_2, X_3) | Y$ is equivalent to 
\begin{enumerate}
\item 
the CIs $X_1 \perp (X_2, X_3) | Y$ and $X_2 \perp X_3 | Y$; or
\item 
the CIs $X_1 \perp (X_2, X_3) | Y$ and $X_2 \perp X_3 | (Y, X_1)$.
\end{enumerate}
In general, a CMI can be expressed as more than one collection of CIs.

This paper focuses on conditional mutual independence, specifically on a single CMI (as opposed to a collection of CMIs). To facilitate the discussion, in the rest of the paper, we will adopt Definition~\ref{CMI_defn} as the definition for conditional mutual independence but include the case $n=2$. In other words, conditional independence is regarded as a special case of conditional mutual independence.

\begin{define}
For random variables $X$ and $Y$, $X$ is a function of a random variable $Y$ if there exists a function $f: {\cal Y} \rightarrow {\cal X}$ such that
\[
p(x|y) = \left\{ \begin{array}{ll}
	1 & \mbox{if $x = f(y)$} \\
	0 & \mbox{otherwise,}
\end{array} \right.
\]
for all $x \in {\cal X}$ and all $y \in {\cal S}_Y$. If so, we write $X = f(Y)$.
\end{define}

As will be discussed in depth, by allowing the random variables in different groups of random variables in a CMI to overlap, conditional mutual independence in fact incorporates functional dependence. To illustrate this, 
consider the CI $X_1 \perp X_2 | Y$. If $X_1 = X_2 = Z$, then the CI is equivalent to the functional dependency (FD) ``$Z$ is a function of $Y$''. If $X_1 = (Z_1, Z_2)$ and $X_2 = (Z_1, Z_3)$, then the CI is equivalent to the FD ``$Z_1$ is a function of $Y$'' and the CI $Z_2 \perp Z_3 | (Y, Z_1)$.

Let $K$ and $K'$ be two CMIs. The following are two very basic questions:
\begin{enumerate}
	\item Is $K$ equivalent to $K'$?
	\item If $K$ and $K'$ are not equivalent, does one of them imply the other?
\end{enumerate}
Rather surprisingly, these questions are highly nontrivial. In this paper, we provide the answers by proving necessary and sufficient conditions for ``$K$ is equivalent to $K'$'' and ``$K$ implies $K'$''.

The {\it Shannon entropy} measures the information or the uncertainty contained in a random variable. The Shannon entropy and related information measures are a very powerful tool for tackling problems that involve conditional independence. By employing these information measures, instead of working directly on the underlying distribution of the random variables, which can be very tedious, one can apply the rich set of operations associated with these information measures and greatly simplify the proofs.  
The only limitation of this tool is that the entropies of the random variables involved must be finite. This point will be expounded in the next two sections. See also the discussion in the concluding section.

Related works in information theory can be found in \cite{KawabataY1992} \cite{Yeung2002} \cite{Yeung2019} \cite{ChanChenYeung2020}. For a given set of random variables $\cal A$, a CMI on $\cal A$ is {\it full} (with respect to $\cal A$) if all the random variables in $\cal A$ are involved in the CMI.
Such a CMI is called a full CMI, or FCMI.
For example, if ${\cal A} = \{ X, Y, Z, T \}$, then $X \perp (Z, T) | Y$ and $\perp(X, Y, T) | Z$ are FCMIs, while 
$Y \perp T | X$ is not because $Z$ is not involved.
Building on the results in \cite{KawabataY1992}, a complete set-theoretic characterization of a collection of FCMIs was obtained in \cite{Yeung2002}. In the same work, this characterization was applied to a Markov random field (MRF) satisfying the strong Markov property, which can be regarded as a collection of FCMIs. The latter result was further developed in \cite{Yeung2019}, where the smallest graph that can always represent a subfield of an MRF was determined in closed form. In \cite{ChanChenYeung2020}, characterization and classification of a collection of full conditional independencies via the structure of the induced set of vanishing atoms were obtained. 

In graphical models, there is a parallel but independent line of research \cite{GeigerPearl1993} \cite{Sadeghi2013} \cite{Sadeghi2016}. In \cite{GeigerPearl1993}, full conditional independence\footnote{
Although a full conditional independency (FCI) is a special case of 
an FCMI, an FCMI can always be expressed as a collection of FCIs. See the remark in Appendix~\ref{app:A}. 
} (termed ``{\it saturated} conditional independence'' therein) was shown to be axiomatizable. In \cite{Sadeghi2016}, which was  built on \cite{Sadeghi2013}, an algorithm producing the smallest graph that can always represent a subfield of an MRF was given (cf.\ the discussion on \cite{Yeung2019} in the last paragraph).

The rest of the paper is organized as follows. In Section~\ref{sec-Pre}, we 
%define the Shannon entropy and 
present the properties of the related information measures that are useful for proving the results in this paper. Sections~\ref{sec-char} and \ref{sec-imply} contain the main results of the paper.
In Section~\ref{sec-char}, we prove a necessary and sufficient condition for the equivalence of two CMIs. 
In Section~\ref{sec-imply}, we tackle the implication problem of a CMI and prove a necessary and sufficient condition for a CMI to imply another CMI. Conclusion and discussion are
given in Section~\ref{sec-conc}.

\section{Preliminaries}
\label{sec-Pre}
A CMI can readily be expressed in terms of the Shannon entropy (hereafter ``entropy''). As such, it is very useful for tackling probability problems whose random variables satisfy certain conditional mutual independence conditions. Moreover, the very rich set of operations associated with entropy and related information measures can very often   greatly simplify the proofs.
In this section, 
%we first define entropy and related information measures. Then 
we state without proof some important properties of these information measures that will be used in proving the results. For a comprehensive treatment of this topic, we refer the reader to \cite[Ch.~2]{Yeung2008}.

% \begin{define}
% 	The entropy $H(X)$ of a random variable $X$ is defined as
% 	\begin{equation}
% 	H(X) = - \sum_x p(x) \log p(x).
% 	\label{def:entropy}
% 	\end{equation}
% \end{define}

% \begin{define}
% 	The joint entropy $H(X,Y)$ of random variables $X$  and $Y$ is defined as
% 	\begin{equation}
% 		H(X,Y) = - \sum_{x,y} p(x,y) \log p(x,y),
% 	\end{equation}
% \end{define}

% In the definitions of entropy and other information measures,
% we adopt the convention that summation is taken over the corresponding
% support.  Such a convention is necessary because, for example, $p(x) \log p(x)$ in (\ref{def:entropy})
% is undefined if $p(x) = 0$. The base of the logarithm in these definitions can be chosen to be any convenient real number greater than 1.  

Denote the set $\{1,2,\ldots,n\}$ by $\mathcal{N}_n$. To simplify notation, we will use $X_\alpha$ to denote $(X_i, i\in \alpha)$, where $\alpha \subseteq \mathcal{N}_n$.
It is well known that entropy satisfy the polymatroidal axioms \cite{Fujishige1978}:
\begin{enumerate}
\item 
$H(X_\alpha) \ge 0$
\item 
$H(X_\alpha) \le H(X_\beta) \ \mbox{if} \ \alpha \subseteq \beta$
\item 
$H(X_\alpha) + H(X_\beta) \ge H(X_{\alpha \cup \beta}) + H(X_{\alpha \cap \beta})$.
\end{enumerate}
where $\alpha, \beta \subseteq {\cal N}_n$.

In information theory (see, e.g. \cite{Gallager68} \cite{CoverThomas1999} \cite{Yeung2008}), in addition to entropy, the following information measures are defined:
\[
\hspace{-2cm} 
\begin{array}{lll}
	\mbox{Conditional Entropy} & &
	H(X|Y) = H(X,Y) - H(Y) \\
	\mbox{Mutual Information} & &
	I(X;Y) = H(X) - H(X|Y) \\
	\mbox{Conditional Mutual Information} & &
	I(X;Y|Z) = H(X|Z) - H(X|Z,Y).
\end{array}
\]
The above definitions are valid provided that all the terms on the right hand side are finite. Note that conditional mutual information is the most general form of Shannon's information measures:
\begin{itemize}
\item 
if $Z$ is a degenerate random variable, then $I(X;Y|Z) = I(X;Y)$;
\item 
if $X = Y$, then $I(X;Y|Z) = H(X|Z)$;
\item 
if $X = Y$ and $Z$ is a degenerate random variable, then $I(X;Y|Z) = H(X)$.
\end{itemize}	

Entropy, mutual information, and their conditional versions are collectively referred to as {\it Shannon's information measures}. It is well known that all Shannon's information measures are nonnegative. These inequalities are known as the {\it basic inequalities} in information theory, and it can be shown that they are equivalent to the polymatroidal axioms \cite[Appendix~14.A]{Yeung2008}. However, there exist constraints on Shannon's information measures beyond the basic inequalities, known as {\it non-Shannon-type inequalities} \cite{ZhangY97}\cite{ZhangY98}.

In the following, we state without proof some important properties of Shannon's information measures that will be used in the rest of the paper.

\begin{proposition}[Chain Rule for Conditional Entropy]
	\[
	H(X_1, X_2, \cdots, X_n|Y) = \sum_{i=1}^n H(X_i|X_1, \cdots, X_{i-1},Y).
	\label{eqn:chain_rule_H|Y}
	\]
	\label{prop:chain_rule_H|Y}
\end{proposition}

\begin{proposition}[Chain Rule for Conditional Mutual Information]
	\[
	I(X_1, X_2, \cdots, X_n;Y|Z) = \sum_{i=1}^n I(X_i;Y|Z, X_1, \cdots, X_{i-1}).
	\label{eqn:chain_rule_I}
	\]
	\label{prop:chain_rule_I}
\end{proposition}

\begin{cor}
$I(X_1, X_2; Y | Z) \ge I(X_1; Y | Z).$
\end{cor}

% \begin{proof}
% The corollary is proved by considering
% \[
% I(X_1, X_2; Y | Z) = I(X_1; Y | Z) + I(X_2; Y | Z, X_1) 
% \ge I(X_1; Y | Z),
% \]
% where the inequality above holds because $I(X_2; Y | Z, X_1)$ is nonnegative.
% \end{proof}

\begin{proposition}
	$X_1$ and $X_2$ are independent conditioning on $Y$ if and only if $I(X_1;X_2|Y) = 0$.
\end{proposition}

This proposition asserts that conditional independence can be completely characterized by setting the corresponding conditional mutual information to zero. Alternatively, since
\begin{eqnarray}
	I(X_1;X_2|Y) & = & H(X_1|Y) - H(X_1|Y, X_2) \nonumber \\
	& = & [H(X_1, Y) - H(Y)] - [H(X_1, X_2, Y) - H(X_2, Y)] \nonumber \\
	& = & H(X_1, Y) + H(X_2, Y) - H(X_1, X_2, Y) - H(Y),
	\label{I}
\end{eqnarray}
a CI is equivalent to setting the corresponding linear combination of joint entropies to zero.

\begin{proposition}
	$X$ is a function of $Y$ if and only if $H(X|Y) = 0$.
\end{proposition}

\begin{proposition}
	$X_1, X_2, \ldots X_n$ are mutually independent conditioning on $Y$ if and only if
	\[
	H(X_1, X_2, \ldots, X_n|Y) = \sum_{i=1}^n H(X_i|Y) .
	\]
\end{proposition}

This proposition asserts that conditional mutual independence can be completely characterized in terms of conditional entropy.

\begin{proposition}[Conditioning Does Not Increase Entropy]
\[
H(Y|X, Z) \leq H(Y|Z),
\]
	with equality if and only if $X$ and $Y$ are independent conditioning on $Z$.
\end{proposition}

\begin{proposition}[Independence Bound for Conditional Entropy]
	\[
	H(X_1, X_2, \ldots, X_n |Y) \le \sum_{i=1}^n H(X_i|Y)
	\]
	with equality if and only if $X_i$, $i = 1, 2, \ldots, n$
	are mutually independent conditioning on $Y$.
\label{thm:indep_bound}
\end{proposition}

\section{Characterization of a CMI}
\label{sec-char}
%In this section, we give characterization of CMIs based on the $s$-variables.
%Motivated by \cite{Yeung2002,ChanChenYeung2020} which characterize and classify a list of full conditional independences via the structure of the induced set of vanishing atoms, in this section, we obtain a characterization of conditional mutual independencies (CMIs), which provides new insight on understanding the theory in \cite{Yeung2002,Yeung2019,ChanChenYeung2020}.

In this section, we seek a complete characterization of a CMI. 
We first define the {\it pure form} of a CMI $K$, denoted by ${\rm pur}(K)$, and for a CMI $\tilde{K}$ in pure form, we define its {\it canonical form}, denoted by ${\rm can}(\tilde{K})$. Then for two CMIs $K$ and $K'$, we establish in Theorem~\ref{mainmainthm} that ``$K$ is equivalent to $K'$'' if and only if ${\rm can}({\rm pur}(K)) = {\rm can}({\rm pur}(K'))$.

In the rest of the  paper, unless otherwise specified, all 
information expressions involve some or all of the random variables $X_1,X_2, \ldots,X_n$. The value of $n$ will be specified when necessary.  We adopt the convention that an empty collection of random 
variables, denoted by $X_{\emptyset}$,
is a degenerate random variable that takes a constant value with probability 1. Accordingly, we write $X_\emptyset = \mbox{constant}$. This way, for any collection $X_A$
of random variables, we have $H(X_A|X_\emptyset) = H(X_A)$
and $H(X_{\emptyset}|X_A)=0$.

We remarked that in order to employ Shannon's information measures, it is required that they all take finite values. We now elaborate this point in our setting.
From Proposition~\ref{thm:indep_bound}, we have
\[
H(X_1, X_2, \ldots, X_n) \le \sum_{i=1}^n H(X_i).
\]
As a consequence, if $H(X_i) < \infty$ for all $i \in {\cal N}_n$, then  $H(X_{{\cal N}_n}) < \infty$. Now, for any $\alpha \subseteq {\cal N}_n$, 
\[
H(X_{{\cal N}_n}) = H(X_\alpha) + H(X_{{\cal N}_n \backslash \alpha} | X_\alpha ),
\]
which implies
\[
H(X_\alpha) = H(X_{{\cal N}_n}) - H(X_{{\cal N}_n \backslash \alpha} | X_\alpha ) \le H(X_{{\cal N}_n}) < \infty.
\]
In other words, if $H(X_i)$ is finite for all $i \in {\cal N}_n$, then 
$H(X_\alpha)$ is finite for all $\alpha \subseteq {\cal N}_n$. In light of (\ref{I}), this in turn ensures that all Shannon information measures involving $X_1, X_2, \ldots, X_n$ are finite. We will make this assumption in the rest of the paper.

\begin{define}\label{mainformula}
	For random variables $X_C$ and $X_{Q_i}$, $1 \le i\le k$, where $k\ge0$, let
	\[
	J(X_{Q_i},1\le i\le k|X_C)=\sum\limits_{i=1}^{k} \, \left[ H(X_{Q_i}|X_C)-H(X_{Q_i},1\le i\le k|X_C) \right] .
	\]
\end{define}

\begin{proposition}
\label{J=0}
For $k\ge 2$, $J(X_{Q_i},1\le i\le k|X_C)\ge0$, with equality if and only if
$X_{Q_i}$, $1\le i\le k$ are mutually independent conditional on $X_C$.
\end{proposition}

This proposition is elementary. The reader may see, for example \cite[Theorem 2.39]{Yeung2008}, for a proof.

\begin{proposition}
\label{J=0b}
	For $k = 0, 1$, $J(X_{Q_i},1\le i\le k|X_C) = 0$.
\end{proposition}

\begin{proof}
When $k=0$, the set $\{i:1\le i\le k\}$ becomes the empty set, and $(X_{Q_i},1\le i\le k)$ becomes $X_{\emptyset}$. Accordingly, 
$J(X_{Q_i},1\le i\le k|X_C)$ becomes
\begin{flalign}
	J(X_{Q_i},1\le i\le 0|X_C)=&\sum\limits_{i=1}^{0}H(X_{Q_i}|X_C)-H(X_{Q_i},1\le i\le0|X_C) \nonumber\\
	=&\ 0-H(X_\emptyset|X_C)\\
	=&\ 0. \nonumber
\end{flalign}
When $k=1$, $J(X_{Q_i},1\le i\le k|X_C)$ becomes
\[
H(X_{Q_1} | X_C) - H(X_{Q_1} | X_C) = 0.
\]
The proposition is proved.
\end{proof}

In view of Propositions~\ref{J=0} and~\ref{J=0b}, we are motivated to introduce the following information-theoretic definition of conditional mutual independence, which incorporates the cases  $k=0$ and $k=1$.

\begin{define}[Conditional Mutual Independence II]
\label{Def-CMI-II}
For $k \ge 0$, $X_{Q_1}, X_{Q_2}, \ldots, X_{Q_k}$ are mutually independent conditioning on $X_C$ if $J(X_{Q_i},1\le i\le k|X_C) = 0$.
\end{define}

\begin{remark}
\label{k01}
When $k = 0, 1$, since $J(X_{Q_i},1\le i\le k|X_C)=0$ always holds according to Proposition~\ref{J=0b}, $X_{Q_1}, X_{Q_2}, \ldots, X_{Q_k}$ are mutually independent conditioning on $X_C$ for every joint distribution for $X_1, X_2, \ldots, X_n$.
\end{remark}

%To facilitate our discussion, we will extend the notation of a CMI to the case $k=0$ and $k=1$ by way of Proposition \ref{mainformula} and refer to such a CMI as a degenerated CMI. Note that a degenerated CMI does not impose any constraint on the distribution of $X_1,X_2,\cdots,X_n$. In this sense, if $k=0$ or $1$, then $\langle Q_i,1\le i\le k \rangle$ is regarded as $\langle\ \rangle$, which is called the empty collection.

%\begin{proposition}\label{nonnegtiveJ}
%	$J(X_{Q_i},1\le i\le k|X_C)\ge0$.
%\end{proposition}
%\begin{proof}
%Consider
%	\begin{equation}\begin{array}{ll}
%J(X_{Q_i},1\le i\le k|X_C)
%&=\sum\limits_{i=1}^{k}H(X_{Q_i}|X_C)-H(X_{Q_i},1\le i\le k|X_C)\\
%&=\sum\limits_{i=1}^{k}H(X_{Q_i}|X_C)-\sum\limits_{i=1}^{k}H(X_{Q_i}|X_C,X_{Q_j},1\le j\le i-1)\\
%&\ge\sum\limits_{i=1}^{k}\left(H(X_{Q_i}|X_C)-H(X_{Q_i}|X_C)\right)\\
%&=0,
%	\end{array}\end{equation}
%which implies $J(X_{Q_i},1\le i\le k|X_C)\ge0$.
%\end{proof}

We will use $\langle Q_i,1\le i\le k\rangle$ to denote a collection of $k$ subsets $Q_i$ of $\mathcal{N}_n$, where $k\ge 0$. In this notation, the $Q_i$'s are allowed to repeat. As such, we have $\langle\{ 2 \}, \{ 2 \}\rangle \ne \langle \{ 2 \} \rangle$, for example. 
If $k=0$, then $\langle Q_i,1\le i\le k\rangle$ is also denoted by $\langle\ \rangle$.
If $k\ge1$, then $\langle Q_1, Q_2, \ldots, Q_k \rangle$ is regarded as equivalent to $\langle Q_{\pi(1)}, Q_{\pi(2)}, \ldots, Q_{\pi(k)}\rangle$, where $\pi: \mathcal{N}_{k}\rightarrow \mathcal{N}_{k}$ is a permutation on $\mathcal{N}_{k}$. Therefore, we write $\langle Q_1, Q_2, \ldots, Q_k \rangle = \langle Q_{\pi(1)}, Q_{\pi(2)}, \ldots, Q_{\pi(k)}\rangle$.
In other words, the order of the $Q_i$'s in $\langle \cdot \rangle$ is immaterial. 

\begin{define}\label{collection=}
Let $\langle Q_i,1\le i\le k \rangle$ and $\langle Q'_j,1\le j\le l \rangle$ be two nonempty collections.
We write $\langle Q_i,1\le i\le k \rangle=\langle Q'_i,1\le i\le l \rangle$ to mean that $l=k$ and there exists a permutation $\pi$ on $\mathcal{N}_k$ such that $p=\pi(q)$. Let $K=(C,\langle Q_i,1\le i\le k\rangle)$ and $K'=(C',\langle Q'_j,1\le j\le l\rangle)$. We say $K=K'$ to mean that $C=C'$ and $\langle Q_i,1\le i\le k \rangle=\langle Q'_j,1\le j\le l \rangle$.
\end{define}

The following proposition is an immediate consequence of Definition \ref{collection=}.

\begin{proposition}\label{collection==}
	Let $\langle Q_i,1\le i\le k \rangle$ and $\langle Q'_j,1\le j\le l \rangle$ be two nonempty collections. Assume that there exist a pair $(p,q)$ such that $Q_{p}=Q'_q$. If $\langle Q_i,1\le i\le k \rangle=\langle Q'_j,1\le j\le l \rangle$, then $\langle Q_i,i\in\mathcal{N}_k\backslash\{p\} \rangle=\langle Q'_j, j\in\mathcal{N}_l\backslash\{q\} \rangle$.
	
\end{proposition}

In the sequel, we will use $K = (C, \langle Q_i, 1 \le i \le k \rangle )$, where $k\ge0$, to denote the CMI 
\[
\mbox{``$X_{Q_1}, X_{Q_2}, \ldots, X_{Q_k}$ are mutually independent conditioning on $X_C$''}
\]
on random variables $X_1, X_2, \ldots, X_n$, where the joint distribution of $X_1, X_2, \ldots, X_n$ is unspecified. Note that a CMI is a logical statement that can be TRUE or FALSE, depending on the joint distribution of $X_1, X_2, \ldots, X_n$. 
If a CMI is TRUE, we also say that it is {\it valid}.
%In this sense, when $k=1$, $K$ can be regarded as valid trivially.
In the rest of the paper, when we refer to a CMI, we always assume that the joint distribution of $X_1, X_2, \ldots, X_n$ is unspecified.

\begin{define}\label{CMI=def}
	Two CMIs $K$ and $K'$ on $X_1, X_2, \ldots, X_n$ are equivalent, denoted by $K\sim K'$, if $K$ and $K'$ are either both TRUE or both FALSE for every joint distribution of $X_1, X_2, \ldots, X_n$.
\end{define}

Evidently, `$\sim$' is an equivalence relation.

\begin{define}\label{degenerateDef}
	A CMI on $X_1,X_2,\cdots,X_n$ is degenerate if it is TRUE for every joint distribution for $X_1, X_2, \ldots, X_n$.
\end{define}

According to Remark~\ref{k01}, $K=(C,\langle Q_i,1\le i\le k\rangle)$ is a degenerate CMI for $k=0,1$.

In $K=(C,\langle Q_i,1\le i\le k\rangle)$, if $Q_j=\emptyset$ for some $1\le j\le k$, then $K$ is regarded as being equivalent to $K'=(C,\langle Q_i,1\le i\le k, i\neq j\rangle)$. This can readily be justified by Proposition \ref{mainformula} by noting that $J(X_{Q_i},1\le i\le k|X_C)=0$ if and only if  $J(X_{Q_i},1\le i\le k,i\neq j|X_C)=0$ (with the convention $X_\emptyset = \mbox{constant}$).
In particular, if $Q_i=\emptyset$ for all $1\le i\le k$, then $K\sim (C,\langle\ \rangle)$, which is degenerate regardless of $C$. Since the set $C$ is immaterial, we will write $(C,\langle\ \rangle)$ as $(\cdot,\langle\ \rangle)$. 
%\end{remark}

%In a CMI $K = (C, \langle Q_i, 1 \le i \le k \rangle )$, if there is only one
%$i$, say $i^*$, such that $Q_i$ is nonempty, then from the discussion below 
%Proposition~\ref{mainformula}, $K$ is always valid. In other words, $K$ does not 
%impose any constraint on the joint distribution of the random variables
%$X_1, X_2, \ldots, X_n$. Likewise, it is also readily seen from Proposition~\ref{mainformula} that $K$ does not 
%impose any constraint on the joint distribution of the random variables
%$X_1, X_2, \ldots, X_n$ if all the $Q_i$'s are empty.

Since a degenerate CMI on $X_1,X_2,\cdots,X_n$ is always valid, it imposes no constraint on the joint distribution of $X_1,X_2,\cdots,X_n$. As all degenerate CMIs are equivalent, with an abuse of notation, they will all be written as $(\cdot,\langle\ \rangle)$. In other words, we write $K=(\cdot,\langle\ \rangle)$ to mean that $K$ is a degenerate CMI.

Let $K=(C, \langle Q_i, 1\le i\le k\rangle)$, where $k \ge 2$. By Proposition~\ref{mainformula}, $K$ is valid if and only if $J(X_{Q_i},1\le i\le k|X_C)=0$. If $Q_1=\cdots=Q_k=Q$ and $J(X_{Q_i},1\le i\le k|X_C)=0$, 
then we can obtain from the latter that
\[kH(X_{Q}|X_C)-H(X_{Q}|X_C)=(k-1)H(X_{Q}|X_C)=0,\]
which implies $H(X_{Q}|X_C)=0$ since $k\ge2$, so that $K$ specifies the {\it functional dependency} ``$X_{Q}$ is a function of $X_{C}$''. From this point of view, functional dependence is a special case of conditional mutual dependence.

In the rest of the paper, unless otherwise specified, we always assume that $k\ge0$.

%By Definition \ref{CMI=def}, we can obtain the follow proposition.
%\begin{proposition}\label{empty-collection}
%	Let $K = (C, \langle Q_i, 1 \le i \le k \rangle )$, then
%	$K\sim (\cdot,\langle\ \rangle)$ if and only if $K=(\cdot,\langle\ \rangle)$.
%\end{proposition}
%\begin{proof}
%It is obvious that $K=(\cdot,\langle\ \rangle)$ implies $K\sim (\cdot,\langle\ \rangle)$.
%Conversely, if $K\sim (\cdot,\langle\ \rangle)$, by Definition \ref{CMI=def}, we know that $K$ and $(\cdot,\langle\ \rangle)$ are either both TRUE or both FALSE for every joint distribution of $X_1, X_2, \ldots, X_n$, which means that $K$ is a degenerated CMI which does not impose any constraint on the distribution of $X_1,X_2,\cdots,X_n$. Then we have $k=0$ or $k=1$. So $K=(\cdot,\langle\ \rangle)$.
%\end{proof}

\begin{define}\label{canonical-form-def}
	Let $K = (C, \langle Q_i, 1 \le i \le k \rangle )$. Then $K$ is said to be in pure form if for all $i$, $Q_i \ne \emptyset$ and $Q_i \cap C = \emptyset$.
\end{define}

\begin{define}\label{def-Pur}
	Let $K = (C, \langle Q_i, 1 \le i \le k \rangle )$. Let $K' = (C, \langle Q_i^\prime, 1 \le i \le k \rangle )$ where $Q_i^\prime = Q_i \backslash C$, with $Q_{i_j}^\prime \ne \emptyset$ for $1 \le j \le l$ and $1\le i_1\le\cdots\le i_l\le k$. Let 
	${\rm pur}(K) = (C, \langle Q_{i_j}^\prime, 1 \le j \le l \rangle )$.
\end{define}

It is evident that for any CMI $K$, ${\rm pur}(K)$ is a CMI in pure form, and is called the pure form of $K$.

%Since $K \sim {\rm cle}(K)$ and ${\rm cle}(K)$ is in canonical form, we will refer to ${\rm cle}(K)$ as the clear form of $K$. 

\begin{proposition}\label{prop-1}
	Let $K=(C, \langle Q_i, 1\le i\le k\rangle)$. Then $J(X_{Q_i},1\le i\le k|X_C)=0$ if and only if $J(X_{Q_i\backslash C},1\le i\le k|X_{C})=0$. 
\end{proposition}
\begin{proof}
%	Let $A$ and $B$ be two index sets. We first consider
%	\begin{equation}\label{eqn-1}
%		\begin{array}{ll}
%			H(X_A,X_B|X_C,X_B) = H(X_A|X_C,X_B) + H(X_B|X_A,X_C,X_B) = H(X_A|X_C,X_B) + 0 = H(X_A|X_C,X_B).
%	\end{array}\end{equation}
	Let $R_i=C\cap Q_i$ for $1\le i\le k$, so that $R_i\subseteq C$ and $R_i\subseteq Q_i$ for all $i$. Then
	\begin{equation*}\begin{array}{ll}
			J(X_{Q_i},1\le i\le k|X_C)\!\!\!
			&=\left[\sum\limits_{i=1}^{k}H(X_{Q_i}|X_C)\right]-H(X_{Q_i},1\le i\le k|X_C)\\
			&=\left[\sum\limits_{i=1}^{k}H(X_{Q_i\backslash R_i},X_{R_i}|X_{C\backslash R_i},X_{R_i})\right]
			-H(X_{Q_i\backslash R_i},X_{R_i},1\le i\le k|X_{C\backslash R_i},X_{R_i},1\le i\le k)\\
			&=\left[\sum\limits_{i=1}^{k}H(X_{Q_i\backslash R_i}|X_{C\backslash R_i},X_{R_i})\right]
			-H(X_{Q_i\backslash R_i},1\le i\le k|X_{C\backslash R_i},X_{R_i},1\le i\le k)\\
			&=\left[\sum\limits_{i=1}^{k}H(X_{Q_i\backslash R_i}|X_{C})\right]-H(X_{Q_i\backslash R_i},1\le i\le k|X_{C})\\
			&=J(X_{Q_i\backslash R_i},1\le i\le k|X_{C}).
	\end{array}\end{equation*}
	Thus $J(X_{Q_i},1\le i\le k|X_C)=0$ if and only if $J(X_{Q_i\backslash R_i},1\le i\le k|X_{C})=0$.
\end{proof}

As an example for Proposition \ref{prop-1}, $(\{1,2,5\},\langle\{1,2\},\{1,2,3\},\{1,4\}\rangle)$ is valid if and only if $(\{1,2,5\},\langle\emptyset,\{3\},\{4\}\rangle)$ is valid, where the latter is equivalent to $(\{1,2,5\},\langle\{3\},\{4\}\rangle)$.

By Proposition \ref{prop-1}, we can obtain the follow corollary.
\begin{cor}\label{prop2}
	Let $K = (C, \langle Q_i, 1 \le i \le k \rangle )$. Then $K \sim {\rm pur}(K)$.
\end{cor}

%\begin{remark}\label{pureformassume}
In view of the above discussion, without loss of generality, we always assume in the rest of this section that a CMI is in pure form. %if not specified, which implies that $Q_i\neq\emptyset$ and $Q_i\cap C=\emptyset$ for $1\le i\le k$, by Definition \ref{canonical-form-def}.
%\end{remark}

\begin{define}
Let $K=(C, \langle Q_i, 1\le i\le k\rangle)$. When $k\ge2$, an index $q$ is called a repeated index if and only if $q\in \cap_{j=1}^{m}Q_{i_j}$, where $2\le m\le k$ and $1\le i_1\le\cdots\le i_m\le k$. 
\end{define}

\begin{define}
	Let $\mathbb{I}_{K}$ be the set of all repeated indices of $K$ if $k\ge2$, and let $\mathbb{I}_{K}=\emptyset$ if $k=0$ or $k=1$.
\end{define}

\begin{lemma}\label{lemma-1}
	 If $K=(C, \langle Q_i, 1\le i\le k\rangle)$ with $k\ge2$ is valid, then $H(X_{q}|X_C)=0$ for any $q\in\mathbb{I}_K$.
\end{lemma}
\begin{proof}
%Without loss of generality, we can assume that $H(X_q|X_C)=a>0$ and $X_{p}={\rm CONSTANT}$ for $p\in\cup_{i=1}^{k}(Q_i\backslash\{q\})$. 
%Since $q\in \mathbb{I}_K$, then we let $q\in\cap_{j=1}^{m}Q_{i_j}$, for $2\le m\le k$ and $1\le i_j\le k$.
%Then we have
%$$
%H(X_{Q_i},1\le i\le k|X_C)=H(X_q,X_{Q_i\backslash\{q\}},1\le i\le k|X_C)=H(X_q)=a.
%$$ 
%and 
%$$
%\sum\limits_{i=1}^{k}H(X_{Q_i}|X_C)=\sum\limits_{j=1}^{m}H(X_{q},X_{Q_{i_j}\backslash\{q\}}|X_C)
%=\sum\limits_{j=1}^{m}H(X_{q}|X_C)=ma.
%$$
%Since $ma\ne a$, the collection $K$ is not valid, which constricts the condition.
%
%Thus the Lemma is proved.
%\end{proof}
%\textbf{Another proof of \ref{lemma-1}.} 

Consider any repetitive index $q\in \cap_{j=1}^{m}Q_{i_j}$, where $m\ge 2$. 
Then $q\in Q_{i_1}\cap Q_{i_2}$.
If $K$ is valid, then $X_{Q_{i_1}}$ and $X_{Q_{i_2}}$ are independent conditioning on $X_C$, which implies $I(X_{Q_{i_1}};X_{Q_{i_2}}|X_C)=0$. Now consider
\begin{equation}\begin{array}{ll}
	I(X_{Q_{i_1}};X_{Q_{i_2}}|X_C)&=I(X_q;X_{Q_{i_2}}|X_C)+I(X_{Q_{i_1}\backslash\{q\}};X_{Q_{i_2}}|X_C,X_q)\\
	&=I(X_q;X_q|X_C)+I(X_q;X_{Q_{i_2}\backslash\{q\}}|X_C,X_q)+I(X_{Q_{i_1}\backslash\{q\}};X_{Q_{i_2}}|X_C,X_q)\\
	&=H(X_q|X_C)+I(X_{Q_{i_1}\backslash\{q\}};X_{Q_{i_2}}|X_C,X_q).
\end{array}\end{equation}
Since $H(X_q|X_C)\ge0$ and $I(X_{Q_{i_1}\backslash\{q\}};X_{Q_{i_2}}|X_C,X_q)\ge0$, $I(X_{Q_{i_1}};X_{Q_{i_2}}|X_C)=0$ implies $H(X_q|X_C)=0$.
The lemma is proved.
\end{proof}

\begin{define}\label{def-can}
	Let $K=(C, \langle Q_i, 1\le i\le k\rangle)$, $\mathbb{I}_{K}$ be the set of all repeated indices of $K$,
	and $P_j=Q_{i_j}\backslash \mathbb{I}_K$, with $1\le j\le t$ and $1\le i_1\le\cdots\le i_t\le k$, be the nonempty sets of  $Q_{i}\backslash \mathbb{I}_K$'s.
	Let the canonical form of $K$ be 
	%${\rm can}(K)= \langle\ \rangle$ if $k=1$.
	%	When $k\ge 2$, let the canonical form of $K$ be
	%	 ${\rm can}(K)= (C, \langle \mathbb{I}_K,\mathbb{I}_K,Q_i\backslash \mathbb{I}_K, 1\le i\le k\rangle)$ if $\mathbb{I}_K\neq\emptyset$, 
	%	 and
	%	${\rm can}(K)= (C, \langle Q_i, 1\le i\le k\rangle)$ if $\mathbb{I}_K=\emptyset$.
	%
	\begin{equation}\label{eq-def-can}
		{\rm can}(K)= \left\{
		\begin{array}{ll}
			(\cdot,\langle\ \rangle)\  &{\rm if} \ \ k=0,1  \\
			(C, \langle \mathbb{I}_K,\mathbb{I}_K\rangle) \ &{\rm if} \ \ k\ge2,\ \mathbb{I}_K\neq\emptyset,\  {\rm and}\ t=0,1\\
				(C, \langle P_j, 1\le j\le t\rangle) \ &{\rm if} \ \ k\ge2,\ \mathbb{I}_K=\emptyset \\
			(C, \langle \mathbb{I}_K,\mathbb{I}_K,P_j, 1\le j\le t\rangle) \ &{\rm if} \ \ k\ge2,\  \mathbb{I}_K\neq\emptyset,\ {\rm and}\ t\ge2.  
	\end{array}
		\right.
	\end{equation}
	
\end{define}

In Definition \ref{def-can}, $K$ is assumed to be in pure form as mentioned.
%in Definition \ref{def-can}, by Definition \ref{canonical-form-def}, which implies that $Q_i\neq\emptyset$ and $Q_i\cap C=\emptyset$ for $1\le i\le k$.
%
%In $K$, 
If $k\ge2$ and $\mathbb{I}_K=\emptyset$, then $P_i=Q_i$ for $1\le i\le k$, $t = k$, and $P_i,\ 1\le i\le k$ and $C$ are disjoint; if $k\ge2$, $\mathbb{I}_K\neq\emptyset$, and $t\ge2$, then $C$, $\mathbb{I}_K$, and $P_j=Q_{i_j}\backslash\mathbb{I}_K$, $1\le j\le t$, are disjoint. 
If $K=(C,\langle \mathbb{I}_K,\mathbb{I}_K,B\rangle)$, where $B\neq \emptyset$, then
${\rm can}(K)=(C,\langle \mathbb{I}_K,\mathbb{I}_K\rangle)$.

\begin{remark}\label{remark-1}
	Let $K = (C, \langle Q_i, 1 \le i \le k \rangle )$ and ${\rm can}(K)=(C, \langle \mathbb{I}_K, \mathbb{I}_K, P_i, 1\le i\le t\rangle)$.
	By Definition~\ref{def-can}, $P_i, 1\le i \le t$, $\mathbb{I}_K$, and $C$ are all disjoint.
\end{remark}

The following basic property of the canonical form can be readily verified from Definition \ref{def-can}.

\begin{proposition}\label{PropIII.5}
	${\rm can}({\rm can}(K))=K$.
\end{proposition}

For the convenience of discussion, by Definition \ref{def-can}, in the rest of the paper, we will regard $(C, \langle \mathbb{I}_K,\mathbb{I}_K,P_j, 1\le j\le t\rangle)$ with $t\neq 1$ as the general form of ${\rm can}(K)$ and adopt the following denotation:
	\begin{equation}\label{varitionofcan}
	(C, \langle \mathbb{I}_K,\mathbb{I}_K,P_j, 1\le j\le t\rangle)\triangleq \left\{
	\begin{array}{ll}
		(\cdot,\langle\ \rangle)\  &{\rm if} \ \ \mathbb{I}_K=\emptyset \ \ {\rm and}\ \ t=0 \\
		(C, \langle P_j, 1\le j\le t\rangle) \ &{\rm if} \ \ \mathbb{I}_K=\emptyset\ \ {\rm and}\ \ t\ge2\\
		(C, \langle \mathbb{I}_K,\mathbb{I}_K\rangle) \ &{\rm if} \ \ \mathbb{I}_K\neq\emptyset,\  {\rm and}\ \ t=0.\\
	%	(C, \langle \mathbb{I}_K,\mathbb{I}_K,P_j, 1\le j\le t\rangle) \ &{\rm if} \ \ \mathbb{I}_K\neq\emptyset,\ {\rm and}\ t\ge2.  
	\end{array}
	\right.
\end{equation}

%\noindent Note that in the rest of the paper, we will use $(C, \langle \mathbb{I}_K,\mathbb{I}_K,P_j, 1\le j\le t\rangle)$ to represent the canonical form of $K$. The definition of ${\rm can}(K)$ in specific cases can be found in \eqref{varitionofcan}.

% {\color{red}
% \begin{remark}\label{remark-2}
% 	Let ${\rm can}(K)=(C, \langle \mathbb{I}_K,\mathbb{I}_K,P_j, 1\le j\le t\rangle)$. Note that from  \eqref{eq-def-can}, if $\mathbb{I}_K\neq\emptyset$ and $t\ge 1$, then $t\ge 2$. In other words,
% if $\mathbb{I}_K\neq\emptyset$, then either $t=0$ (in this case ${\rm can}(K)=(C, \langle \mathbb{I}_K,\mathbb{I}_K\rangle)$, i.e., the second case in \eqref{eq-def-can}) or $t\ge 2$.
% \end{remark}}
% {\color{blue}When I read "Note that..." again, I am confused till I read the later sentence "In other words,...". So should we rewrite it by a clearer way?}

\begin{lemma}\label{cor-1}
	Let $K=(C,\langle Q_i, 1\le i\le k\rangle)$.
	If $K$ or ${\rm can}(K)$ is valid, then $H(X_{\mathbb{I}_K}|X_C)=0$.
\end{lemma}
\begin{proof}
If $k=0$ or $k=1$, then $\mathbb{I}_K=\emptyset$, and $X_{\mathbb{I}_K}$ can be regard as a constant. So, $H(X_{\mathbb{I}_K}|X_C)=0$ holds obviously.
If $k\ge2$, and if $K$ is valid, since $H(X_q|X_C)=0$ for any $q\in\mathbb{I}_K$, we can obtain $H(X_{\mathbb{I}_K}|X_C)=0$ by considering
$\sum_{q\in\mathbb{I}_K}H(X_q|X_C)\ge H(X_{\mathbb{I}_K}|X_C)$.
Since $\mathbb{I}_K$ is also the repeated set of ${\rm can}(K)$, if ${\rm can}(K)$ is valid, then $H(X_{\mathbb{I}_K}|X_C)=0$.
\end{proof}

We give an algorithm below to find the canonical form of a CMI $K$.
%
% {\color{blue}Notice that we set $K$ as "general" in this algorithm before, but we always assume "$K$ is in pure form" in this section. So, should we give an algorithm for "general $K$"? In this version, I am writing the algorithm with K as the pure form. Please note that there are some changes in the algorithm.}
\\[0.2cm]
\begin{breakablealgorithm}
	\caption{Canonical form}
	\begin{algorithmic}
		\REQUIRE A collection $K = (C, \langle Q_i, 1 \le i \le k \rangle ), k\ge 1$.
		\ENSURE The canonical form $K'$.
		\IF{$k=1$}
		\STATE Let $K'=(\cdot,\langle\ \rangle)$.
		\ELSE
		\STATE Let $N=\mathcal{N}_k$,
		%\STATE {\color{red}\sout{Let $Q_i=Q_i\backslash C$, where $1\le i\le k$.}} 
        %{\color{blue} Because $K$ is assumed pure form, $Q_i\cap C=\emptyset$.}
	%	\STATE Let $K = (C, \langle Q_i, 1 \le i \le k \rangle)$.
		 and let $\mathbb{I}_K$ be the set of all repeated indices of $K$.
		\IF{$\mathbb{I}_K\neq \emptyset$}
		\STATE Let $Q_i=Q_i\backslash \mathbb{I}_K$.
		\WHILE{$Q_j=\emptyset$ for some $j\in N$}
		\STATE Let $N=N\backslash\{j\}$.
		\IF{$|N|= 1$}
		\STATE Let $K'=(C,\langle\mathbb{I}_K,\mathbb{I}_K \rangle)$, and break the WHILE loop.
		\ELSE
		\STATE Let $K'=(C,\langle\mathbb{I}_K,\mathbb{I}_K,Q_m, m\in N\rangle)$.
		\ENDIF
		\ENDWHILE
		\ELSE  
	\STATE Let $K'=(C,\langle Q_i,1\le i\le k\rangle).$
	\COMMENT{According to the assumption that $K$ is in pure form, $Q_i\neq \emptyset$ for~ $1\le i\le k$.}
%		\IF{$Q_i\neq \emptyset$ for $1\le i\le k$}
%		\STATE Let $K'=(C,\langle \emptyset,\emptyset,Q_i,1\le i\le k\rangle).$
%		\ELSE
%		\WHILE{$Q_j=\emptyset$ for some $1\le j\le k$}
%		\STATE Let $N=N\backslash\{j\}$.
%		\IF{$|N|= 1$}
%		\STATE Let $K'=\langle\ \rangle$, and break the WHILE loop.
%		\ELSE
%		\STATE Let $K'=(C,\langle \emptyset,\emptyset,Q_m, m\in N\rangle)$.
%		\ENDIF
%		\ENDWHILE
%		\ENDIF
		\ENDIF
		\ENDIF
		\RETURN $K'$.
	\end{algorithmic}
\end{breakablealgorithm}
\ 
\\[0.2cm]
\ 
\begin{lemma}\label{LemIII.2}
	If $H(X_{\mathbb{I}_K|X_C})=0$, then for all $k\ge0$,
	$$H(X_{Q_i},1\le i\le k|X_C)=H(X_{(\cup_{i=1}^{k}Q_i)\backslash\mathbb{I}_K}|X_C)$$ 
	and $$H(X_{Q_i}|X_C)=H(X_{Q_i\backslash\mathbb{I}_K}|X_C).$$
\end{lemma}
	
\begin{proof}
If $H(X_{\mathbb{I}_K}|X_C)=0$, 
then
\begin{equation}
	\label{add3}		
	\begin{array}{ll}
H(X_{Q_i}|X_C)
=H(X_{Q_i\backslash\mathbb{I}_K},X_{Q_i\cap\mathbb{I}_K}|X_C)
=H(X_{Q_i\backslash\mathbb{I}_K}|X_C),
\end{array}\end{equation}	
and
\begin{equation}
	\label{add6}		
	\begin{array}{ll}
		H(X_{Q_i},1\le i\le k|X_C)
=H(X_{(\cup_{i=1}^{k}Q_i)\backslash\mathbb{I}_K},X_{\mathbb{I}_K}|X_C)
=H(X_{(\cup_{i=1}^{k}Q_i)\backslash\mathbb{I}_K}|X_C).
		\end{array}\end{equation}	

\end{proof}

We can readily obtain the following corollary via Lemma \ref{cor-1} and the above lemma.
\begin{cor}\label{corofLemmaIII2}
	If $K=(C, \langle Q_i, 1\le i\le k\rangle)$ is valid, then  $$H(X_{Q_i},1\le i\le k|X_C)=H(X_{(\cup_{i=1}^{k}Q_i)\backslash\mathbb{I}_K}|X_C)$$ 
	and $$H(X_{Q_i}|X_C)=H(X_{Q_i\backslash\mathbb{I}_K}|X_C).$$
\end{cor}

\begin{theorem}\label{k-cank}
	%Let $K=(C, \langle Q_i, 1\le i\le k\rangle)$. Then
	 $K\sim {\rm can}(K)$.
\end{theorem}

\begin{proof}
Let $K=(C, \langle Q_i, 1\le i\le k\rangle)$.
%If $k=1$, then $K=\langle\ \rangle$. By Proposition \ref{empty-collection} and \ref{empty-can}, we obtain 
%$K\sim\langle\ \rangle$ and ${\rm can}(K)=\langle\ \rangle$. Thus $K\sim {\rm can}(K)$.
%
%Next, we assume that $k\ge2$.
If $K$ or ${\rm can}(K)$ is valid, we have
	\begin{equation}\label{add8}
	H(\mathbb{I}_K|X_C)=0 
	\end{equation}
	by Lemma \ref{cor-1}.
	Then we have
	\begin{equation}\begin{array}{ll}
	J(X_{Q_i},1\le i\le k|X_C)
	&=\left[\sum\limits_{i=1}^{k}H(X_{Q_i}|X_C)\right]-H(X_{Q_i},1\le i\le k|X_C)\\
	&\overset{\eqref{add3}}{=}\left[\sum\limits_{i=1}^{k}H(X_{Q_i\backslash\mathbb{I}_K}|X_C)\right]-H(X_{\mathbb{I}_K},X_{\mathbb{I}_K},X_{(\cup_{i=1}^k Q_i)\backslash\mathbb{I}_K}|X_C)\\
	&\overset{\eqref{add8}}{=}2H(X_{\mathbb{I}_K}|X_C)+\left[\sum\limits_{i=1}^{k}H(X_{Q_i\backslash\mathbb{I}_K}|X_C)\right]-H(X_{\mathbb{I}_K},X_{\mathbb{I}_K},X_{Q_i\backslash\mathbb{I}_K},1\le i\le k|X_C)\\
	&=J(X_{\mathbb{I}_K},X_{\mathbb{I}_K},X_{Q_i\backslash\mathbb{I}_K},1\le i\le k|X_C),\\
	&=J(X_{\mathbb{I}_K},X_{\mathbb{I}_K},X_{P_i},1\le i\le t|X_C),
	\end{array}\end{equation}
%where without loss of generality assume that $Q_{i}\backslash \mathbb{I}_K\neq\emptyset, 1\le i\le t$ and $Q_{i}\backslash \mathbb{I}_K=\emptyset,\ t+1\le i\le k$, 
implying that $K$ is valid if and only if ${\rm can}(K)$ is valid.
Thus, $K\sim {\rm can}(K)$.
\end{proof}

\begin{proposition}\label{empty-can}
	Let $K$ be a CMI. Then $K=(\cdot,\langle\ \rangle)$ if and only if ${\rm can}(K)=(\cdot,\langle\ \rangle)$.
\end{proposition}

\begin{proof}
By Theorem \ref{k-cank}, we have $K\sim {\rm can}(K)$, which implies that $K$ is a degenerate CMI if and only if ${\rm can}(K)$ is a degenerate CMI.
If $K=(\cdot,\langle\ \rangle)$, then $K$ is a degenerate CMI, so ${\rm can}(K)$ is a degenerate CMI, i.e., ${\rm can}(K)=(\cdot,\langle\ \rangle)$.
Conversely, if ${\rm can}(K)=(\cdot,\langle\ \rangle)$, then $ {\rm can}(K)$ is a degenerate CMI,  so $K$ is a degenerate CMI, i.e., $K=(\cdot,\langle\ \rangle)$.
\end{proof}

\begin{proposition}
Let $K=(C, \langle Q_i, 1\le i\le k\rangle)$, $\mathbb{I}_{K}$ be the set of all repeated indices of $K$, and  $P_j=Q_{i_j}\backslash \mathbb{I}_K$, with $1\le j\le t,\ 1\le i_1\le\cdots\le i_t\le k$, be the nonempty sets of  $Q_{i}\backslash \mathbb{I}_K$'s. Then $K$ is valid if and only if $(C,\langle \mathbb{I}_K,\mathbb{I}_K \rangle)$ and $(C, \langle P_i, 1\le i\le t\rangle)$ are both valid.
\end{proposition}
\begin{proof}

Assume that $(C,\langle \mathbb{I}_K,\mathbb{I}_K \rangle)$ and $(C, \langle P_i, 1\le i\le t\rangle)$ are both valid.  From the former we have $H(X_{\mathbb{I}_K}|X_C)=I(X_{\mathbb{I}_K};X_{\mathbb{I}_K}|X_C)=0$, and from the latter we have $J(X_{P_i},1\le i\le t|X_C)=0$ by 
Definition \ref{Def-CMI-II}.
%Proposition \ref{mainformula}. 
Then
\begin{align}
		J(X_{Q_i},1\le i\le k|X_C)
		&=\left[\sum\limits_{i=1}^{k}H(X_{Q_i}|X_C)\right]-H(X_{Q_i},1\le i\le k|X_C) \nonumber\\
		&\overset{}{=}\left[\sum\limits_{i=1}^{k}H(X_{Q_i\backslash\mathbb{I}_K}|X_C)\right]-H(X_{(\cup_{i=1}^{k}Q_i)\backslash\mathbb{I}_K}|X_C) \nonumber\\
		&=\left[\sum\limits_{i=1}^{k}H(X_{Q_i\backslash\mathbb{I}_K}|X_C)\right]-H(X_{Q_i\backslash\mathbb{I}_K},1\le i\le k|X_C) \nonumber\\
		&=\left[\sum\limits_{i=1}^{t}H(X_{P_i}|X_C)\right]-H(X_{P_i},1\le i\le t|X_C) \nonumber\\
		&=~J(X_{P_i},1\le i\le t|X_C) \label{J7}\\
		&=~0, \nonumber
\end{align}
where the second equality above follows from $H(X_{\mathbb{I}_K}|X_C)=0$ and Lemma \ref{LemIII.2}.
Thus $K$ is valid by
Definition \ref{Def-CMI-II}.

%{\color{blue}Because "$k\ge2$" is set in Proposition III.1. But we need "$k\ge0$" here.}

To prove the converse, assume that $K$ is valid, i.e., $J(X_{Q_i},1\le i\le k|X_C)=0$. By Lemma \ref{cor-1}, we have $H(X_{\mathbb{I}_K}|X_C)=0$, or equivalently, $(C,\langle \mathbb{I}_K,\mathbb{I}_K \rangle)$ is valid. 
Since the equality \eqref{J7} holds when $H(X_{\mathbb{I}_K}|X_C)=0$, we have $J(X_{P_i},1\le i\le t|X_C)=J(X_{Q_i},1\le i\le k|X_C)=0$. Thus,
 $(C, \langle P_i, 1\le i\le t\rangle)$ is also valid.
 
The proof is accomplished.
\end{proof}

\begin{lemma}\label{set-lem}
	Let $A_i, 1\le i\le k$ 
	be $k$ nonempty disjoint sets of indices
	and $B_j, 1\le j\le l$ be $l$ nonempty disjoint sets of indices, where $k\ge 2$ or $l\ge 2$. Let $A=\cup_{i=1}^{k}A_i$, $B=\cup_{j=1}^{l}B_j$. If $A=B$ and $\langle A_i,1\le i\le k\rangle\neq\langle B_j,1\le j\le l\rangle$, then there exists a pair $(i,j)$ such that $A_i\cap B_{j}\neq \emptyset$, and either $A_i\backslash B_j\neq\emptyset$ or $B_j\backslash A_i\neq\emptyset$. 
\end{lemma}
\begin{proof}
Assume that $A=B$ and $\langle A_i,1\le i\le k\rangle\neq\langle B_j,1\le j\le l\rangle$. We observe that under these assumptions, for any $i\in\mathcal{N}_k$, there exists $\alpha(i)\in\mathcal{N}_l$ such that $A_i\cap B_{\alpha(i)}\neq\emptyset$.
We will prove the lemma by contradiction. Assume the contrary, i.e., for any pair $(i,j)$, we have $A_i\backslash B_j=\emptyset$ and $B_j\backslash A_i=\emptyset$. Then $A_i\backslash B_{\alpha(i)}=\emptyset$ and $B_{\alpha(i)}\backslash A_i=\emptyset$, so we obtain $A_i=B_{\alpha(i)}$ for $1\le i\le k$.
Since $A_i$'s are disjoint, we have $k\le l$ and $B_{\alpha(i)}$'s are disjoint, which implies that for all $1\le i,i'\le k$, if $i\neq i'$, then
 $A_i\neq A_{i'}$ and $B_{\alpha(i)}\neq B_{\alpha(i')}$.
  If $k=l$, then $\langle A_i,1\le i\le k\rangle=\langle B_j,1\le j\le l\rangle$, a contradiction to the assumption that $\langle A_i,1\le i\le k\rangle\neq\langle B_j,1\le j\le l\rangle$. If $k<l$, then
  \[A=\bigcup\limits_{i=1}^{k}B_{\alpha(i)}\ \subsetneqq\  \bigcup\limits_{j=1}^{l}B_j=B\]
   because $B_j$'s are nonempty, 
a contradiction to the assumption that $A=B$. The proof is accomplished.
%
%Since $A=B$, for an arbitrary $i\in\mathcal{N}_k$, there exists one $j$ such that $A_i\cap B_j\neq\emptyset$.
%For this pair $(i,j)$, if both $A_i\backslash B_j\neq\emptyset$ and $B_j\backslash A_i\neq\emptyset$ do not hold, then $A_i=B_j$. So for any $i$, we have an $j$ such that $A_i=B_j$. Since $B_j$'s are disjoint, we have $k\le l$.
%Without loss of generality, we can assume $A_i=B_i$, $1\le i\le k$.
%Then we claim that $k=l$. Otherwise, without loss of generality, we let $k<l$.
%Then we select $m\in B_{k+1}$. Since $B_j$'s are disjoint, $m\notin\bigcup\limits_{j=1}^{k}B_j$, implying $m\notin A$, which contradicts with $A=B$. Now we have $k=l$, which implies $\langle A_i,1\le i\le k\rangle=\langle B_j,1\le j\le l\rangle$.
%Thus, either $A_i\backslash B_j\neq\emptyset$ or $B_j\backslash A_i\neq\emptyset$ holds.
%The proof of this lemma is accomplished.
\end{proof}

\begin{lemma}\label{k=or=l=1}
%Let $K = (C, \langle Q_i, 1 \le i \le k \rangle )$ and $K' = (C', \langle Q'_j, 1 \le j \le l \rangle )$ be two CMIs in pure form. 
If $K=(\cdot,\langle\ \rangle)$ or $K'=(\cdot,\langle\ \rangle)$, then $K \sim K'$ if and only if ${\rm can}(K) = {\rm can}(K')$.
\end{lemma}

\begin{proof}
%If $k=0,1$ or $l=0,1$, 
Without loss of generality, we assume $K'=(\cdot,\langle\ \rangle)$, i.e., $K'$ is a degenerate CMI.

We first prove the ``only if'' part. Assume ${\rm can}(K)={\rm can}(K')$.
Since $K'=(\cdot,\langle\ \rangle)$,
we have ${\rm can}(K')=(\cdot,\langle\ \rangle)$ by Proposition \ref{empty-can}.
If $K\sim K'$, then $K$ is a degenerate CMI, or $K=(\cdot,\langle\ \rangle)$. Thus by Proposition \ref{empty-can}, we have ${\rm can}(K)=(\cdot,\langle\ \rangle) = {\rm can}(K')$.

To prove the ``if'' part, assume ${\rm can}(K) = {\rm can}(K')$. 
Since we have asumed that $K'=(\cdot,\langle\ \rangle)$, we have
 ${\rm can}(K')=(\cdot,\langle\ \rangle)$. 
 Then ${\rm can}(K)={\rm can}(K')=(\cdot,\langle\ \rangle)$. 
 Again by Proposition \ref{empty-can}, we obtain $K=(\cdot,\langle\ \rangle)=K'$. Hence, $K\sim K'$.
\end{proof}

\begin{lemma}\label{Lemiii5}
	Let $K = (C, \langle Q_i, 1 \le i \le k \rangle )$.  If $K\neq (\cdot,\langle\ \rangle)$, then either one of the following holds:
%	\begin{list}%
%		{(C\arabic{cond})}{\usecounter{cond}}
%		\end{list}
\begin{itemize}
	\item[i)] $\mathbb{I}_K=\emptyset$ and at least two $Q_i\backslash \mathbb{I}_K$ are nonempty;
	\item[ii)] $\mathbb{I}_K\neq\emptyset$.
\end{itemize}

\end{lemma}
\begin{proof}
	Assume neither i) nor ii) holds. Then we have $\mathbb{I}_K=\emptyset$ and at most one $Q_i\backslash \mathbb{I}_K$ is nonempty, which implies ${\rm can}(K)=(\cdot,\langle\ \rangle)$. By Proposition \ref{empty-can}, we obtain $K=(\cdot,\langle\ \rangle)$, which contradicts our assumption that $K\neq(\cdot,\langle\ \rangle)$.
\end{proof}

\begin{lemma}\label{can=}
Let $K = (C, \langle Q_i, 1 \le i \le k \rangle )$ and $K' = (C', \langle Q'_j, 1 \le j \le l \rangle )$, and let ${\rm can}(K)=(C, \langle \mathbb{I}_K, \mathbb{I}_K, P_i, 1\le i\le t\rangle)$ and 
${\rm can}(K')=(C',\langle \mathbb{I}_{K'}, \mathbb{I}_{K'}, P'_j, 1\le j\le s\rangle)$.
Then ${\rm can}(K) = {\rm can}(K')$ if and only if $C = C'$, $\mathbb{I}_{K'}=\mathbb{I}_{K}$ and 
$\langle P_i, 1\le i\le t\rangle=\langle P'_j, 1\le j\le s\rangle$.
\end{lemma}
\begin{proof}
It is obvious that if $C = C'$, $\mathbb{I}_{K'}=\mathbb{I}_{K}$, and $\langle P_i, 1\le i\le t\rangle=\langle P'_j, 1\le j\le s\rangle$,
then ${\rm can}(K) = {\rm can}(K')$.
Conversely, if ${\rm can}(K) = {\rm can}(K')$, then by Definition \ref{collection=}, we have $C = C'$ and $\langle \mathbb{I}_{K}, \mathbb{I}_{K}, P_i, 1\le i\le t\rangle = \langle \mathbb{I}_{K'}, \mathbb{I}_{K'}, P'_j, 1\le j\le s\rangle$.
Next, we will prove that $\mathbb{I}_{K'}=\mathbb{I}_{K}$ and $\langle P_i, 1\le i\le t\rangle = \langle  P'_j, 1\le j\le s\rangle$ by contradiction. 
Assume $\mathbb{I}_{K}\neq \mathbb{I}_{K'}$. Then at least one of $\mathbb{I}_{K}$ or $\mathbb{I}_{K'}$ is nonempty. Without loss of generality, assume $\mathbb{I}_{K}\neq \emptyset$. 
%Let $m\in\mathbb{I}_K$. 
Since 
$\langle \mathbb{I}_{K}, \mathbb{I}_{K}, P_i, 1\le i\le t\rangle = \langle \mathbb{I}_{K'}, \mathbb{I}_{K'}, P'_j, 1\le j\le s\rangle$ and $\mathbb{I}_{K}\neq \mathbb{I}_{K'}$, by Definition \ref{collection=}, there exist
$1\le j_1, j_2\le s$, $j_1\neq j_2$, such that 
 $ P'_{j_1}= P'_{j_2}=\mathbb{I}_K$ (note that $\mathbb{I}_K$ appears twice in $\langle \mathbb{I}_{K}, \mathbb{I}_{K}, P_i, 1\le i\le t\rangle$).
This contradicts that $P'_{j}$\ 's are disjoint, and therefore we conclude that $\mathbb{I}_{K}= \mathbb{I}_{K'}$. 
Then, by Proposition \ref{collection==} and $\langle \mathbb{I}_{K}, \mathbb{I}_{K}, P_i, 1\le i\le t\rangle = \langle \mathbb{I}_{K'}, \mathbb{I}_{K'}, P'_j, 1\le j\le s\rangle$, we obtain $\langle P_i, 1\le i\le t\rangle = \langle P'_j, 1\le j\le s\rangle$.
\end{proof}

\begin{lemma}\label{LemmaIII-8}
	Let $K = (C, \langle Q_i, 1 \le i \le k \rangle )$ and $K' = (C', \langle Q'_j, 1 \le j \le l \rangle )$, and let ${\rm can}(K)=(C, \langle \mathbb{I}_K, \mathbb{I}_K, P_i, 1\le i\le t\rangle)$ and 
	${\rm can}(K')=(C',\langle \mathbb{I}_{K'}, \mathbb{I}_{K'}, P'_j, 1\le j\le s\rangle)$.
	Let $P=\cup_{i=1}^{t}P_i$ and 
	$P'=\cup_{j=1}^{s}P'_j$.
	If $K \sim K'$, then $\mathbb{I}_K\cap P'=\emptyset$ and $\mathbb{I}_{K'}\cap P=\emptyset$.
\end{lemma}
\begin{proof}
We will prove the lemma by contradiction. If $\mathbb{I}_K\cap P'\neq\emptyset$, without loss of generality, assume $\mathbb{I}_K\cap P'_1\neq\emptyset$. Let $m_0\in\mathbb{I}_K\cap P'_1$. By Remark~\ref{remark-1},  we have 
$\mathbb{I}_K\cap C=\emptyset$, $\mathbb{I}_K\cap P=\emptyset$, 
$P'_1\cap \left(\cup_{j=2}^{s}P'_j\right)=\emptyset$,
$P'_1\cap \mathbb{I}_{K'}=\emptyset$,
and $P'_1\cap C'=\emptyset$.
Thus we obtain
\begin{equation}\label{cond-1a}
	m_0\in P'_1\subseteq P',\ 
	m_0\notin P'_j, 2\le j\le s,\ 
	m_0\notin \mathbb{I}_{K'},\ 
	m_0\notin C',\ 
	m_0\in \mathbb{I}_K,\  
	m_0\notin P,\ 
	m_0\notin C.
\end{equation}

We now construct a joint distribution for $X_1, \ldots, X_n$ as follows. Let $U$ be a random variable such that $H(U) > 0$.
Let 
\[
X_m = \left\{ \begin{array}{ll}
	U & \mbox{if $m = m_0$} \\
	0 & \mbox{otherwise}.
\end{array} \right.
\]
By \eqref{cond-1a}, we obtain
\begin{align*}
		H(X_{\mathbb{I}_K},X_{\mathbb{I}_K},X_{P_i},1\le i\le t|X_C)&=H(X_{\mathbb{I}_K})=H(X_{m_0})=H(U)\\
		2H(X_{\mathbb{I}_K}|X_C)+\sum\limits_{i=1}^{t}H(X_{P_i}|X_C)&=2H(X_{\mathbb{I}_K})=2H(X_{m_0})=2H(U)\\
		H(X_{\mathbb{I}_{K'}},X_{\mathbb{I}_{K'}},X_{P'_j},1\le j\le s|X_{C'})&=H(X_{P'_1})=H(X_{m_0})=H(U)\\
		2H(X_{\mathbb{I}_{K'}}|X_{C'})+\sum\limits_{j=1}^{s}H(X_{P'_j}|X_{C'})&=H(X_{P'_1})=H(X_{m_0})=H(U).
\end{align*}
Thus we have
\begin{align*}
		H(X_{\mathbb{I}_{K}},X_{\mathbb{I}_{K}},X_{P_i},1\le i\le t|X_{C})
		&\neq 2H(X_{\mathbb{I}_{K}}|X_{C})+\sum\limits_{i=1}^{t}H(X_{P_i}|X_{C})\\
		H(X_{\mathbb{I}_{K'}},X_{\mathbb{I}_{K'}},X_{P'_j},1\le j\le s|X_{C'})
		&= 2H(X_{\mathbb{I}_{K'}}|X_{C'})+\sum\limits_{j=1}^{s}H(X_{P'_j}|X_{C'}),
\end{align*}
implying that $K$ is invalid and $K'$ is valid. 
Then $K\nsim K'$, which contradicts the assumption that $K\sim K'$.
Thus $\mathbb{I}_K\cap P'=\emptyset$. 

In the same way, we can prove that if $K\sim K'$, then $\mathbb{I}_{K'}\cap P=\emptyset$.
The proof is accomplished.
\end{proof}

The technique for constructing a joint distribution for $X_1,\ldots,X_n$ in the proofs of Lemma \ref{LemmaIII-8} and other results in this paper has its root in the theory of $I$-measure developed in \cite{Yeung1991}. The reader is referred  to the discussion in \cite[Chapter 3]{Yeung2008}, in particular Theorem 3.11 and Problem 10 therein.

\begin{theorem}\label{prop1}
	Let $K = (C, \langle Q_i, 1 \le i \le k \rangle )$ and $K' = (C', \langle Q'_j, 1 \le j \le l \rangle )$ be two non-degenerate CMIs, and let ${\rm can}(K)=(C, \langle \mathbb{I}_K, \mathbb{I}_K, P_i, 1\le i\le t\rangle)$ and 
	${\rm can}(K')=(C',\langle \mathbb{I}_{K'}, \mathbb{I}_{K'}, P'_j, 1\le j\le s\rangle)$.
	If $K\sim K'$, then $\langle P_i, 1\le i\le t\rangle = \langle P'_j, 1\le j\le s\rangle$.
\end{theorem}

% {\color{red}
% \sout{
% %\begin{remark}
% 	In Theorem~\ref{prop1}, $K$ and $K'$ are assumed to be in pure form. In general, $K$ and $K'$ are not necessarily in pure form. Then according to Theorem~\ref{prop1}, $K \sim K'$ if and only if ${\rm can}({\rm pur}(K)) = {\rm can}({\rm pur}(K'))$.
% %\end{remark}
% }
% }

% {\color{blue} Remark III.4 (the red text) was given here before, I think it should be in Theorem III.3?}

\begin{proof}
%\bigskip \noindent
%{\it Proof of Theorem~\ref{prop1}.} 
	We will prove the theorem by contradiction. 
	Let $P=\cup_{i=1}^{t}P_i$ and 
	$P'=\cup_{j=1}^{s}P'_j$.
	Assume that $K\sim K'$ and consider the following cases.

	\noindent {\bf Case 1.} $P\neq P'$
	
	Without loss of generality, assume that $P'\backslash P\neq \emptyset$. This implies that $P'_j\backslash P\neq \emptyset$ for at least one $1\le j\le s$. Further assume without loss of generality that $P'_1\backslash P\neq \emptyset$,
	and let $m_1\in P'_1\backslash P$. 
	Since we assume that $K\sim K'$, by Lemma \ref{LemmaIII-8}, we have $\mathbb{I}_K\cap P'=\emptyset$, which implies that $m_1\notin \mathbb{I}_K$.
	Together with Remark \ref{remark-1}, we have
	\begin{equation}\label{case1a-b}
		m_1\in P'_1\subseteq P',\ 
		m_1\notin P'_j, 2\le j\le s,\ 
		m_1\notin \mathbb{I}_{K'},\ 
		m_1\notin C',\ 
		m_1\notin \mathbb{I}_K,\  
		m_1\notin P.
	\end{equation}
	
	%Note that in the following discussion, some cases only require \eqref{case1a-b}, while some other cases require a more detailed discussion of the range of $m_1$, which can be divided into the following two sub-cases:
	
	\noindent Since $K'\neq(\cdot,\langle\ \rangle)$, by Lemma \ref{Lemiii5}, either one of the follow conditions holds:
	\begin{itemize}
		\item[i)] $\mathbb{I}_{K'}=\emptyset$ and $s\ge2$;
		\item[ii)] $\mathbb{I}_{K'}\neq\emptyset$.
	\end{itemize}
	
	%\noindent For case i), when $\mathbb{I}_{K'}=\emptyset$, we have $s\ge2$.
	\noindent For Case ii), since we have assumed that $P'_1\neq\emptyset$, 
    we have $s\ge 1$. Since $s\neq1$ (cf. the discussion below Proposition \ref{PropIII.5}), we have $s\ge2$.
	Thus for both Case i) and Case ii), $s\ge2$. So, we can assume without loss of generality that $P'_2\neq\emptyset$ and let $m_2\in P'_2$.  Since we assume that $K\sim K'$, by Lemma \ref{LemmaIII-8}, we obtain $m_2\notin \mathbb{I}_K$. 
    Together with Remark \ref{remark-1} and \eqref{case1a-b}, we have 
    \begin{align}
    \label{ThmIII.2-2}
		&\ \hspace{4.5cm} m_1\notin \mathbb{I}_K,\  
		m_1\notin P,\ 
		m_2\notin \mathbb{I}_K \ 
		%m_2\in P,\ 
		%m_2\notin C
        \\
    \label{ThmIII.2-1}
		&m_1\in P'_1,\ 
		m_1\notin P'_j, 2\le j\le s,\ 
		m_1\notin \mathbb{I}_{K'},\ 
		m_1\notin C',\ 
        m_2\in P'_2,\ 
		m_2\notin P'_j, j\in\mathcal{N}_{s}\backslash\{2\},\ 
		m_2\notin \mathbb{I}_{K'},\ 
		m_2\notin C'.
    \end{align}
    \noindent We now construct a joint distribution for $X_1, \ldots, X_n$ as follows. Let $U$ be a  random variable such that $H(U) > 0$.
	Let 
	\begin{equation}\label{DistributionThmIII.2}
	X_m = \left\{ \begin{array}{ll}
		U & \mbox{if $m =m_1, m_2$} \\
		0 & \mbox{otherwise}.
	\end{array} \right.
	\end{equation}
    By \eqref{ThmIII.2-1}, we obtain
	\begin{align*}
		H(X_{\mathbb{I}_{K'}},X_{\mathbb{I}_{K'}},X_{P'_j},1\le j\le s|X_{C'})&=H(X_{P'_1},X_{P'_2})=H(X_{m_1},X_{m_2})=H(U,U)=H(U)\\
		2H(X_{\mathbb{I}_{K'}}|X_{C'})+\sum\limits_{j=1}^{s}H(X_{P'_j}|X_{C'})&=H(X_{P'_1})+H(X_{P'_2})=H(X_{m_1})+H(X_{m_2})=2H(U).
	\end{align*}
	Thus we have
	\begin{align*}
H(X_{\mathbb{I}_{K'}},X_{\mathbb{I}_{K'}},X_{P'_j},1\le j\le s|X_{C'})
		&\neq 2H(X_{\mathbb{I}_{K'}}|X_{C'})+\sum\limits_{j=1}^{s}H(X_{P'_j}|X_{C'}),
	\end{align*}
	implying $K'$ is invalid. 

    If $m_2\in P$ and $m_1\in C$, we have $m_2\notin C$ by Remark \ref{remark-1}, and by \eqref{ThmIII.2-2} we further obtain
    \begin{align*}
   H(X_{\mathbb{I}_K},X_{\mathbb{I}_K},X_{P_i},1\le i\le t|X_C)&=H(X_{P}|X_C)=H(X_{m_2}|X_{m_1})=H(U|U)=0\\
	2H(X_{\mathbb{I}_K}|X_C)+\sum\limits_{i=1}^{t}H(X_{P_i}|X_C)&=H(X_{P}|X_C)=H(X_{m_2}|X_{m_1})=H(U|U)=0.      \end{align*}
    
    If $m_2\in P$ and $m_1\notin C$, we have $m_2\notin C$ by Remark \ref{remark-1}, and by \eqref{ThmIII.2-2} we further obtain
    \begin{align*}
		H(X_{\mathbb{I}_K},X_{\mathbb{I}_K},X_{P_i},1\le i\le t|X_C)&=H(X_{P})=H(X_{m_2})=H(U)\\
		2H(X_{\mathbb{I}_K}|X_C)+\sum\limits_{i=1}^{t}H(X_{P_i}|X_C)&=H(X_{P})=H(X_{m_2})=H(U).
    \end{align*}
    
    If $m_2\notin P$, by \eqref{ThmIII.2-2}, we obtain
    \begin{align*}
		H(X_{\mathbb{I}_K},X_{\mathbb{I}_K},X_{P_i},1\le i\le t|X_C)&=0\\
		2H(X_{\mathbb{I}_K}|X_C)+\sum\limits_{i=1}^{t}H(X_{P_i}|X_C)&=0.
    \end{align*}
Therefore,
	\begin{align*}
		H(X_{\mathbb{I}_{K}},X_{\mathbb{I}_{K}},X_{P_i},1\le i\le t|X_{C})
		= 2H(X_{\mathbb{I}_{K}}|X_{C})+\sum\limits_{i=1}^{t}H(X_{P_i}|X_{C}),
	\end{align*}
	implying $K$ is valid. 
    
    For the distribution for $X_1,\ldots,X_m$ defined by \eqref{DistributionThmIII.2}, we have proved that $K$ is valid but $K'$ is invalid. 
	Therefore, $K\nsim K'$, a contradiction to our assumption that $K\sim K'$.

				\noindent {\bf Case 2.} $P=P'$
				
				Recall that $P=\cup_{i=1}^{t} P_i$ and $P'=\cup_{j=1}^{s}P'_j$, where $P_i$ and $P'_j$ are nonempty for all $i$ and $j$. If 
				$P=P'=\emptyset$, then $k=l=0$, i.e., $P=P'=\langle\ \rangle$. For this case, the claim that $K\sim K'$ implies $\langle P_i, 1\le i\le t\rangle = \langle P'_j, 1\le j\le s\rangle$ is proved.
				Now consider the case that $P=P'\neq\emptyset$, so that $t\ge1$ and $s\ge1$. If $t=s=1$, then $\langle P_1\rangle=\langle P'_1\rangle$, which is to be proved. Therefore, we assume that $t\ge2$ or $s\ge2$.
				Since $P_i$, $1\le i\le t$ are disjoint, $P'_j$, $1\le j\le s$ are disjoint, and $t\ge2$ or $s\ge2$,
				by Lemma \ref{set-lem}, there exist $i$ and $j$ such that $P'_j\cap P_i\neq \emptyset$ and $P'_j\backslash P_i\ne \emptyset$ or $P_i\backslash P'_j\neq \emptyset$.
				Without loss of generality, we assume that
				$P'_1\cap P_1\neq \emptyset$ and $P'_1\backslash P_1\ne \emptyset$ (the case $P_1\backslash P'_1\ne \emptyset$ can be treated in exactly the same way).
				Let $m_3\in P'_1\cap P_1$ and $m_4\in P'_1\backslash P_1$.
				Since $P=P'$, there exists one  $2\le i\le t$ such that $m_4\in P_i$.
				Without loss of generality, let $m_4\in P_2$.
				Together with Remark \ref{remark-1}, we have
				\begin{align}
					&
					m_3\in P_1,\ 
					m_3\notin P_i, 2\le i\le t,\ 
					m_3\notin \mathbb{I}_{K},\ 
					m_3\notin C\nonumber\\
					&		m_3\in P'_1,\ 
					m_3\notin P'_j, 2\le j\le s,\ 
					m_3\notin \mathbb{I}_{K'},\ 
					m_3\notin C'.\label{con-m3}\\
					&
					m_4\in P_2,\ 
					m_4\notin P_i,\  i\in\mathcal{N}_{t}\backslash\{2\},\ 
					m_4\notin \mathbb{I}_{K},\ 
					m_4\notin C\nonumber\\
					&		m_4\in P'_1,\ 
					m_4\notin P'_j, 2\le j\le s,\ 
					m_4\notin \mathbb{I}_{K'},\ 
					m_4\notin C'.\label{con-m4}
				\end{align}

				We now construct a joint distribution for $X_1, \ldots, X_n$ as follows. Let $U$ be a random variable such that $H(U)>0$.
				Let 
				\[
				X_m = \left\{ \begin{array}{ll}
					U & \mbox{if $m = m_3, m_4$} \\
					0 & \mbox{otherwise}.
				\end{array} \right.
				\]
				By \eqref{con-m3} and \eqref{con-m4}, we obtain
				\begin{align*}
					H(X_{\mathbb{I}_K},X_{\mathbb{I}_K},X_{P_i},1\le i\le t|X_C)
					&=H(X_{P_1},X_{P_2}|X_C)
					{=}H(X_{m_3},X_{m_4})=H(U,U)=H(U)\\
					2H(X_{\mathbb{I}_K}|X_C)+\sum\limits_{i=1}^{t}H(X_{P_i}|X_C)
					&=H(X_{P_1}|X_C)+H(X_{P_2}|X_C)
					=H(X_{m_3})+H(X_{m_4})=2H(U)\\
					H(X_{\mathbb{I}_{K'}},X_{\mathbb{I}_{K'}},X_{P'},1\le i\le s|X_{C'})&=H(X_{P'_1}|X_{C'})=H(X_{m_3},X_{m_4})=H(U,U)=H(U)\\
					2H(X_{\mathbb{I}_{K'}}|X_{C'})+\sum\limits_{i=1}^{s}H(X_{P'_i}|X_{C'})
					&=H(X_{P'_1}|X_{C'})
					=H(X_{m_3},X_{m_4})=H(U,U)
					=H(U).
				\end{align*}
				Thus we have
				\begin{align*}
					H(X_{\mathbb{I}_{K}},X_{\mathbb{I}_{K}},X_{P_i},1\le i\le t|X_{C})
					&\neq 2H(X_{\mathbb{I}_{K}}|X_{C})+\sum\limits_{i=1}^{t}H(X_{P_i}|X_{C})\\
					H(X_{\mathbb{I}_{K'}},X_{\mathbb{I}_{K'}},X_{P'_j},1\le j\le s|X_{C'})
					&= 2H(X_{\mathbb{I}_{K'}}|X_{C'})+\sum\limits_{j=1}^{s}H(X_{P'_j}|X_{C'}),
				\end{align*}
				implying $K$ is invalid and $K'$ is valid. 
				Therefore, $K\nsim K'$, a contradiction to our assumption that $K\sim K'$.
				
				\smallskip
				The proof is accomplished. %\hfill $\Box$
\end{proof}

%\begin{theorem}\label{mainthmQ=Q'}
%		Let $K = (C, \langle Q_i, 1 \le i \le k \rangle )$ and $K' = (C', \langle Q'_j, 1 \le j \le l \rangle )$ be two non-degenerated CMIs, and let ${\rm can}(K)=(C, \langle \mathbb{I}_K, \mathbb{I}_K, P_i, 1\le i\le t\rangle)$ and 
%	${\rm can}(K')=(C',\langle \mathbb{I}_{K'}, \mathbb{I}_{K'}, P'_j, 1\le j\le s\rangle)$.
%	If $K\sim K'$ and $\langle P_i, 1\le i\le t\rangle=\langle P'_j, 1\le j\le s\rangle$, then $C=C'$ and $\mathbb{I}_{K}=\mathbb{I}_{K'}$.
%\end{theorem}

%\begin{proof}
%%If $C=C'$ and $\mathbb{I}_K=\mathbb{I}_{K'}$, then ${\rm can}(K)={\rm can}(K')$, which implies $K\sim K'$.
%
%
%\end{proof}

%Note that  Theorem \ref{mainthmQ=Q'} still holds if $s=0$ or $t=0$.

The next theorem, which is the main result of this section, gives a complete characterization of a pure CMI in terms of its canonical form.

\begin{theorem}\label{mainmainthm}
	Let $K = (C, \langle Q_i, 1 \le i \le k \rangle )$ and $K' = (C', \langle Q'_j, 1 \le j \le l \rangle )$.
	Then $K \sim K'$ if and only if ${\rm can}(K) = {\rm can}(K')$.
\end{theorem}

%\begin{remark}
%	In Theorem~\ref{mainmainthm}, $K$ and $K'$ are assumed to be in pure form. In general, $K$ and $K'$ are not %necessarily in pure form. Then according to Theorem~\ref{mainmainthm}, $K \sim K'$ if and only if ${\rm can}({\rm pur}%(K)) = {\rm can}({\rm pur}(K'))$.
%\end{remark}

%\noindent
%{\it Proof of Theorem~\ref{mainmainthm}.} \ 
\begin{proof}
We first consider the case that either $K$ or $K'$ is degenerate, i.e., either
%
%By Lemma \ref{k=or=l=1}, if
		 $K=(\cdot,\langle\ \rangle)$ or $K'=(\cdot,\langle\ \rangle)$. 
	By Lemma~\ref{k=or=l=1}, $K \sim K'$ if and only if ${\rm can}(K) = {\rm can}(K')$, i.e., the theorem to be proved.
	
	Thus we only have to consider the case that
	 $K\neq(\cdot,\langle\ \rangle)$ and $K'\neq(\cdot,\langle\ \rangle)$, which imply $k\ge 2$ and $l\ge2$.
	Let ${\rm can}(K)=(C, \langle \mathbb{I}_K, \mathbb{I}_K, P_i, 1\le i\le t\rangle)$ and 
	${\rm can}(K')=(C',\langle \mathbb{I}_{K'}, \mathbb{I}_{K'}, P'_j, 1\le j\le s\rangle)$.
	%
	%By Lemma \ref{can=}, we need to prove that $K \sim K'$ if and only if  $C = C'$, $\mathbb{I}_{K'}=\mathbb{I}_{K}$ and $\langle P_i, 1\le i\le t\rangle = \langle P'_j, 1\le j\le s\rangle$. 	
	%
	Since `$\sim$' is transitive (because `$\sim$' is an equivalence relation), 
	it follows from Theorem \ref{k-cank} that if ${\rm can}(K) = {\rm can}(K')$, 
	then $K\sim {\rm can}(K) = {\rm can}(K')\sim K'$, implying that $K \sim K'$.
	Thus we only need to prove the converse, i.e., if $K\sim K'$, then ${\rm can}(K) = {\rm can}(K')$.
	By Theorem~\ref{prop1}, if $K\sim K'$, then $\langle P_i, 1\le i\le t\rangle=\langle P'_j, 1\le j\le s\rangle$.
	%By Theorem \ref{mainthmQ=Q'}, 
	So, together with Lemma \ref{can=}, we only need to prove that if $K\sim K'$ and $\langle P_i, 1\le i\le t\rangle=\langle P'_j, 1\le j\le s\rangle$, then $C=C'$ and $\mathbb{I}_{K}=\mathbb{I}_{K'}$.
	We will prove this by contradiction.
	Assume that $K\sim K'$ and $\langle P_i, 1\le i\le t\rangle = \langle P'_j, 1\le j\le s\rangle$. Let $P=\cup_{i=1}^{t}P_i$ and 
	$P'=\cup_{j=1}^{s}P'_j$, so that $P=P'$. Consider the following cases.
	
	\noindent {\bf Case 1.} $\mathbb{I}_K\ne \mathbb{I}_{K'}$
	
	Without loss of generality, assume $\mathbb{I}_{K'}\backslash\mathbb{I}_K\neq\emptyset$. Let $m_5\in\mathbb{I}_{K'}\backslash\mathbb{I}_K$. 
	Since we assume that $K\sim K'$,
	by Lemma \ref{LemmaIII-8}, we obtain $m_5\notin P$.
	By Remark \ref{remark-1},  we obtain
	\begin{equation}\label{cond-1b}
		m_5\in\mathbb{I}_{K'},\ 
		m_5\notin P',\ 
		m_5\notin C',\ 
		m_5\notin \mathbb{I}_{K},\ 
		m_5\notin P.
	\end{equation}
	
	We now construct a joint distribution for $X_1, \ldots, X_n$ as follows. Let $U$ be a  random variable such that $H(U) > 0$, and
	let 
	\[
	X_m = \left\{ \begin{array}{ll}
		U & \mbox{if $m = m_5$} \\
		0 & \mbox{otherwise}.
	\end{array} \right.
	\]
	
	By \eqref{cond-1b}, we obtain
	\begin{align*}
		H(X_{\mathbb{I}_K},X_{\mathbb{I}_K},X_{P_i},1\le i\le t|X_C)&=0\\
		2H(X_{\mathbb{I}_K}|X_C)+\sum\limits_{i=1}^{t}H(X_{P_i}|X_C)&=0\\
		H(X_{\mathbb{I}_{K'}},X_{\mathbb{I}_{K'}},X_{P'_j},1\le j\le s|X_{C'})&=H(X_{\mathbb{I}_{K'}})=H(X_{m_5})=H(U)\\
		2H(X_{\mathbb{I}_{K'}}|X_{C'})+\sum\limits_{j=1}^{s}H(X_{P'_j}|X_{C'})&=2H(X_{\mathbb{I}_{K'}})=H(X_{m_5})=2H(U).
	\end{align*}
	Thus we have
	\begin{align*}
		H(X_{\mathbb{I}_{K}},X_{\mathbb{I}_{K}},X_{P_i},1\le i\le t|X_{C})
		&= 2H(X_{\mathbb{I}_{K}}|X_{C})+\sum\limits_{i=1}^{t}H(X_{P_i}|X_{C})\\
		H(X_{\mathbb{I}_{K'}},X_{\mathbb{I}_{K'}},X_{P'_j},1\le j\le s|X_{C'})
		&\neq 2H(X_{\mathbb{I}_{K'}}|X_{C'})+\sum\limits_{j=1}^{s}H(X_{P'_j}|X_{C'}),
	\end{align*}
	implying $K$ is valid and $K'$ is invalid. 
	Therefore, $K\nsim K'$, a contradiction to our assumption that $K\sim K'$.
	
	\noindent {\bf Case 2.}  $\mathbb{I}_K= \mathbb{I}_{K'}$ and $C\ne C'$
	
	Since $C\ne C'$, without loss of generality, we assume $C'\backslash C\neq\emptyset$ and let $m_7\in C'\backslash C$. Since $P=P'$ and $\mathbb{I}_{K}=\mathbb{I}_{K'}$, by Remark \ref{remark-1}, we have
	\begin{equation}\label{con-m7}
		m_7\notin C, \ 
		m_7\notin P,\  
		m_7\notin \mathbb{I}_{K},\  
		m_7\in C', \ 
		m_7\notin P',\  
		m_7\notin \mathbb{I}_{K'}.
	\end{equation}
	
	Since $K$ and $K'$ are two non-degenerate CMIs and assumed in pure form, we have $K\neq (\cdot,\langle\ \rangle)$ and $K'\neq (\cdot,\langle\ \rangle)$.  By Lemma \ref{Lemiii5}, either
	one of the follow conditions holds:
	\begin{itemize}
		\item[i)] $\mathbb{I}_K= \mathbb{I}_{K'}\neq\emptyset$;
		
		\item[ii)] $\mathbb{I}_K= \mathbb{I}_{K'}=\emptyset$, $t\ge2$, and $s\ge2$.
	\end{itemize}
	
	\noindent {\bf Case 2a.} $\mathbb{I}_K= \mathbb{I}_{K'}\neq\emptyset$
	
	Let $m_6\in \mathbb{I}_K= \mathbb{I}_{K'}$. 
	Since we assume that $K\sim K'$, 
	by Remark \ref{remark-1} and Lemma \ref{LemmaIII-8}, we have
	\begin{equation}\label{con-m6}
		m_6\notin P, \ 
		m_6\in \mathbb{I}_K,\ 
		m_6\notin C,\ 
		m_6\notin P',\ 
		m_6\in \mathbb{I}_{K'},\ 
		m_6\notin C'.
	\end{equation}

	We now construct a joint distribution for $X_1, \ldots, X_n$ as follows. Let $U$ be a random variable such that $H(U) > 0$.
	Let 
	\[
	X_m = \left\{ \begin{array}{ll}
		U & \mbox{if $m = m_6,m_7$} \\
		0 & \mbox{otherwise}.
	\end{array} \right.
	\]
	By \eqref{con-m7} and \eqref{con-m6}, we obtain
	\begin{align*}
		H(X_{\mathbb{I}_K},X_{\mathbb{I}_K},X_{P_i},1\le i\le t|X_C)&=H(X_{\mathbb{I}_K}|X_C)=H(X_{m_6})=H(U)\\
		2H(X_{\mathbb{I}_K}|X_C)+\sum\limits_{i=1}^{t}H(X_{P_i}|X_C)&=2H(X_{\mathbb{I}_K}|X_C)=2H(X_{m_6})=2H(U)\\
		H(X_{\mathbb{I}_{K'}},X_{\mathbb{I}_{K'}},X_{P'_j},1\le j\le s|X_{C'})&=H(X_{\mathbb{I}_{K'}}|X_{C'})=H(X_{m_6}|X_{m_7})=H(U|U)=0\\
		2H(X_{\mathbb{I}_{K'}}|X_{C'})+\sum\limits_{j=1}^{s}H(X_{P'_j}|X_{C'})&=2H(X_{\mathbb{I}_{K'}}|X_{C'})=2H(X_{m_6}|X_{m_7})=2H(U|U)=0.
	\end{align*}
	Thus we have
	\begin{align*}
		H(X_{\mathbb{I}_{K}},X_{\mathbb{I}_{K}},X_{P_i},1\le i\le t|X_{C})
		&\neq 2H(X_{\mathbb{I}_{K}}|X_{C})+\sum\limits_{i=1}^{t}H(X_{P_i}|X_{C})\\
		H(X_{\mathbb{I}_{K'}},X_{\mathbb{I}_{K'}},X_{P'_j},1\le j\le s|X_{C'})
		&= 2H(X_{\mathbb{I}_{K'}}|X_{C'})+\sum\limits_{j=1}^{s}H(X_{P'_j}|X_{C'}),
	\end{align*}
	implying $K$ is invalid and $K'$ is valid. 
	Therefore, $K\nsim K'$, a contradiction to our assumption that $K\sim K'$.

		\noindent {\bf Case 2b.} $\mathbb{I}_K= \mathbb{I}_{K'}=\emptyset$, $t\ge2$, and $s\ge2$
		
		By $\langle P_i, 1\le i\le t\rangle = \langle P'_j, 1\le j\le s\rangle$ and Definition \ref{mainformula}, we have $t=s$. Without loss of generality, we assume
		$P_1=P'_1\neq\emptyset$ and $P_2=P'_2\neq\emptyset$.
		Let $m_{10}\in P_1=P'_1$ and $m_{11}\in P_2=P'_2$. Since we assume that $K\sim K'$, by Remark~\ref{remark-1}, we have
		\begin{align}
			&m_{10}\in P_1\subseteq P,\ 
			m_{10}\notin P_i, 2\le i\le t,\ 
			m_{10}\notin \mathbb{I}_K,\ 
			m_{10}\notin C\nonumber\\
			&m_{10}\in P'_1\subseteq P',\ 
			m_{10}\notin P'_j, 2\le j\le s,\ 
			m_{10}\notin \mathbb{I}_{K'},\ 
			m_{10}\notin C'.\ 
			\label{con-m10}
			\\
			&m_{11}\in P_2\subseteq P,\ 
			m_{11}\notin P_i,\  i\in\mathcal{N}_{t}\backslash\{2\},\ 
			m_{11}\notin \mathbb{I}_{K},\ 
			m_{11}\notin C\nonumber\\
			&m_{11}\in P'_2\subseteq P',\ 
			m_{11}\notin P'_j, \ 
			j\in\mathcal{N}_{s}\backslash\{2\},\ 
			m_{11}\notin \mathbb{I}_{K'},\ 
			m_{11}\notin C'.\label{con-m11}
		\end{align}

			We now construct a joint distribution for $X_1, \ldots, X_n$ as follows. Let $U$ be a random variables such that $H(U) > 0$.
			Let 
			\[
			X_m = \left\{ \begin{array}{ll}
				U & \mbox{if $m =m_7,m_{10},m_{11}$} \\
				%	V & \mbox{if $m=m_7$}\\
				0 & \mbox{otherwise}.
			\end{array} \right.
			\]
			By \eqref{con-m7}, \eqref{con-m10}, and \eqref{con-m11}, we obtain
			\begin{align*}
				H(X_{\mathbb{I}_K},X_{\mathbb{I}_K},X_{P_i},1\le i\le t|X_C)
				&=H(X_{P_{1}},X_{P_{2}}|X_C)
				=H(X_{m_{10}},X_{m_{11}})
				=H(U,U)=H(U)\\
				2H(X_{\mathbb{I}_K}|X_C)+\sum\limits_{i=1}^{t}H(X_{P_i}|X_C)
				&=H(X_{P_1}|X_C)+H(X_{P_2}|X_C)
				=H(X_{m_{10}})+H(X_{m_{11}})
				=2H(U)\\
				H(X_{\mathbb{I}_{K'}},X_{\mathbb{I}_{K'}},X_{P'_i},1\le j\le s|X_{C'})
				&=H(X_{P'_{1}},X_{P'_{2}}|X_{C'})
				=H(X_{m_{10}},X_{m_{11}}|X_{m_7})=H(U,U|U)=0\\
				2H(X_{\mathbb{I}_{K'}}|X_{C'})+\sum\limits_{j=1}^{s}H(X_{P'_i}|X_{C'})
				&=H(X_{P'_1}|X_{C'})+H(X_{P'_2}|X_{C'})
				=H(X_{m_{10}}|X_{m_{7}})+H(X_{m_{11}}|X_{m_{7}})
				=2H(U|U)=0,
			\end{align*}
			Thus we have
			\begin{align*}
				H(X_{\mathbb{I}_{K}},X_{\mathbb{I}_{K}},X_{P_i},1\le i\le t|X_{C})
				&\neq 2H(X_{\mathbb{I}_{K}}|X_{C})+\sum\limits_{i=1}^{t}H(X_{P_i}|X_{C})\\
				H(X_{\mathbb{I}_{K'}},X_{\mathbb{I}_{K'}},X_{P'_j},1\le j\le s|X_{C'})
				&= 2H(X_{\mathbb{I}_{K'}}|X_{C'})+\sum\limits_{j=1}^{s}H(X_{P'_j}|X_{C'}),
			\end{align*}
			implying $K$ is invalid and $K'$ is valid. 
			Therefore, $K\nsim K'$, a contradiction to our assumption that $K\sim K'$.
			
			Combining {\bf Case 1} and {\bf Case 2}, we have proved the converse and hence the theorem.
The proof is accomplished.
\end{proof}

\section{Implication of a CMI}
\label{sec-imply}

In probability problems, we are often given a CMI or a set of CMIs, and we need to determine whether another given CMI is logically implied. This very basic problem, referred to as the {\it implication problem}, has a
solution if only FCMIs are considered \cite{GeigerPearl1993} \cite{Yeung2002} (see also \cite[Section~12.2]{Yeung2008}). 
The general implication problem, however, has recently been proved to be undecidable if the number of random variables involved is unbounded \cite{Li2023}.\footnote{A problem is undecidable if it is a decision problem for which it is impossible to construct a single algorithm that always produces a correct ``yes'' or ``no'' answer for every possible input in finite time.}

The implication problem is subsumed by the $p$-{\it representability} problem, which studies 
the compatibility of CIs. The latter problem has been solved only up to four random variables \cite{Matus1999}, whose proof involves a variation of the constrained non-Shannon-type inequality reported in \cite{ZhangY97}.

It is noteworthy that the aforementioned problems are all subsumed by the problem of determining all achievable entropy functions (provided that the entropies of the random variables involved are finite), which not only is of fundamental importance in information theory, but also is intimately related to a number of subjects in information sciences (network coding theory, Kolmogorov complexity, cryptography), mathematics (probability theory, combinatorics, group theory, matrix theory, matroid theory), and physics (quantum mechanics). We refer the reader to \cite{Yeung2015} for an exposition on this topic.

In this section, we discuss the implication problem of a CMI, which is not necessarily in pure form. Specifically, for two CMIs $K$ and $K'$, we establish in Theorem~\ref{Kimplication} a necessary and sufficient condition for ``$K$ implies $K'$''.

\begin{define}
Let $K$ and $K'$ be two CMIs.
We say ``$K$ implies $K'$'' to mean that if $K$ is valid, then $K'$ is valid.
\end{define}

It is readily seen that $K\sim K'$ if and only if ``$K$ implies $K'$'' and ``$K'$ implies $K$''.

\begin{proposition}\label{equivalentCMI}
	Let $K$ and $K'$ be two CMIs. Then $K \sim K'$ if and only if ${\rm can}({\rm pur}(K)) = {\rm can}({\rm pur}(K'))$.
\end{proposition}
\begin{proof}
	By Corollary \ref{prop2} and the transitivity of `$\sim$', it can readily be shown that $K \sim K'$ if and only if ${\rm pur}(K) \sim {\rm pur}(K')$. By Theorem \ref{mainmainthm}, we have ${\rm pur}(K) \sim {\rm pur}(K')$ if and only if ${\rm can}({\rm pur}(K)) = {\rm can}({\rm pur}(K'))$. 
	Thus, $K \sim K'$ if and only if ${\rm can}({\rm pur}(K)) = {\rm can}({\rm pur}(K'))$.
\end{proof}

We give a quick summary of the rest of the section here. In Theorems~\ref{IVthm3} to~\ref{IVthm5}, we prove three necessary conditions for ``$K$ implies $K'$''. By combining these necessary conditions, we define a {\it sub-CMI} of $K$ in Definition~\ref{sub-collection-def}. Finally, we establish in Theorem~\ref{Kimplication} that ``$K$ implies $K'$'' if and only if $K'$ is a sub-CMI of $K$.
%{\color{red}Since this section is quite long, let us first give the structure of this section. On the whole, we give some necessary conditions for $K$ to imply $K'$, including the set inclusion relation between the two collections $K$ and $K'$. Then, starting from Definition IV.3, we are going to give some sufficient conditions for $K$ to imply $K'$. Based on these results, we finally give the main result of this section, Theorem \ref{Kimplication}, which gives a necessary and sufficient condition for ``$K$ implies $K'$''. We first state Lemmas \ref{IVlem1} and \ref{IVlem2} that are instrumental in proving these results. }
%Since the proofs of these two lemmas are very tedious, they are deferred to Appendices A-1 and A-2.

\begin{lemma}\label{emptyimplyempty}
	Let $K$ and $K'$ be two CMIs. If $K=(\cdot,\langle\ \rangle)$ and $K$ implies $K'$, then $K'=(\cdot,\langle\ \rangle)$.
\end{lemma}
\begin{proof}
We will prove this lemma by contradiction.
We need to prove that if $K=(\cdot,\langle\ \rangle)$ and $K'\neq(\cdot,\langle\ \rangle)$, then $K$ cannot imply $K'$.
Let ${\rm can}({\rm pur}(K'))=(C, \langle \mathbb{I}_K, \mathbb{I}_K, P'_j, 1\le j\le s\rangle)$. Assume $K'\neq(\cdot,\langle\ \rangle)$. By Lemma \ref{Lemiii5}, either one of the following conditions holds,
\begin{itemize}
	\item[i)] $\mathbb{I}_{K'}=\emptyset$ and $s\ge2$; 
	\item[ii)] $\mathbb{I}_{K'}\neq\emptyset$. 
\end{itemize}

%{\color{blue} In the previous text, $K'$ was mistakenly written as $K$.}

Since $K=(\cdot,\langle\ \rangle)$ is valid for all distributions on $X_1,\ldots,X_n$. However, it can be readily seen that there exists some distribution on $X_1,\ldots,X_n$ such that $K'$ is invalid in case i) and ii), which contradicts that $K$ implies $K'$. 
\end{proof}

\begin{lemma}\label{IVlem1}
	Let $K$ and $K'$ be two CMIs, and let ${\rm can}({\rm pur}(K))=(C, \langle \mathbb{I}_K, \mathbb{I}_K, P_i, 1\le i\le t\rangle)$ and 
	${\rm can}({\rm pur}(K'))=(C',\langle \mathbb{I}_{K'}, \mathbb{I}_{K'}, P'_j, 1\le j\le s\rangle)$, where $s\ge2$. Assume $C\backslash C'\neq\emptyset$, and let $m_1\in C\backslash C'$, $m_2\in P'_{j_1}$, and $m_3\in P'_{j_2}$, where $1\le j_1<j_2\le s$. Then $K$ is valid and $K'$ is invalid for the joint distribution for $X_1,\ldots,X_n$ defined by 
\begin{align}\label{rconstruction-1a}
	X_m = \left\{ \begin{array}{ll}
		U & \mbox{if $m =m_1, m_2, m_3$} \\
		0 & \mbox{otherwise},
	\end{array} \right.
\end{align}
where $U$ is a random variable such that $H(U)>0$.
\end{lemma}
\begin{proof}
	Let $P=\cup_{i=1}^{t}P_i$ and $\mathbb{P}=\mathbb{I}_{K}\cup P$.
	Since $m_1\in C\backslash C'$, by Remark
	 \ref{remark-1}, we have $m_1\in C$ and $m_1\notin \mathbb{P}$.
	
We first prove that $K$ is valid
for the joint distribution for $X_1,\ldots,X_n$ defined by \eqref{rconstruction-1a}.
By \eqref{rconstruction-1a}, all the random variables $X_m, m\in\mathcal{N}_n$ can only be equal to $U$ or $0$.
Since $m_1\in C$, we have
\begin{align*}
	H(X_{\mathbb{I}_{K}},X_{\mathbb{I}_{K}},X_{P_i},1\le i\le t|X_{C})
	&=H(X_{\mathbb{I}_{K}},X_{\mathbb{I}_{K}},X_{P_i},1\le i\le t|U)\\
	&=H(U|U)\ \mbox{or}\ H(0|U)\\
	&=0.
\end{align*}
Similarly, $H(X_{\mathbb{I}_{K}}|X_{C})=0$ and $H(X_{P_i}|X_{C})=0, 1\le i\le t$.
%By \eqref{rconstruction-1a}, no matter for $\{m_2,m_3\}\cap C=\emptyset$ or $\{m_2,m_3\}\cap C\neq\emptyset$, we have $X_{C}=\{U,0\}$.
%For cases $\{m_2,m_3\}\cap \mathbb{P}\neq\emptyset$ or $\{m_2,m_3\}\cap \mathbb{P}=\emptyset$,
%we always obtain from \eqref{rconstruction-1a} that $X_{\mathbb{P}}=\{U,0\}$ or $X_{\mathbb{P}}=\{0\}$.
%Thus $H(X_{\mathbb{I}_{K}},X_{\mathbb{I}_{K}},X_{P_i},1\le i\le t|X_{C})
%=0$.
%Since $\mathbb{I}_K\subseteq\mathbb{P}$ and $P_i\subseteq\mathbb{P}$, $1\le i\le t$, we obtain $H(X_{\mathbb{I}_{K}}|X_{C})=0$ and $H(X_{P_i}|X_{C})=0, 1\le i\le t$. 
Therefore,
\begin{align*}
H(X_{\mathbb{I}_{K}},X_{\mathbb{I}_{K}},X_{P_i},1\le i\le t|X_{C})
=2H(X_{\mathbb{I}_{K}}|X_{C})+\sum\limits_{i=1}^{t}H(X_{P_i}|X_{C})=0,
\end{align*}
implying that $K$ is valid.
	
	Next, we prove that $K'$ is invalid
	for the joint distribution for $X_1, X_2, \ldots,X_n$ defined by \eqref{rconstruction-1a}.
	Let 
	$P'=\cup_{j=1}^{s}P'_j$ and $\mathbb{P}'=\mathbb{I}_{K'}\cup P'$.
	Since $m_1\in C\backslash C'$, we have $m_1\notin C'$.
	Without loss of generality, assume $j_1=1$ and $j_2=2$, so that $m_2\in P'_1$ and $m_3\in P'_2$. Together with Remark \ref{remark-1}, we have $m_2\notin C'\cup \mathbb{I}_{K'}$ and $m_3\notin C'\cup \mathbb{I}_{K'}$. Thus by \eqref{rconstruction-1a}, we obtain $X=0$ for all $X\in X_{C'}$. Again by \eqref{rconstruction-1a}, no matter whether $m_1\in \mathbb{P}'$ or $m_1\notin \mathbb{P}'$, we have
	$	H(X_{\mathbb{I}_{K'}},X_{\mathbb{I}_{K'}},X_{P'_j},1\le j\le s|X_{C'})=H(U)>0$, and
	\begin{align*}
		&2H(X_{\mathbb{I}_{K'}}|X_{C'})+\sum\limits_{j=1}^{s}H(X_{P'_j}|X_{C'})\\
		&=2H(X_{\mathbb{I}_{K'}}|X_{C'})+H(X_{P'_1}|X_{C'})+H(X_{P'_2}|X_{C'})+\sum\limits_{j=3}^{s}H(X_{P'_j}|X_{C'})\\
		&=2H(X_{\mathbb{I}_{K'}}|X_{C'})+2H(U)+\sum\limits_{j=3}^{s}H(X_{P'_j}|X_{C'})\\
		&\ge 2H(U)\\
		&>0.
	\end{align*}
	Thus
	\begin{align*}
		H(X_{\mathbb{I}_{K'}},X_{\mathbb{I}_{K'}},X_{P'_j},1\le j\le s|X_{C'})
		&\neq 2H(X_{\mathbb{I}_{K'}}|X_{C'})+\sum\limits_{j=1}^{s}H(X_{P'_j}|X_{C'}),
	\end{align*}
	implying that $K'$ is invalid. 
\end{proof}

\begin{lemma}\label{IVlem2}
	Let $K$ and $K'$ be two CMIs, and let ${\rm can}({\rm pur}(K))=(C, \langle \mathbb{I}_K, \mathbb{I}_K, P_i, 1\le i\le t\rangle)$ and 
	${\rm can}({\rm pur}(K'))=(C',\langle \mathbb{I}_{K'}, \mathbb{I}_{K'}, P'_j, 1\le j\le s\rangle)$. Assume $C\backslash C'\neq\emptyset$ and $\mathbb{I}_{K'}\neq\emptyset$, and let $m_1\in C\backslash C'$ and $m_4\in \mathbb{I}_{K'}$. Then $K$ is valid and $K'$ is invalid for the joint distribution for $X_1,\ldots,X_n$ defined by
	\begin{align}\label{rconstruction-2a}
		X_m = \left\{ \begin{array}{ll}
			U & \mbox{if $m =m_1, m_4$} \\
			0 & \mbox{otherwise},
		\end{array} \right.
	\end{align}
	where $U$ is a random variable such that $H(U)>0$.
\end{lemma}

\begin{proof}
	Let 
	$P=\cup_{i=1}^{t}P_i$ and $\mathbb{P}=\mathbb{I}_{K}\cup\left(\cup_{i=1}^{t}P_i\right)$.
	Since $m_1\in C\backslash C'$, we have $m_1\in C$ and $m_1\notin \mathbb{P}$.
	
	We first prove $K$ is valid
	for the joint distribution for $X_1,\ldots,X_n$ defined by \eqref{rconstruction-2a}.	
By \eqref{rconstruction-2a}, all $X_m, m\in\mathcal{N}_n$ can only be equal to $U$ or $0$. Since $m_1\in C$, following the steps in the proof of Lemma \ref{IVlem1} for $K$ being valid, we conclude that the same holds here.
%	\begin{align*}
%		H(X_{\mathbb{I}_{K}},X_{\mathbb{I}_{K}},X_{P_i},1\le i\le t|X_{C})
%		=2H(X_{\mathbb{I}_{K}}|X_{C})+\sum\limits_{i=1}^{t}H(X_{P_i}|X_{C})=0,
%	\end{align*}
%	implying $K$ is valid.
%%
%	\\[0.2cm]
%\ 
	
	Next, we prove that $K'$ is invalid
	for the joint distribution for $X_1,\ldots,X_n$ defined by \eqref{rconstruction-2a}.	
	Let 
	$P'=\cup_{j=1}^{s}P'_j$ and $\mathbb{P}'=\mathbb{I}_{K'}\cup\left(\cup_{j=1}^{s}P'_j\right)$. Since $m_1\in C\backslash C'$, we have $m_1\notin C'$.
	By Remark \ref{remark-1}, we have $m_4\notin C'$, and thus by \eqref{rconstruction-2a}, we have $X=0$ for all $X\in X_{C'}$. Then, no matter whether $m_1\in \mathbb{P}'$ or $m_1\notin \mathbb{P}'$, we obtain that
	$		H(X_{\mathbb{I}_{K'}},X_{\mathbb{I}_{K'}},X_{P'_j},1\le j\le s|X_{C'})=H(U)$, and
	$2H(X_{\mathbb{I}_{K'}}|X_{C'})+\sum\limits_{j=1}^{s}H(X_{P'_j}|X_{C'})\ge2H(X_{\mathbb{I}_{K'}}|X_{C'})=2H(U)$.
	Thus
	\begin{align*}
		H(X_{\mathbb{I}_{K'}},X_{\mathbb{I}_{K'}},X_{P'_j},1\le j\le s|X_{C'})
		&\neq 2H(X_{\mathbb{I}_{K'}}|X_{C'})+\sum\limits_{j=1}^{s}H(X_{P'_j}|X_{C'}),
	\end{align*}
	implying that $K'$ is invalid. 
\end{proof}

\begin{theorem}\label{IVthm1}
Let $K$ and $K'$ be two non-degenerate CMIs, and let ${\rm can}({\rm pur}(K))=(C, \langle \mathbb{I}_K, \mathbb{I}_K, P_i, 1\le i\le t\rangle)$ and 
${\rm can}({\rm pur}(K'))=(C',\langle \mathbb{I}_{K'}, \mathbb{I}_{K'}, P'_j, 1\le j\le s\rangle)$. If $K$ implies $K'$, then $C\subseteq C'$.
\end{theorem}
\begin{proof}
We will prove this theorem by contradiction. 
Assume $C\backslash C'\neq\emptyset$, and let $m_1\in C\backslash C'$. Since $K'\neq(\cdot,\langle\ \rangle)$, by Lemma \ref{Lemiii5}, either one of the follow conditions holds:
\begin{itemize}
	\item[i)] $\mathbb{I}_{K'}=\emptyset$ and $s\ge2$;
	\item[ii)] $\mathbb{I}_{K'}\neq\emptyset$.
\end{itemize}

\noindent {\bf Case 1.} $\mathbb{I}_{K'}=\emptyset$ and $s\ge2$

Without loss of generality, 
assume that $P'_1\neq\emptyset$ and $P'_2\neq\emptyset$, and 
let $m_2\in P'_1$ and $m_3\in P'_2$.
Consider the joint distribution for $X_1,\ldots,X_n$ defined by
\begin{align}\label{rconstruction}
	X_m = \left\{ \begin{array}{ll}
		U & \mbox{if $m =m_1, m_2, m_3$} \\
		0 & \mbox{otherwise},
	\end{array} \right.
\end{align}
where $U$ is a random variable such that $H(U)>0$. By Lemma \ref{IVlem1}, $K$ is valid while $K'$ is invalid. Therefore, $K$ does not imply $K'$, a contradiction to our assumption that $K$ implies $K'$.

\noindent {\bf Case 2.} $ \mathbb{I}_{K'}\neq\emptyset$

Let $m_4\in \mathbb{I}_{K'}$. Consider the joint distribution for $X_1,\ldots,X_n$ defined by
\begin{align}\label{rconstruction}
	X_m = \left\{ \begin{array}{ll}
		U & \mbox{if $m =m_1, m_4$} \\
		0 & \mbox{otherwise},
	\end{array} \right.
\end{align}
where $U$ is a random variable such that $H(U)>0$. By Lemma \ref{IVlem2}, $K$ is valid while $K'$ is invalid. Therefore, $K$ does not imply $K'$, a contradiction to our assumption that $K$ implies $K'$.

\vspace{12pt}
The proof is accomplished.
\end{proof}

\begin{cor}\label{IVCC''}
	Let $K$ and $K'$ be two non-degenerate CMIs, and let ${\rm can}({\rm pur}(K))=(C, \langle \mathbb{I}_K, \mathbb{I}_K, P_i, 1\le i\le t\rangle)$ and 
	${\rm can}({\rm pur}(K'))=(C',\langle \mathbb{I}_{K'}, \mathbb{I}_{K'}, P'_j, 1\le j\le s\rangle)$. 
	%	Let $K''=(C'\backslash\mathbb{I}_K,\langle \mathbb{I}_{K'}\backslash\mathbb{I}_K, \mathbb{I}_{K'}\backslash\mathbb{I}_K, P'_j\backslash\mathbb{I}_K, 1\le j\le s\rangle)$ and ${\rm can}({\rm pur}(K''))=(C'',\langle \mathbb{I}_{K''}, \mathbb{I}_{K''}, P''_j, 1\le j\le r\rangle)$.
	If $K$ implies $K'$, then $C\subseteq C'\backslash\mathbb{I}_K$.
\end{cor}

\begin{proof}
	If $K$ implies $K'$, then by Theorem \ref{IVthm1}, we have $C\subseteq C'$. Since $C\cap \mathbb{I}_K=\emptyset$ by Remark \ref{remark-1}, we have $C\subseteq C'\backslash\mathbb{I}_K$.
\end{proof}

\begin{theorem}\label{IVthm2}
Let $K$ and $K'$ be two non-degenerate CMIs, and let ${\rm can}({\rm pur}(K))=(C, \langle \mathbb{I}_K, \mathbb{I}_K, P_i, 1\le i\le t\rangle)$ and 
${\rm can}({\rm pur}(K'))=(C',\langle \mathbb{I}_{K'}, \mathbb{I}_{K'}, P'_j, 1\le j\le s\rangle)$.
If $K$ implies $K'$, then $\mathbb{I}_{K'}\subseteq \mathbb{I}_{K}$.
\end{theorem}

\begin{proof}
We will prove this theorem by contradiction.
Let 
$P=\cup_{i=1}^{t}P_i$.
Assume $\mathbb{I}_{K'}\backslash\mathbb{I}_{K}\neq\emptyset$, and let $m_1\in \mathbb{I}_{K'}\backslash\mathbb{I}_{K}$. 

\noindent {\bf Case 1.} $m_1\in P$

Without loss of generality, let $m_1\in P_1$.
Together with Remark \ref{remark-1}, we have
\begin{align}
	&m_1\notin C,\ m_1\notin \mathbb{I}_K,\ m_1\in P_1,\ m_1\notin P_i, 2\le i\le t,\ m_1\notin C',\  m_1\notin P'_j, 1\le j\le s,\  m_1\in\mathbb{I}_{K'}.\label{eq27a}
\end{align}
We now construct a joint distribution for $X_1, \ldots, X_n$ as follows. Let $U$ be a random variable such that $H(U)>0$.
Let 
\[
X_m = \left\{ \begin{array}{ll}
	U & \mbox{if $m =m_1$} \\
	0 & \mbox{otherwise}.
\end{array} \right.
\]
By \eqref{eq27a}, we obtain
\begin{align*}
	H(X_{\mathbb{I}_K},X_{\mathbb{I}_K},X_{P_i},1\le i\le t|X_C)
	&=H(X_{P_1}|X_C)
	=H(X_{m_1})=H(U)\\
	2H(X_{\mathbb{I}_K}|X_C)+\sum\limits_{i=1}^{t}H(X_{P_i}|X_C)
	&=H(X_{P_1}|X_C)
	=H(X_{m_1})=H(U)\\
	H(X_{\mathbb{I}_{K'}},X_{\mathbb{I}_{K'}},X_{P'_j},1\le j\le s|X_{C'})&=H(X_{\mathbb{I}_{K'}},X_{\mathbb{I}_{K'}}|X_{C'})=H(X_{m_1},X_{m_1})=H(U,U)=H(U)\\
	2H(X_{\mathbb{I}_{K'}}|X_{C'})+\sum\limits_{j=1}^{s}H(X_{P'_j}|X_{C'})
	&=2H(X_{\mathbb{I}_{K'}}|X_{C'})
	=2H(X_{m_1})
	=2H(U).
\end{align*}
Thus we have
\begin{align*}
	H(X_{\mathbb{I}_{K}},X_{\mathbb{I}_{K}},X_{P_i},1\le i\le t|X_{C})
	&= 2H(X_{\mathbb{I}_{K}}|X_{C})+\sum\limits_{i=1}^{t}H(X_{P_i}|X_{C})\\
	H(X_{\mathbb{I}_{K'}},X_{\mathbb{I}_{K'}},X_{P'_j},1\le j\le s|X_{C'})
	&\neq 2H(X_{\mathbb{I}_{K'}}|X_{C'})+\sum\limits_{j=1}^{s}H(X_{P'_j}|X_{C'}),
\end{align*}
implying $K$ is valid and $K'$ is invalid. 
Therefore, $K$ does not imply $K'$, a contradiction to our assumption that $K$ implies $K'$.

\noindent {\bf Case 2.} $m_1\notin P$

Together with Remark \ref{remark-1}, we have
\begin{align}
	m_1\notin \mathbb{I}_K,\ m_1\notin P,\ m_1\notin C',\  m_1\notin P'_j, 1\le j\le s,\  m_1\in\mathbb{I}_{K'}.\label{thmIV2m1case1}
\end{align}
We now construction a joint distribution for $X_1, \ldots, X_n$ as follows. Let $U$ be a random variable such that $H(U)>0$.
Let 
\[
X_m = \left\{ \begin{array}{ll}
	U & \mbox{if $m =m_1$} \\
	0 & \mbox{otherwise}.
\end{array} \right.
\]
By \eqref{thmIV2m1case1}, we obtain
\begin{align*}
	H(X_{\mathbb{I}_K},X_{\mathbb{I}_K},X_{P_i},1\le i\le t|X_C)
	&=0\\
	2H(X_{\mathbb{I}_K}|X_C)+\sum\limits_{i=1}^{t}H(X_{P_i}|X_C)
	&=0\\
	H(X_{\mathbb{I}_{K'}},X_{\mathbb{I}_{K'}},X_{P'_j},1\le j\le s|X_{C'})&=H(X_{\mathbb{I}_{K'}},X_{\mathbb{I}_{K'}}|X_{C'})=H(X_{m_1},X_{m_1})=H(U,U)=H(U)\\
	2H(X_{\mathbb{I}_{K'}}|X_{C'})+\sum\limits_{j=1}^{s}H(X_{P'_j}|X_{C'})
	&=2H(X_{\mathbb{I}_{K'}}|X_{C'})
	=2H(X_{m_1})
	=2H(U).
\end{align*}
Thus we have
\begin{align*}
	H(X_{\mathbb{I}_{K}},X_{\mathbb{I}_{K}},X_{P_i},1\le i\le t|X_{C})
	&= 2H(X_{\mathbb{I}_{K}}|X_{C})+\sum\limits_{i=1}^{t}H(X_{P_i}|X_{C})\\
	H(X_{\mathbb{I}_{K'}},X_{\mathbb{I}_{K'}},X_{P'_j},1\le j\le s|X_{C'})
	&\neq 2H(X_{\mathbb{I}_{K'}}|X_{C'})+\sum\limits_{j=1}^{s}H(X_{P'_j}|X_{C'}),
\end{align*}
implying $K$ is valid and $K'$ is invalid. 
Therefore, $K$ does not imply $K'$, a contradiction to our assumption that $K$ implies $K'$.

\vspace{12pt}
The proof is accomplished.
\end{proof}

%{\color{red}\rule{10cm}{0.1cm}}

\begin{define}\label{RKK'}
Let $K$ and $K'$ be two CMIs, and let ${\rm can}({\rm pur}(K))=(C, \langle \mathbb{I}_K, \mathbb{I}_K, P_i, 1\le i\le t\rangle)$ and ${\rm can}({\rm pur}(K'))=(C',\langle \mathbb{I}_{K'}, \mathbb{I}_{K'}, P'_j, 1\le j\le s\rangle)$. 
Let $T_l=P'_{i_l}\backslash \mathbb{I}_K$, with $1\le l\le u$ and $1\le i_1\le \cdots\le i_u\le s$, be the nonempty sets of $P'_{j}\backslash \mathbb{I}_K$'s.
The CMI
\begin{equation}
R_{K}^{K'} = \left\{ \begin{array}{ll}
( \cdot, \langle \ \rangle) & \mbox{if $\mathbb{I}_{K'}\backslash\mathbb{I}_K = \emptyset$ and $u=0,1$} \\
(C'\backslash\mathbb{I}_K,\langle T_l, 1\le l\le u\rangle) & \mbox{if $\mathbb{I}_{K'}\backslash\mathbb{I}_K = \emptyset$ and $u \ge 2$} \\
(C'\backslash\mathbb{I}_K,\langle \mathbb{I}_{K'}\backslash\mathbb{I}_K, \mathbb{I}_{K'}\backslash\mathbb{I}_K \rangle & \mbox{if $\mathbb{I}_{K'}\backslash\mathbb{I}_K \ne \emptyset$ and $ u =0,1$} \\
(C'\backslash\mathbb{I}_K,\langle \mathbb{I}_{K'}\backslash\mathbb{I}_K, \mathbb{I}_{K'}\backslash\mathbb{I}_K, T_l, 1\le l\le u\rangle) & \mbox{if $\mathbb{I}_{K'}\backslash\mathbb{I}_K \ne \emptyset$ and $ u \ge 2$}
\end{array} \right.
\end{equation} 
is called ``$K'$ conditioning on $K$''.
\end{define}

Note that if $K$ is degenerate, ${\rm can}({\rm pur}(K))$ is also degenerate by Proposition \ref{empty-can}. Then by \eqref{varitionofcan}, $\mathbb{I}_K=\emptyset$. Therefore, if $K$ is degenerate, then $R^{K'}_{K}={\rm can}({\rm pur}(K'))$. 

For the convenience of discussion, in the rest of the paper, we will regard $(C'\backslash\mathbb{I}_K,\langle \mathbb{I}_{K'}\backslash\mathbb{I}_K, \mathbb{I}_{K'}\backslash\mathbb{I}_K, P'_j\backslash\mathbb{I}_K, 1\le j\le s\rangle)$ as the general form of $R_{K}^{K'}$ and adopt the following denotation:
\begin{equation}
(C'\backslash\mathbb{I}_K,\langle \mathbb{I}_{K'}\backslash\mathbb{I}_K, \mathbb{I}_{K'}\backslash\mathbb{I}_K, P'_j\backslash\mathbb{I}_K, 1\le j\le s\rangle) 
\triangleq \left\{ \begin{array}{ll}
( \cdot, \langle \ \rangle) & \mbox{if $\mathbb{I}_{K'}\backslash\mathbb{I}_K = \emptyset$ and $s=0,1$} \\
(C'\backslash\mathbb{I}_K,\langle P'_j\backslash\mathbb{I}_K, 1\le j\le s\rangle) & \mbox{if $\mathbb{I}_{K'}\backslash\mathbb{I}_K = \emptyset$ and $s \ge 2$} \\
(C'\backslash\mathbb{I}_K,\langle \mathbb{I}_{K'}\backslash\mathbb{I}_K, \mathbb{I}_{K'}\backslash\mathbb{I}_K \rangle & \mbox{if $\mathbb{I}_{K'}\backslash\mathbb{I}_K \ne \emptyset$ and $ s =0$.} 	
\end{array} \right.
\end{equation}

\begin{proposition}\label{remark-1a}
	Let $K$ and $K'$ be two non-degenerate CMIs. Let ${\rm can}({\rm pur}(K))=(C, \langle \mathbb{I}_K, \mathbb{I}_K, P_i, 1\le i\le t\rangle)$, and ${\rm can}({\rm pur}(K'))=(C',\langle \mathbb{I}_{K'}, \mathbb{I}_{K'}, P'_j, 1\le j\le s\rangle)$. Let $K''=R^{K'}_{K}$ and ${\rm can}({\rm pur}(K''))=(C'',\langle \mathbb{I}_{K''}, \mathbb{I}_{K''}, P''_j, 1\le j\le r\rangle)$. 
	Then 
	\begin{itemize}
	\item[i)] $C''\cap \mathbb{I}_K=\emptyset$, and $P''_j\cap\mathbb{I}_K=\emptyset$, for $1\le j\le r$;
	\item[ii)] If $K$ implies $K'$, then  $\mathbb{I}_{K''}=\emptyset$ and ${\rm can}({\rm pur}(K'')=(C'',\langle P''_j, 1\le j\le r\rangle)$;
	\item[iii)] If $K''\neq(\cdot,\langle\ \rangle)$, then $r\ge2$.
	\end{itemize}
\end{proposition}
\begin{proof}
	By Definitions \ref{def-Pur} and \ref{def-can}, we see that $C''=C'\backslash\mathbb{I}_K$. Thus, $C''\cap\mathbb{I}_K=\emptyset$. Also, since the repeated indices of $\langle \mathbb{I}_{K'}\backslash\mathbb{I}_K, \mathbb{I}_{K'}\backslash\mathbb{I}_K, P'_j\backslash\mathbb{I}_K, 1\le j\le s\rangle$ can only be contained in $\mathbb{I}_{K'}\backslash\mathbb{I}_K$, we have $\mathbb{I}_{K''}=\mathbb{I}_{K'}\backslash\mathbb{I}_K$, and $P''_j, 1\le j\le r$ are the nonempty sets of $P'_j\backslash\mathbb{I}_{K}, 1\le j\le s$, which implies that $P''_j\cap\mathbb{I}_K=\emptyset$, for $1\le j\le r$.
	Thus, i) is proved.

% {\color{blue}Because $P''_j, 1\le j\le r$ are the nonempty sets of $P'_j\backslash\mathbb{I}_{K}, 1\le j\le s$, $(P'_j\backslash\mathbb{I}_{K})\cap \mathbb{I}_K=\emptyset$, we obtain directly that $P''_j\cap\mathbb{I}_K=\emptyset$. So "Remark III.2" is not needed?
% }

If $K$ implies $K'$, then by Theorem \ref{IVthm2},  $\mathbb{I}_{K'}\backslash\mathbb{I}_{K}=\emptyset$. Thus $\mathbb{I}_{K''}=\emptyset$, which means that ${\rm can}({\rm pur}(K'')=(C'',\langle P''_j, 1\le j\le r\rangle)$.
Thus, ii) is proved.

Finally, if $K''\neq(\cdot,\langle\ \rangle)$, then ${\rm can}({\rm pur}(K''))\neq(\cdot,\langle\ \rangle)$, impling that $r\ge2$.
This proves iii).
\end{proof}

\begin{proposition}\label{K-K'=K-KK'}
	Let $K$ and $K'$ be two non-degenerate CMIs. Let ${\rm can}({\rm pur}(K))=(C, \langle \mathbb{I}_K, \mathbb{I}_K, P_i, 1\le i\le t\rangle)$ and ${\rm can}({\rm pur}(K'))=(C',\langle \mathbb{I}_{K'}, \mathbb{I}_{K'}, P'_j, 1\le j\le s\rangle)$.
\newcounter{cond}
\begin{list}%
{\arabic{cond})}{\usecounter{cond}}
\item
If $R_{K}^{K'}= (\cdot,\langle\ \rangle)$, then $K$ implies $K'$ if and only if $C\subseteq C'$.
\item 
If $R_{K}^{K'} \ne (\cdot,\langle\ \rangle)$, then $K$ implies $K'$ if and only if $K$ implies $R_{K}^{K'}$.
\end{list}
\end{proposition}

We first prove a technical lemma to be used in the proof of Proposition~\ref{K-K'=K-KK'}.

\begin{lemma} \label{tech}
Let $K$ and $K'$ be two non-degenerate CMIs. Let ${\rm can}({\rm pur}(K))=(C, \langle \mathbb{I}_K, \mathbb{I}_K, P_i, 1\le i\le t\rangle)$ and ${\rm can}({\rm pur}(K'))=(C',\langle \mathbb{I}_{K'}, \mathbb{I}_{K'}, P'_j, 1\le j\le s\rangle)$. If $K$ is valid and $C \subseteq C'$, then 
\[
J(X_{\mathbb{I}_{K'}}, X_{\mathbb{I}_{K'}}, X_{P'_j}, 1\le j \le s |X_{C'})
=
J(X_{\mathbb{I}_{K'}\backslash \mathbb{I}_{K}},X_{\mathbb{I}_{K'}\backslash \mathbb{I}_{K}},X_{P'_j \backslash \mathbb{I}_{K}},1\le j \le s|X_{C'\backslash\mathbb{I}_K}) .
\label{erqotf}
\]
\end{lemma}

\begin{proof}
Assume that $K$ is valid and $C \subseteq C'$. By Remark~\ref{remark-1}, $C$ and $\mathbb{I}_{K}$ are disjoint, and so $C \subseteq C'$ implies $C\subseteq C'\backslash\mathbb{I}_{K}$.
By Lemma \ref{cor-1}, if $K$ is valid, then $H(X_{\mathbb{I}_K}|X_C)=0$, so that from the foregoing, we have
\begin{align}
H(X_{\mathbb{I}_K}|X_{C'}) & =0\label{34-1}
\end{align}
and 
\begin{align}
	H(X_{\mathbb{I}_K}|X_{C'\backslash \mathbb{I}_K}) & =0.\label{34-2}
 \end{align}
It follows that
\begin{eqnarray*}
\lefteqn{J(X_{\mathbb{I}_{K'}}, X_{\mathbb{I}_{K'}}, X_{P'_j}, 1\le j \le s |X_{C'})}\nonumber\\
&=& \hspace{-2mm} 2H(X_{\mathbb{I}_{K'}}|X_{C'})+\sum\limits_{j=1}^{s}H(X_{P'_j}|X_{C'})-H(X_{\mathbb{I}_{K'}},X_{\mathbb{I}_{K'}},X_{P'_j},1\le j\le s|X_{C'})\nonumber\\
&\overset{\eqref{34-1}}{=}& \hspace{-2mm} 2H(X_{\mathbb{I}_{K'}\backslash \mathbb{I}_{K}}|X_{C'})+\sum\limits_{j=1}^{s}H(X_{P'_j \backslash \mathbb{I}_{K}}|X_{C'})-H(X_{\mathbb{I}_{K'}\backslash \mathbb{I}_{K}},X_{\mathbb{I}_{K'}\backslash \mathbb{I}_{K}},X_{P'_j\backslash \mathbb{I}_{K}},1\le j \le s|X_{C'})\nonumber\\
&{=}& \hspace{-2mm} 2H(X_{\mathbb{I}_{K'}\backslash \mathbb{I}_{K}}|X_{C'\backslash\mathbb{I}_K},X_{\mathbb{I}_K\cap C'})+\sum\limits_{j=1}^{s}H(X_{P'_j\backslash \mathbb{I}_{K}}|X_{C'\backslash\mathbb{I}_K},X_{\mathbb{I}_K\cap C'}) \nonumber \\
& & \ \ \ -H(X_{\mathbb{I}_{K'}\backslash \mathbb{I}_{K}},X_{\mathbb{I}_{K'}\backslash \mathbb{I}_{K}},X_{P'_j\backslash \mathbb{I}_{K}},1\le j\le s|X_{C'\backslash\mathbb{I}_K},X_{\mathbb{I}_K\cap C'})\nonumber\\
&\overset{\eqref{34-2}}{=}& \hspace{-2mm} 2H(X_{\mathbb{I}_{K'}\backslash \mathbb{I}_{K}}|X_{C'\backslash\mathbb{I}_K})+\sum\limits_{j=1}^{s}H(X_{P'_j\backslash \mathbb{I}_{K}}|X_{C'\backslash\mathbb{I}_K}) -H(X_{\mathbb{I}_{K'}\backslash \mathbb{I}_{K}},X_{\mathbb{I}_{K'}\backslash \mathbb{I}_{K}},X_{P'_j\backslash \mathbb{I}_{K}},1\le j\le s|X_{C'\backslash\mathbb{I}_K})\nonumber\\
&=&J(X_{\mathbb{I}_{K'}\backslash \mathbb{I}_{K}},X_{\mathbb{I}_{K'}\backslash \mathbb{I}_{K}},X_{P'_j\backslash \mathbb{I}_{K}},1\le j\le s|X_{C'\backslash\mathbb{I}_K}), \nonumber
\end{eqnarray*}
proving the lemma.
\end{proof}

\noindent
{\it Proof of Proposition~\ref{K-K'=K-KK'}.}
We first prove the ``only if" part. Assume that $K$ implies
$K'$. Then $C \subseteq C'$ by Theorem~\ref{IVthm1}.
Further assume that $K$ is valid. Then by the assumption that $K$ implies
$K'$, $K'$ is also valid, i.e.
\begin{equation}
J(X_{\mathbb{I}_{K'}}, X_{\mathbb{I}_{K'}}, X_{P'_j}, 1\le j\le s|X_{C'}) = 0.
\label{qigpva}
\end{equation}
Now $K$ is valid and $C \subseteq C'$.
By \eqref{qigpva} and Lemma~\ref{tech}, we obtain that 
\begin{equation}
J(X_{\mathbb{I}_{K'}\backslash \mathbb{I}_{K}},X_{\mathbb{I}_{K'}\backslash \mathbb{I}_{K}},X_{P'_j\backslash \mathbb{I}_{K}},1\le j\le s|X_{C'\backslash\mathbb{I}_K}) = 0,
\label{5438t9j}
\end{equation}
i.e., $R^{K'}_K$ is valid. Then it can readily be checked that the ``only if" part has been proved for both Cases~1) and 2).

We now prove the ``if" part.  For Case~1), Assume that 
$K$ is valid and $C \subseteq C'$. Since $R_{K}^{K'}= (\cdot,\langle\ \rangle)$, the degenerate CMI which is always valid, we see that \eqref{5438t9j} holds (cf.\ Definition~\ref{RKK'}). By Lemma~\ref{tech}, we obtain \eqref{qigpva}, i.e., $K'$ is valid. This proves the ``if" part of Case~1). For Case~2), assume that 
$K$ implies $R_{K}^{K'}$ and $K$ is valid. Since $R_{K}^{K'} \ne (\cdot,\langle\ \rangle)$, by Theorem~\ref{IVthm1}, 
we have $C \subseteq C'$. Since $K$ is valid and 
$K$ implies $R_{K}^{K'}$, $R_{K}^{K'}$ is also valid, i.e., 
\eqref{5438t9j} holds. Then invoke Lemma~\ref{tech} to conclude that (\ref{qigpva}) holds, i.e., $K'$ is valid. This proves the ``if" part for Case~2).

The proof is accomplished. \hfill $\Box$

%Note that, if $R_K^{K'}\neq (\cdot,\langle\ \rangle)$, then by Theorem \ref{IVthm1}, we can obtain $C\subseteq C'\backslash\mathbb{I}_K$ from ``$K$ implies $R_{K}^{K'}$''. Howevere, if $R_K^{K'}= (\cdot,\langle\ \rangle)$, we cannot obtain $C\subseteq C'\backslash\mathbb{I}_K$ from ``$K$ implies $R_{K}^{K'}$''.

\begin{theorem}\label{IVthm3}
	Let $K$ and $K'$ be two non-degenerate CMIs. Let ${\rm can}({\rm pur}(K))=(C, \langle \mathbb{I}_K, \mathbb{I}_K, P_i, 1\le i\le t\rangle)$. Let $K''=R_{K}^{K'}$
	 and ${\rm can}({\rm pur}(K''))=(C'',\langle \mathbb{I}_{K''}, \mathbb{I}_{K''}, P''_j, 1\le j\le r\rangle)$, and  let $P=\cup_{i=1}^{t}P_i$ and $P''=\cup_{i=1}^{r}P''_i$.	If $K$ implies $K'$ and $R_{K}^{K'}\neq (\cdot,\langle\ \rangle)$, then $P''\subseteq P$.
\end{theorem}

\begin{proof}
Let ${\rm can}({\rm pur}(K'))=(C',\langle \mathbb{I}_{K'}, \mathbb{I}_{K'}, P'_j, 1\le j\le s\rangle)$. 
%By Definition \ref{RKK'}, $K''=R_{K}^{K'}=(C'\backslash\mathbb{I}_K,\langle \mathbb{I}_{K'}\backslash\mathbb{I}_K, \mathbb{I}_{K'}\backslash\mathbb{I}_K, P'_j\backslash\mathbb{I}_K, 1\le j\le s\rangle)$.
%By Remark \ref{remark-1}, $C'$, $\mathbb{I}_{K'}$ and $P'_j,1\le j\le s$ are disjoint. Thus $C'\backslash\mathbb{I}_{K}$, $\mathbb{I}_{K'}\backslash\mathbb{I}_{K}$ and $P'_j\backslash\mathbb{I}_{K},1\le j\le s$ are disjoint. 
%%%%
%Since $\mathbb{I}_{K''}$ is the set of repeated indices of 
%%
%\begin{align}\label{**}
%\langle \mathbb{I}_{K'}\backslash\mathbb{I}_K, \mathbb{I}_{K'}\backslash\mathbb{I}_K, P'_j\backslash\mathbb{I}_K, 1\le j\le s\rangle,
%\end{align}
%%
%and noting that the repeated indices of \eqref{**} can only be contained in $\mathbb{I}_{K'}\backslash\mathbb{I}_K$,
%we have  $\mathbb{I}_{K''}\subseteq(\mathbb{I}_{K'}\backslash\mathbb{I}_{K})$.
%%%%%
By Proposition \ref{remark-1a}, if $K$ implies $K'$ and $R_{K}^{K'}\neq (\cdot,\langle\ \rangle)$, then $\mathbb{I}_{K''}=\emptyset$ and ${\rm can}({\rm pur}(K'')=(C'',\langle P''_j, 1\le j\le r\rangle)$ with $r\ge 2$.
%
%In this sense, we can regard ${\rm can}({\rm pur}(K''))=(\cdot,\langle\ \rangle)$ as ${\rm can}({\rm pur}(K''))$ with $P''=\emptyset$, thus $P''\subseteq P$.
%Further, if $K''\neq(\cdot,\langle\ \rangle)$, then ${\rm can}({\rm pur}(K''))\neq(\cdot,\langle\ \rangle)$, impling that $r\ge2$.
%
%We now assume  ${\rm can}({\rm pur}(K''))\neq(\cdot,\langle\ \rangle)$, i.e., $r\ge2$. 
By Proposition \ref{K-K'=K-KK'} and $R_{K}^{K'}\neq (\cdot,\langle\ \rangle)$, $K$ implies $K'$ if and only if $K$ implies $R_{K}^{K'}$. So, we only need to prove that if $K$ implies $R_{K}^{K'}$, i.e., if $K$ implies $K''$, then $P''\subseteq P$.

%If $K$ is valid, then by Lemma \ref{cor-1}, $H(X_{\mathbb{I}_K}|X_C)=0$, which implies $K'\sim K''$. Thus, $K$ implies $K'$ if and only if $K$ implies $K''$. 
%We will first prove that there exists only one corresponding $1\le i\le t$ such that $P''_{j}\subseteq P_{i}$ for every $1\le j\le r$. 
%

We now prove the claim by contradiction. Assume $P''\backslash P\neq\emptyset$.
Without loss of generality, assume $P''_1\backslash P\neq\emptyset$, and let $m_1\in P''_1\backslash P$. 
%which by Remark \ref{remark-1} implies that $m_1\notin P''_j, 2\le j\le r,\ m_1\notin \mathbb{I}_{K''}$ and $m_1\notin C''$. 
Since $r\ge 2$, $P''_2\neq\emptyset$, and we let $m_2\in P''_2$. By Remark \ref{remark-1}, %implies that $m_2\notin P''_j, j\in\mathcal{N}_{s}\backslash\{2\},\  m_2\notin \mathbb{I}_{K''}$ and $m_2\notin C''$. We conclude that
we have
\begin{align}\label{m1m2in''}
	m_1\in P''_1,\ 
	m_1\notin P''_j, 2\le j\le r,\ 
	m_1\notin \mathbb{I}_{K''},\ 
	m_1\notin C'',\
	m_2\in P''_2,\ 
	m_2\notin P''_j, j\in\mathcal{N}_{s}\backslash\{2\},\ 
	m_2\notin \mathbb{I}_{K''},\ 
	m_2\notin C''.
\end{align}

 If $K$ implies $K'$, then by Remark \ref{remark-1} and Theorem \ref{IVCC''}, we have $C\subseteq C'\backslash\mathbb{I}_K=C''$ (cf. Definitions \ref{def-Pur} and \ref{def-can}). Since $m_1\notin C''$ and $m_2\notin C''$, we have $m_1\notin C$ and $m_2\notin C$.
By Proposition \ref{remark-1a}, $P''_1\cap \mathbb{I}_K=\emptyset$ and $P''_2\cap \mathbb{I}_K=\emptyset$. Since $m_1\in P''_1$ and $m_2\in P''_2$, we have $m_1\notin\mathbb{I}_{K}$ and $m_2\notin\mathbb{I}_{K}$.
Thus
\begin{align}\label{m1m2in}
	m_1\notin C,\  
	m_1\notin \mathbb{I}_K,\  
	m_1\notin P,\ 
		m_2\notin \mathbb{I}_K, \ 
	m_2\notin C.
\end{align}

We now construction a joint distribution for $X_1, \ldots, X_n$ as follows. Let $U$ be a  random variable such that $H(U) > 0$.
Let 
\begin{equation}\label{distributionThmIV3}
X_m = \left\{ \begin{array}{ll}
	U & \mbox{if $m =m_1, m_2$} \\
	0 & \mbox{otherwise}.
\end{array} \right.
\end{equation}
By \eqref{m1m2in''}, we obtain
\begin{align*}
H(X_{\mathbb{I}_{K''}},X_{\mathbb{I}_{K''}},X_{P''_j},1\le j\le s|X_{C''})&=H(X_{P''_1},X_{P''_2})=H(X_{m_1},X_{m_2})=H(U,U)=H(U)\\
2H(X_{\mathbb{I}_{K''}}|X_{C''})+\sum\limits_{j=1}^{s}H(X_{P'_j}|X_{C''})&=H(X_{P''_1})+H(X_{P''_2})=H(X_{m_1})+H(X_{m_2})=2H(U).
\end{align*}

\noindent Therefore,
\begin{align*}
	H(X_{\mathbb{I}_{K''}},X_{\mathbb{I}_{K''}},X_{P''_j},1\le j\le s|X_{C''})
	&\neq 2H(X_{\mathbb{I}_{K''}}|X_{C''})+\sum\limits_{j=1}^{s}H(X_{P''_j}|X_{C''}),
\end{align*}
implying that $K''$ is invalid. 

By \eqref{m1m2in}, if $m_2\in P$, then
\begin{align*}
H(X_{\mathbb{I}_K},X_{\mathbb{I}_K},X_{P_i},1\le i\le t|X_C)&=H(X_{P_i},1\le i\le t)=H(X_{m_2})=H(U)\\
2H(X_{\mathbb{I}_K}|X_C)+\sum\limits_{i=1}^{t}H(X_{P_i}|X_C)&=\sum\limits_{i=1}^{t}H(X_{P_i})=H(X_{m_2})=H(U),
\end{align*}
and if $m_2\notin P$, then
\begin{align*}
	H(X_{\mathbb{I}_K},X_{\mathbb{I}_K},X_{P_i},1\le i\le t|X_C)&=0\\
	2H(X_{\mathbb{I}_K}|X_C)+\sum\limits_{i=1}^{t}H(X_{P_i}|X_C)&=0.
\end{align*}
Therefore,
\begin{align*}
	H(X_{\mathbb{I}_{K}},X_{\mathbb{I}_{K}},X_{P_i},1\le i\le t|X_{C})
	= 2H(X_{\mathbb{I}_{K}}|X_{C})+\sum\limits_{i=1}^{t}H(X_{P_i}|X_{C}),
\end{align*}
implying that $K$ is valid.

For the distribution for $X_1,\ldots,X_n$ defined by \eqref{distributionThmIV3}, we have proved that $K$ is valid but $K'$ is invalid. Therefore, $K$ does not imply $K''$, a contradiction to our assumption that $K$ implies $K''$.
The proof is accomplished.
\end{proof}

\begin{theorem}\label{IVthm3a}
	Let $K$ and $K'$ be two non-degenerate CMIs. Let ${\rm can}({\rm pur}(K))=(C, \langle \mathbb{I}_K, \mathbb{I}_K, P_i, 1\le i\le t\rangle)$, and let $K''=R_{K}^{K'}$
	and ${\rm can}({\rm pur}(K''))=(C'',\langle \mathbb{I}_{K''}, \mathbb{I}_{K''}, P''_j, 1\le j\le r\rangle)$.
Further let	 
%$P=\cup_{i=1}^{t}P_i$, 
$S=C\cup\left(\cup_{i=1}^{t}P_i\right)$ and
$P''=\cup_{j=1}^{r}P''_j$. If $K$ implies $K'$ and $R_K^{K'}\neq(\cdot,\langle\ \rangle)$, then $C\subseteq C''\subseteq S\backslash P''$.
\end{theorem}

\begin{proof}
Let ${\rm can}({\rm pur}(K'))=(C',\langle \mathbb{I}_{K'}, \mathbb{I}_{K'}, P'_j, 1\le j\le s\rangle)$. 
If $K$ implies $K'$ and $R_K^{K'}\neq(\cdot,\langle\ \rangle)$, by Proposition \ref{remark-1a}, we have $\mathbb{I}_{K''}=\emptyset$ and $r\ge2$.
By Definitions \ref{def-Pur} and \ref{def-can}, we see that $C''=C'\backslash\mathbb{I}_K$.
If $K$ implies $K'$, then by Corollary \ref{IVCC''}, $C\subseteq C'\backslash\mathbb{I}_K=C''$.
Next, we will prove $C''\subseteq S\backslash P''$ by contradiction.

Assume $C''\backslash(S\backslash P'')\neq\emptyset$, and let $m_5\in C''\backslash(S\backslash P'')$. Since $r\ge2$, without loss of generality, we assume $P''_1\neq\emptyset$ and $P''_2\neq\emptyset$, and let $m_3\in P''_1$ and $m_4\in P''_2$. 

\noindent {\bf Case 1.} $m_3\in  P_i$ and $m_4\in P_i$ for some $1\le i\le t$.

% {\color{blue}We already have $m_3\in P''_1$ and $m_4\in P''_2$, as well as $m_3\in P_1$ and $m_4\in P_1$. Therefore, Remark III.2 is sufficient to obtain (43) and (44).}

Without loss of generality, assume $m_3\in P_1$ and $m_4\in P_1$.
Together with $m_3\in P''_1$, $m_4\in P''_2$, and Remark \ref{remark-1}, we have 
\begin{align}
%	&m_5\notin C,\  m_5\notin \mathbb{I}_K,\  m_5\notin P_i, 1\le i\le t,\ 
%	m_5\in C'', \ m_5\notin P''_j, 1\le j\le r,\  m_5\notin\mathbb{I}_{K''}\label{IVthmm5}\\
	&m_3\notin C,\  m_3\notin \mathbb{I}_K,\  m_3\in P_1,\ m_3\notin P_i, 2\le i\le t,\ 
	m_3\notin C'',\  
	%m_3\notin \mathbb{I}_{K''},\  
	m_3\in P''_1,\ m_3\notin P''_j, 2\le j\le r\label{IVthmm3a}\\
	&m_4\notin C,\  m_4\notin \mathbb{I}_K,\  m_4\in P_1,\ m_4\notin P_i, 2\le i\le t,\ 
	m_4\notin C'',\  
	%m_4\notin \mathbb{I}_{K''},\  
	m_4\in P''_2,\ m_4\notin P''_j, j\in\mathcal{N}_{r}\backslash\{2\}\label{IVthmm4a}.
\end{align}

We now construct a joint distribution for $X_1, \ldots, X_n$ as follows. Let $U$ be a random variables such that $H(U)>0$.
Let 
\[
X_m = \left\{ \begin{array}{ll}
	U & \mbox{if $m =m_3,m_4$} \\
	0 & \mbox{otherwise}.
\end{array} \right.
\]
By \eqref{IVthmm3a}, \eqref{IVthmm4a} and $\mathbb{I}_{K''}=\emptyset$, we obtain
\begin{align*}
	H(X_{\mathbb{I}_K},X_{\mathbb{I}_K},X_{P_i},1\le i\le t|X_C)&=H(X_{P_1})=H(X_{m_3},X_{m_4})=H(U)\\
	2H(X_{\mathbb{I}_K}|X_C)+\sum\limits_{i=1}^{t}H(X_{P_i}|X_C)&=H(X_{P_1})=H(X_{m_3},X_{m_4})=H(U)\\
	H(X_{\mathbb{I}_{K''}},X_{\mathbb{I}_{K''}},X_{P''_j},1\le j\le s|X_{C''})&=H(X_{P''_1},X_{P''_2})=H(X_{m_3},X_{m_4})=H(U)\\
	2H(X_{\mathbb{I}_{K''}}|X_{C''})+\sum\limits_{j=1}^{s}H(X_{P''_j}|X_{C''})&=H(X_{P''_1})+H(X_{P''_2})=H(X_{m_3})+H(X_{m_4})=2H(U).
\end{align*}
We obtain
\begin{align*}
	H(X_{\mathbb{I}_{K}},X_{\mathbb{I}_{K}},X_{P_i},1\le i\le t|X_{C})
	&= 2H(X_{\mathbb{I}_{K}}|X_{C})+\sum\limits_{i=1}^{t}H(X_{P_i}|X_{C})\\
	H(X_{\mathbb{I}_{K''}},X_{\mathbb{I}_{K''}},X_{P''_j},1\le j\le s|X_{C''})
	&\neq 2H(X_{\mathbb{I}_{K''}}|X_{C''})+\sum\limits_{j=1}^{s}H(X_{P''_j}|X_{C''}),
\end{align*}
implying  that $K$ is valid and $K''$ is invalid. 
Therefore, $K$ does not imply $K''$, a contradiction to our assumption that $K$ implies $K''$.

\noindent {\bf Case 2.} $m_3\in  P_i$ and $m_4\in P_j$ for some $1\le i, j\le t$, $i\neq j$.

By Proposition \ref{remark-1a},  $C''\cap\mathbb{I}_K=\emptyset$.
Since $m_5\in C''\backslash(S\backslash P'')$, we have $m_5\notin\mathbb{I}_K$.
Without loss of generality, assume $m_3\in P_1$ and $m_4\in P_2$.
Together with Remark \ref{remark-1} and Proposition \ref{remark-1a}, we have 
\begin{align}
&m_5\notin C,\  m_5\notin \mathbb{I}_K,\  m_5\notin P_i, 1\le i\le t,\ 
m_5\in C'', \ m_5\notin P''_j, 1\le j\le r,\  %m_5\notin\mathbb{I}_{K''}
\label{IVthmm5}\\
&m_3\notin C,\  m_3\notin \mathbb{I}_K,\  m_3\in P_1,\ m_3\notin P_j, 2\le j\le t,\ 
m_3\notin C'',\  
%m_3\notin \mathbb{I}_{K''},\  
m_3\in P''_1,\ m_3\notin P''_j, 2\le j\le r\label{IVthmm3}\\
&m_4\notin C,\  m_4\notin \mathbb{I}_K,\  m_4\in P_2,\ m_4\notin P_j, j\in\mathcal{N}_{t}\backslash\{2\},\ 
m_4\notin C'',\  
%m_4\notin \mathbb{I}_{K''},\  
m_4\in P''_2,\ m_4\notin P''_j, j\in\mathcal{N}_{r}\backslash\{2\}\label{IVthmm4}.
\end{align}

We now construct a joint distribution for $X_1, \ldots, X_n$ as follows. Let $U$ and $V$ be
independent binary random variables that are
uniformly distributed on $\{0,1\}$, and let $W=U+V$ mod $2$.
Then
\begin{align*}
H(U)=H(V)=H(W)=a>0,\ 
H(U,V)=H(U,W)=H(V,W)=2a,\ 
H(U,V,W)=2a,
\end{align*}
where $a={\rm log}2$.
Let 
\[
X_m = \left\{ \begin{array}{ll}
	U & \mbox{if $m =m_3$} \\
		V & \mbox{if $m =m_4$} \\
			W & \mbox{if $m =m_5$} \\
	0 & \mbox{otherwise}.
\end{array} \right.
\]
By \eqref{IVthmm5}, \eqref{IVthmm3}, \eqref{IVthmm4} and $\mathbb{I}_{K''}=\emptyset$, we obtain
\begin{align*}
	H(X_{\mathbb{I}_K},X_{\mathbb{I}_K},X_{P_i},1\le i\le t|X_C)&=H(X_{P_1},X_{P_2})=H(X_{m_3},X_{m_4})=H(U,V)=2a\\
	2H(X_{\mathbb{I}_K}|X_C)+\sum\limits_{i=1}^{t}H(X_{P_i}|X_C)&=H(X_{P_1})+H(X_{P_2})=H(X_{m_3})+H(X_{m_4})=H(U)+H(V)=2a\\
	H(X_{\mathbb{I}_{K''}},X_{\mathbb{I}_{K''}},X_{P''_j},1\le j\le s|X_{C''})&=H(X_{P''_1},X_{P''_2}|X_{C''})=H(X_{m_3},X_{m_4}|X_{m_5})=H(U,V|W)\\
	&=H(U,V,W)-H(W)=a\\
	2H(X_{\mathbb{I}_{K''}}|X_{C''})+\sum\limits_{j=1}^{s}H(X_{P''_j}|X_{C'})&=H(X_{P''_1}|X_{C''})+H(X_{P''_2}|X_{C''})=H(X_{m_3}|X_{m_5})+H(X_{m_4}|X_{m_5})\\
	&=H(U|W)+H(V|W)=H(U,W)+H(V,W)-2H(W)=2a.
\end{align*}
We obtain
\begin{align*}
	H(X_{\mathbb{I}_{K}},X_{\mathbb{I}_{K}},X_{P_i},1\le i\le t|X_{C})
	&= 2H(X_{\mathbb{I}_{K}}|X_{C})+\sum\limits_{i=1}^{t}H(X_{P_i}|X_{C})\\
	H(X_{\mathbb{I}_{K''}},X_{\mathbb{I}_{K''}},X_{P''_j},1\le j\le s|X_{C''})
	&\neq 2H(X_{\mathbb{I}_{K''}}|X_{C''})+\sum\limits_{j=1}^{s}H(X_{P''_j}|X_{C''}),
\end{align*}
implying that $K$ is valid and $K''$ is invalid. 
Therefore, $K$ does not imply $K''$, a contradiction to our assumption that $K$ implies $K''$.

%Now, we proved that prove $C''\subseteq R_1\backslash R_2$, where $R_1=C\cup\mathbb{I}_K\cup\left(\cup_{i=1}^{t}P_i\right)$, $R_2=\mathbb{I}_{K''}\cup\left(\cup_{j=1}^{r}P''_j\right)$.

The proof is accomplished.
\end{proof}

\begin{theorem}\label{IVthm5}
Let $K$ and $K'$ be two non-degenerate CMIs. Let ${\rm can}({\rm pur}(K))=(C, \langle \mathbb{I}_K, \mathbb{I}_K, P_i, 1\le i\le t\rangle)$, $K''=R_{K}^{K'}$,
	and ${\rm can}({\rm pur}(K''))=(C'',\langle \mathbb{I}_{K''}, \mathbb{I}_{K''}, P''_j, 1\le j\le r\rangle)$. Further let	 $P=\cup_{i=1}^{t}P_i$, $P''=\cup_{j=1}^{r}P''_j$. %Assume $P''\subseteq P$. 
If $K$ implies $K'$ and $R_{K}^{K'}\neq(\cdot,\langle\ \rangle)$, then for any $m_1\in P''_{j_1}$ and $m_2\in P''_{j_2}$, where $1\le j_1, j_2\le r,\ j_1\neq j_2$, we have $m_1\in P_{i_1}$ and $m_2\in P_{i_2}$, where $1\le i_1, i_2\le t,\ i_1\neq i_2$.

\end{theorem}

\begin{proof}
%We will prove this theorem by contradiction.
By Theorem \ref{IVthm3} and Proposition \ref{remark-1a}, if $K$ implies $K'$ and $R_{K}^{K'}\neq(\cdot,\langle\ \rangle)$, then $P''\subseteq P$, $\mathbb{I}_{K''}=\emptyset$, and $r\ge2$. 
Assume that $m_1\in P''_{j_1}$ and $m_2\in P''_{j_2}$, where $1\le j_1, j_2\le r,\ j_1\neq j_2$.
Evidently, $m_1\neq m_2$.
Since 
%$m_1\in P''_{j_1}$ and $m_2\in P''_{j_2}$, where $1\le j_1, j_2\le r,\ j_1\neq j_2$ and 
$P''\subseteq P$, we have $m_1\in P_{i_1}$ and $m_2\in P_{i_2}$, where $1\le i_1,i_2\le t$.
We need to prove that $i_1\neq i_2$.

We now prove the theorem by contradiction.
Assume that 
%there exist $m_1\in P''_{j_1}$ and $m_2\in P''_{j_2}$, such that $m_1,m_2\in P_{i}$ 
$i_1=i_2=i$ for some $1\le i\le t$. Without loss of generality, let $m_1\in P''_1$, $m_2\in P''_2$, $m_1\in P_1$ and $m_2\in P_1$.
Together with Remark \ref{remark-1},
%and Proposition \ref{remark-1a}, 
 we have 
\begin{align}
	&m_1\notin C,\  m_1\notin \mathbb{I}_K,\  m_1\in P_1,\ m_1\notin P_j, 2\le j\le t,\ 
	m_1\notin C'',\  
	%m_1\notin \mathbb{I}_{K''},\  
	m_1\in P''_1,\ m_1\notin P''_j, 2\le j\le r\label{IVthmm3aa}\\
	&m_2\notin C,\  m_2\notin \mathbb{I}_K,\  m_2\in P_1,\ m_2\notin P_j, 2\le j\le t,\ 
	m_2\notin C'',\  
	%m_2\notin \mathbb{I}_{K''},\  
	m_2\in P''_2,\ m_2\notin P''_j, j\in\mathcal{N}_{r}\backslash\{2\}\label{IVthmm4aa}.
\end{align}

We now construct a joint distribution for $X_1, \ldots, X_n$ as follows. Let $U$ be a random variables such that $H(U)>0$.
Let 
\[
X_m = \left\{ \begin{array}{ll}
	U & \mbox{if $m =m_1,m_2$} \\
	0 & \mbox{otherwise}.
\end{array} \right.
\]
By \eqref{IVthmm3aa}, \eqref{IVthmm4aa} and $\mathbb{I}_{K''}=\emptyset$, we obtain
\begin{align*}
	H(X_{\mathbb{I}_K},X_{\mathbb{I}_K},X_{P_i},1\le i\le t|X_C)&=H(X_{P_1})=H(X_{m_1},X_{m_2})=H(U)\\
	2H(X_{\mathbb{I}_K}|X_C)+\sum\limits_{i=1}^{t}H(X_{P_i}|X_C)&=H(X_{P_1})=H(X_{m_1},X_{m_2})=H(U)\\
	H(X_{\mathbb{I}_{K''}},X_{\mathbb{I}_{K''}},X_{P''_j},1\le j\le s|X_{C''})&=H(X_{P''_1},X_{P''_2})=H(X_{m_1},X_{m_2})=H(U)\\
	2H(X_{\mathbb{I}_{K''}}|X_{C''})+\sum\limits_{j=1}^{s}H(X_{P''_j}|X_{C''})&=H(X_{P''_1})+H(X_{P''_2})=H(X_{m_1})+H(X_{m_2})=2H(U).
\end{align*}
As a result,
\begin{align*}
	H(X_{\mathbb{I}_{K}},X_{\mathbb{I}_{K}},X_{P_i},1\le i\le t|X_{C})
	&= 2H(X_{\mathbb{I}_{K}}|X_{C})+\sum\limits_{i=1}^{t}H(X_{P_i}|X_{C})\\
	H(X_{\mathbb{I}_{K''}},X_{\mathbb{I}_{K''}},X_{P''_j},1\le j\le s|X_{C''})
	&\neq 2H(X_{\mathbb{I}_{K''}}|X_{C''})+\sum\limits_{j=1}^{s}H(X_{P''_j}|X_{C''}),
\end{align*}
implying that $K$ is valid and $K''$ is invalid. 
Therefore, $K$ does not imply $K''$, a contradiction to our assumption that $K$ implies $K''$.
The proof is accomplished.
\end{proof}

In Theorems~\ref{IVthm3}, \ref{IVthm3a}, and \ref{IVthm5}, we have given three necessary conditions for ``$K$ implies $K'$''. It turns out but is far from obvious that these three conditions together form a sufficient condition for ``$K$ implies $K'$''. In the following, we use these three conditions to define a ``sub-CMI'', which completely characterizes all the CMIs that are implied by a given CMI. This will be formally proved in Theorem~\ref{SubCMI}, the main result of this section.

\begin{define}[Sub-CMI]
\label{sub-collection-def}
Let $K$ and $K'$ be two CMIs, and let ${\rm can}({\rm pur}(K))=(C, \langle \mathbb{I}_K, \mathbb{I}_K, P_i, 1\le i\le t\rangle)$ and
${\rm can}({\rm pur}(K'))=(C', \langle \mathbb{I}'_K, \mathbb{I}'_K, P'_j, 1\le j\le s\rangle)$.
Let $K''=R_{K}^{K'}$
	and ${\rm can}({\rm pur}(K''))=(C'',\langle \mathbb{I}_{K''}, \mathbb{I}_{K''}, P''_j, 1\le j\le r\rangle)$. Further let	 $P=\cup_{i=1}^{t}P_i$, $S=C\cup P$, and $P''=\cup_{j=1}^{r}P''_j$. 
Then $K'$ is called a sub-CMI of $K$ if one of the following conditions holds:
\begin{itemize}	
\item[i)] $K' = (\cdot, \langle \ \rangle)$;
\item[ii)] ${\rm can}({\rm pur}(K''))=(\cdot,\langle\ \rangle)$ and $C\subseteq C'$;
\item[iii)] ${\rm can}({\rm pur}(K''))\neq(\cdot,\langle\ \rangle)$;
$\mathbb{I}_{K''}=\emptyset$; $P''\subseteq P$; $C\subseteq C''\subseteq S\backslash P''$; if $m_1\in P''_{j_1}$ and $m_2\in P''_{j_2}$, where $1\le j_1,j_2\le r,\ j_1\neq j_2$, then $m_1\in P_{i_1}$ and $m_2\in P_{i_2}$, where $1\le i_1, i_2\le t,\ i_1\neq i_2$.
\end{itemize}
\end{define}

% {\color{blue} ${\rm can}({\rm pur}(K''))\neq(\cdot,\langle\ \rangle)$ should be noted in iii). Because we cannot clearly obtain ${\rm can}({\rm pur}(K''))\neq(\cdot,\langle\ \rangle)$ from iii).}

The following example illustrates this definition.

\begin{example}
	Let $n=5$ and $K=(\{1\},\langle\{1,2\},\{2,3\},\{4\},\{5\}\rangle)$.
	In the following, we give three examples of a sub-CMI of $K$.
	%Then by Definition \ref{sub-collection-def}, we have 
	\begin{itemize}
		\item[1)] 
		$K'=(\{1\},\langle\{2\},\{2\}\rangle)$ is a sub-CMI of $K$ because ${\rm can}({\rm pur}(K'))=(\{1\},\langle\{2\},\{2\}\rangle)$,  $\mathbb{I}_K=\{2\}$, and $K'' = R_{K}^{K'}=(\{1\},\langle\ \rangle)=(\cdot,\langle\ \rangle)$, so that ${\rm can}({\rm pur}(K''))=(\cdot,\langle\ \rangle)$ and $C=\{1\}\subseteq \{1\}=C'$, satisfying ii) of Definition~\ref{sub-collection-def}.
		\item[2)]  $K'=(\{1,3\},\langle\{2\},\{3\},\{4\}\rangle)$ is a sub-CMI of $K$ because ${\rm can}({\rm pur}(K'))=(\{1,3\},\langle\{2\},\{4\}\rangle)$,  $\mathbb{I}_K=\{2\}$, and $K'' = R_{K}^{K'}=(\{1,3\},\langle\{4\}\rangle)=(\cdot,\langle\ \rangle)$, so that ${\rm can}({\rm pur}(K''))=(\cdot,\langle\ \rangle)$ and $C=\{1\}\subseteq \{1,3\}=C'$, satisfying ii) of Definition~\ref{sub-collection-def}.
		%\item[3)] 
		%$K'=(\{1\},\langle\{3\},\{4\},\{5\}\rangle)$ is a sub-CMI of $K$ because ${\rm can}({\rm pur}(K'))=(\{1\},\langle\{3\},\{4\},\{5\}\rangle)$,  $\mathbb{I}_K=\{2\}$, and $R_{K}^{K'}=(\{1\},\langle\{3\},\{4\},\{5\}\rangle)$, and then condition ii) of Definition \ref{sub-collection-def} can be verified readily.
		%
		\item[3)]
		$K'=(\{1,2\},\langle\{1\}, \{3\},\{4\}\rangle)$ is a sub-CMI of $K$, which is explained as follows. First, $\mathbb{I}_K=\{2\}$,
		${\rm can}({\rm pur}(K))=(\{1\},\langle\{2\},\{2\}$, $\{3\},\{4\},\{5\}\rangle)$, and ${\rm can}({\rm pur}(K'))=(\{1,2\},\langle\{3\},\{4\}\rangle)$.
		Then
		\begin{align*}
			&C=\{1\}, \mathbb{I}_K=\{2\},  P_1=\{3\}, P_2=\{4\}, P_3=\{5\}, P=\bigcup\limits_{i=1}^{3}P_i=\{3,4,5\}, \\
			&S=C\cup P=\{1,3,4,5\},
			C'=\{1,2\}, \mathbb{I}_{K'}=\emptyset, P'_1=\{3\}, P'_2=\{4\}.
		\end{align*}
		Following Definition \ref{RKK'}, we have
		\begin{align*}
			&C''=C'\backslash\mathbb{I}_K=\{1\},\ 
			\mathbb{I}_{K''}=\mathbb{I}_{K'}\backslash\mathbb{I}_{K}=\emptyset,\ 
			P''_1=P'_1\backslash\mathbb{I}_K=\{3\},\ 
			P''_2=P'_2\backslash\mathbb{I}_K=\{4\},\\
			&P''=\bigcup\limits_{j=1}^{2}P_j^{\prime\prime}=\{3,4\}, \  K''=R_{K}^{K'}=(\{1\},\langle\{3\},\{4\}\rangle),
			S \backslash P'' = \{1,5\}.
		\end{align*} 
		Then, we readily see that 
        ${\rm can}({\rm pur}(K''))\neq(\cdot,\langle\ \rangle)$,
        $\mathbb{I}_{K''}=\emptyset$, $P''\subseteq P$, and $C\subseteq C''\subseteq S\backslash P''$.
		We also see that $m_1\in P''_{j_1}$ and $m_2\in P''_{j_2}$ ($1\le j_1,j_2\le 2,\ j_1\neq j_2$) imply $m_1\in P_{i_1}$ and $m_2\in P_{i_2}$ ($1\le i_1, i_2\le 3,\ i_1\neq i_2$).	
		Therefore, $K'$ satisfies iii) of Definition~\ref{sub-collection-def}.

		%Since ${\rm can}({\rm pur}(K_2))=(\{1,2\},\langle\{3\},\{4\}\rangle)$,  $\mathbb{I}_K=\{2\}$, and 
		%Thus by Definition \ref{sub-collection-def}, $K_2$ is a sub-collection of $K$.  
		%\item[3)]
		%$K_3=(\{1\},\langle\{1,2,3\},\{4,5\}\rangle)$ is a sub-collection of $K$. Since ${\rm can}({\rm pur}(K_3))=(\{1\},\langle\{2,3\},\{4,5\}\rangle)$,  $\mathbb{I}_K=\{2\}$, and $R_{K}^{K_3}=(\{1\}\backslash\{2\},\langle\{2,3\}\backslash\{2\},\{4,5\}\backslash\{2\}\rangle)=(\{1\},\langle\{3\},\{4,5\}\rangle)$. 
		%Thus by Definition \ref{sub-collection-def}, $K_3$ is a sub-collection of $K$.  
	\end{itemize}
	
	%
	%Let $R$ and $E$ be two arbitrary index sets. In the same way, the following collections are sub-collections of $K$:
	%	$$\begin{array}{ll}
		%&((\{1\},R_1),\langle\{2\},E_1\rangle), ((\{1\},R_2),\langle\{2\},\{2\},E_2\rangle),
		%		(\{1\},\langle\{3\},\{4\}\rangle),(\{1,2\},\langle\{3\},\{4\}\rangle),\\&	(\{1\},\langle\{2\},\{3\},\{4\}\rangle),	(\{1\},\langle\{2,3\},\{4\}\rangle), (\{1\},\langle\{3\},\{2,4\}\rangle).
		%	\end{array}$$
	
\end{example}

\begin{proposition}
\label{IIK}
If $K'$ is a sub-CMI of $K$, then $\mathbb{I}_{K''}=\emptyset$.
\end{proposition}
\begin{proof}
If $K'$ is a sub-CMI of $K$, then $K'$ satisfies 
i), ii), or iii) of Definition~\ref{sub-collection-def}.
If $K'$ satisfies i), then $K' = (\cdot, \langle \ \rangle)$ implies $K'' = (\cdot, \langle \ \rangle)$
and ${\rm can}({\rm pur}(K''))=(\cdot,\langle\ \rangle)$, where the latter implies $\mathbb{I}_{K''}=\emptyset$.
If $K'$ satisfies ii), then ${\rm can}({\rm pur}(K''))=(\cdot,\langle\ \rangle)$ again implies $\mathbb{I}_{K''}=\emptyset$. 
Finally, if $K'$ satisfies iii), then obviously $\mathbb{I}_{K''}=\emptyset$. The proof is accomplished by combining all the three cases.
\end{proof}

%\begin{remark}
%Let $K$ be an arbitrary CMI. Let $\bar{K}=(\mathcal{N}_n,\langle\ \rangle)$ and $K'=(\cdot,\langle\ \rangle)$. 
%Then $R_{K}^{\bar{K}}=(\cdot,\langle\ \rangle)$. Since $C\subseteq \mathcal{N}_n$, we obtain by case i) of Definition \ref{sub-collection-def} that $\bar{K}$ is a sub-CMI of any $K$.
%And we have $\bar{K}=K'$ by Definition \ref{def-can}, so $K'$ is a sub-CMI of any $K$.
%\end{remark}}

\begin{proposition}\label{PropIV4}
If $K=(\cdot,\langle\ \rangle)$, then $(\cdot,\langle\ \rangle)$ is the only sub-CMI of $K$. 
\end{proposition}
\begin{proof}
First of all, by i) of Definition~\ref{sub-collection-def}, 
$(\cdot,\langle\ \rangle)$ is a sub-CMI of any CMI $K$, in particular for $K=(\cdot,\langle\ \rangle)$.
Let ${\rm can}({\rm pur}(K))=(C, \langle \mathbb{I}_K, \mathbb{I}_K, P_i, 1\le i\le t\rangle)$. If $K=(\cdot,\langle\ \rangle)$, then ${\rm can}({\rm pur}(K)) =  (\cdot,\langle\ \rangle)$, and by \eqref{varitionofcan}, we have $\mathbb{I}_K=\emptyset$ and $t=0,1$. Assume that $(\cdot,\langle\ \rangle)$ is not the only sub-CMI  of $K$. 
In other words, there exists a CMI $K'\neq (\cdot,\langle\ \rangle)$ which is a sub-CMI of $K$. 
In the following, we will show that $K'$ cannot satisfy i) to iii) of Definition~\ref{sub-collection-def}.

Obviously, $K'$ does not satisfy i) of Definition~\ref{sub-collection-def}.
Since $K'\neq(\cdot,\langle\ \rangle)$, we have ${\rm pur}(K')\neq(\cdot,\langle\ \rangle)$, which by Proposition~\ref{empty-can} implies that ${\rm can}({\rm pur}(K'))\neq(\cdot,\langle\ \rangle)$.\footnote{In Proposition~\ref{empty-can}, $K$ is assumed to be in pure form.}
Let $K''=R_{K}^{K'}$
and ${\rm can}({\rm pur}(K''))=(C'',\langle \mathbb{I}_{K''}, \mathbb{I}_{K''}, P''_j, 1\le j\le r\rangle)$. Since $\mathbb{I}_K=\emptyset$, we obtain $K''=R_{K}^{K'}={\rm can}({\rm pur}(K'))\neq (\cdot,\langle\ \rangle)$, 
which implies that ${\rm can}({\rm pur}(K''))={\rm can}({\rm pur}(K'))\neq(\cdot,\langle\ \rangle)$. Therefore, ii) of Definition~\ref{sub-collection-def} is not satisfied.

Furthermore, since $K''\neq(\cdot,\langle\ \rangle)$, Lemma~\ref{Lemiii5} implies that either one of the following conditions holds:
\begin{itemize}
\item[a)] $\mathbb{I}_{K''}=\emptyset$ and $r\ge2$;
\item[b)] $\mathbb{I}_{K''}\neq\emptyset$.
 \end{itemize}
 We will prove by contradiction that iii) of Definition~\ref{sub-collection-def} is not satisfied.
Assume the contrary that iii) of Definition~\ref{sub-collection-def} is satisfied. 
If a) above holds,  
then $P''\neq\emptyset$. This implies $t\neq0$, because otherwise $P=\emptyset$, violating $P''\subseteq P$ in iii) of Definition~\ref{sub-collection-def}. It follows that $t=1$ and  $(P''_1\cup P''_2)\subseteq P'' \subseteq P = P_1$. Then there exist $m_1\in P''_1$ and $m_2\in P''_2$ such that $\{m_1,m_2\}\subseteq P_1$, which violates the implication in iii) of Definition~\ref{sub-collection-def}, which is a contradiction.
On the other hand, if b) above holds, then $\mathbb{I}_{K''}=\emptyset$ in iii) of Definition~\ref{sub-collection-def} is violated, again a contradiction.

Hence, we have shown that if $K'\neq (\cdot,\langle\ \rangle)$, then it does not satisfy i) to iii) of Definition~\ref{sub-collection-def}, i.e, it is not a sub-CMI of $K$. Therefore, $(\cdot,\langle\ \rangle)$ is the only sub-CMI  of $K$, proving the proposition.
\end{proof}

\begin{proposition}\label{thmin[3]}
	Let $C$ and $Q_i$ be subsets of $\mathcal{N}_n$ and $W_i$ be a subset of $Q_i$ for $1\le i\le k$, where $k\ge 0$. If $K_1=(C,\langle Q_i,\ 1\le i\le k\rangle)$ is valid, 
	then $K_2=(C,\langle {W_i}, 1\le i\le k\rangle)$ is valid.
\end{proposition}
This proposition is rudimentary and its proof is omitted.
%\begin{proof}
%If $K_1$ is valid, i.e.,
%\begin{equation}\label{Q-CMI}
%	J(X_{Q_i},1\leq i\leq k|X_C)=
%	H(X_{Q_i},1\leq i\leq k|X_C)-\sum_{i=1}^kH(X_{Q_i}|X_C)=0.
%\end{equation}
%Then consider
%	\begin{equation}\begin{array}{ll}
%			&J(X_{W_i},1\leq i\leq k|X_C)\\
%			&=\sum_{i=1}^kH(X_{W_i}|X_C)-H(X_{W_i},1\leq i\leq k|X_C)\\ &=\sum_{i=1}^kH(X_{W_i}|X_C)-(H(X_{Q_i},1\le i\le k|X_C)-H(X_{Q_i-W_i},1\leq i\leq k|X_C,X_{W_i},1\le i\le k))\\
%			&=\sum_{i=1}^kH(X_{W_i}|X_C)-
%			\left(\sum_{i=1}^k H(X_{Q_i}|X_C)-\sum_{i=1}^kH(X_{Q_i-W_i}|X_C,X_{W_j},1\le j\le k,X_{Q_l-W_l},1\le l\le i-1)\right)\\
%			&\le \sum_{i=1}^kH(X_{W_i}|X_C)-
%			\sum_{i=1}^k H(X_{Q_i}|X_C)+\sum_{i=1}^kH(X_{Q_i-W_i}|X_C,X_{W_i})\\
%			&=0.
%	\end{array}\end{equation}
%	Since $J(X_{W_i},1\leq i\leq k|X_C)\ge 0$ by Proposition \ref{mainformula},
%	we can obtain $J(X_{W_i},1\leq i\leq k|X_C)= 0$, i.e.,
%	\begin{equation}\begin{array}{ll}
%			H(X_{W_i},1\leq i\leq k|X_C) 
%			=\sum_{i=1}^kH(X_{W_i}|X_C).
%	\end{array}\end{equation}
%	The lemma is proved.
%\end{proof}

%Note that the follow lemma is a generalized version of Yeung's theorem \cite[Theorem 12.5]{Yeung2002}.

\begin{lemma}
	\label{thmin[2]}
Let $C$ and $Q_i$ be subsets of $\mathcal{N}_n$, and 
$W_i$ and $V_i$ be subsets of $Q_i$ such that $Q_i=W_i\cup V_i$ for $1\le i\le k$, where $k\ge 0$. If $K_1=(C,\langle Q_i,\ 1\le i\le k\rangle)$ is valid, 
then $K_2=((C,{V_j},1\le j\le k),\langle {W_i}, 1\le i\le k\rangle)$ is valid.
\end{lemma}

\begin{proof}
If $K_1$ is valid, then
\begin{equation}\label{Q-CMI}
H(X_{Q_i},1\leq i\leq k|X_C)=\sum_{i=1}^kH(X_{Q_i}|X_C).
\end{equation}

Consider
\begin{equation}\label{W-ineq1}
	\begin{array}{ll}
		&J(X_{W_i},1\leq i\leq k|X_C,X_{V_j},1\leq j\leq k)\\
		&=\sum\limits_{i=1}^kH(X_{W_i}|X_C,X_{V_j},1\leq j\leq k)-H(X_{W_i},1\leq i\leq k|X_C,X_{V_j},1\leq j\leq k)\\
		& =\sum\limits_{i=1}^kH(X_{W_i}|X_C,X_{V_j},1\leq j\leq k)-\left(H(X_{Q_i},1\leq i\leq k|X_C)-H(X_{V_j},1\leq j\leq k|X_C)\right)\\
		&\overset{\eqref{Q-CMI}}{=}\sum\limits_{i=1}^kH(X_{W_i}|X_C,X_{V_j},1\leq j\leq k)-\left(\sum\limits_{i=1}^kH(X_{Q_i}|X_C)  -\sum\limits_{i=1}^kH(X_{V_i}|X_C,X_{V_j},1\leq j\leq i-1)\right)
		\\
		& \le\sum\limits_{i=1}^kH(X_{W_i}|X_C,X_{V_j},1\leq j\leq i)  -\left(\sum\limits_{i=1}^kH(X_{Q_i}|X_C,X_{V_j},1\leq j\leq i-1) \right.\\ &  \left.\ \ \ \ \ \ \ \  -\sum\limits_{i=1}^kH(X_{V_i}|X_C,X_{V_j},1\leq j\leq i-1)\right) \\
		&=\sum\limits_{i=1}^kH(X_{W_i}|X_C,X_{V_j},1\leq j\leq i)-\sum\limits_{i=1}^kH(X_{W_i}|X_C,X_{V_j},1\leq j\leq i) \\
		& =0,
\end{array}\end{equation}

Since $J(X_{W_i},1\leq i\leq k|X_C,X_{V_i},1\leq i\leq k)\ge 0$ by Proposition \ref{mainformula},
we obtain $J(X_{W_i},1\leq i\leq k|X_C,X_{V_i},1\leq i\leq k)= 0$, i.e.,
\begin{equation}\begin{array}{ll}
	 H(X_{W_i},1\leq i\leq k|X_C,X_{V_i},1\leq i\leq k) 
	 =\sum_{i=1}^kH(X_{W_i}|X_C,X_{V_j},1\leq j\leq k),
\end{array}\end{equation}
implying that $K_2$ is valid.
The lemma is proved.
\end{proof}

\begin{theorem}\label{newlemma1}
	Let $C$ and $Q_i$ be subsets of $\mathcal{N}_n$,
$W_i\subseteq Q_i$, $1\le i\le k$, where $k\ge 0$, and let $Q=\cup_{i=1}^{k}Q_i$.  Then $K_1=(C,\langle Q_i,\ 1\le i\le k\rangle)$ is valid if and only if $K_2(R)=((C,R),\langle {W_i}, 1\le i\le k\rangle)$ is valid for any $R\subseteq Q\backslash(\cup_{i=1}^{k}W_i)$.
\end{theorem}
\begin{proof}
Assume that $K_2(R)=((C,R),\langle {W_i}, 1\le i\le k\rangle)$ is valid for any $R\subseteq Q\backslash(\cup_{i=1}^{k}W_i)$. In particular, if $W_i=Q_i$ for $1\le i\le k$, then $R=\emptyset$ and $K_2(R)$ becomes $K_1$. This proves the ``if'' part. 
%
%In $K_2$, let $W_i=Q_i$ for $1\le i\le k$, then $R=\emptyset$ which implies $K_1=K_2(R)$. Thus, if $K_2(R)$ is valid, then $K_1$ is valid.
%	

We now prove the ``only if'' part.
Let $R=\cup_{i=1}^{k}R_i$, where $R_i\subseteq Q_i\backslash(\cup_{j=1}^{k}W_j)$, $1\leq i\leq k$. Let $\widetilde{Q}_i=W_i\cup R_i$ for $1\le i\le k$.
Then $\widetilde{Q}_i\subseteq Q_i$.
According to Proposition \ref{thmin[3]}, if $K_1=(C,\langle Q_i,1\le i\le k\rangle)$ is valid, then $K_3=(C,\langle \widetilde{Q}_i,1\le i\le k\rangle)$ is valid. Then by Lemma \ref{thmin[2]}, if $K_3$ is valid, then $K_2(R)=((C, R_i,1\le i\le k),\langle W_i,1\le i\le k \rangle)=((C, R),\langle W_i,1\le i\le k \rangle)$ is valid.
\end{proof}

\begin{theorem}\label{newlemma1aa}
	Let $C$ and $Q_i$ be subsets of $\mathcal{N}_n$,
	$W_i\subseteq Q_i$, $1\le i\le k$, where $k\ge 0$, and let $Q=\cup_{i=1}^{k}Q_i$.  Let $G_j=\cup_{i\in A_j}W_{i}$ for $1\le j\le r\le k$, where $A_j\subseteq\mathcal{N}_{k}, 1\le j\le r$ are disjoint. Let $R$ be any subset of $Q\backslash (\cup_{j=1}^r G_j)$.
	If $K_1=(C,\langle Q_i,\ 1\le i\le k\rangle)$ is valid, then $K_2(R)=((C,R),\langle {G_j}, 1\le j\le r\rangle)$ is valid.
\end{theorem}
\begin{proof}
Consider
\begin{align}
&\sum_{j=1}^r H(X_{G_j}|X_C,X_R)-H(X_{G_j},1\le j\le r|X_C,X_R)\nonumber\\
&=\sum_{j=1}^r H(X_{\cup_{i\in A_j}W_{i}}|X_C,X_R)-H(X_{\cup_{i\in A_j}W_{i}}, 1\le j\le r|X_C,X_R)\nonumber\\
&=\sum_{j=1}^r H(X_{W_{i}},i\in A_j|X_C,X_R)-H(X_{W_{i}},i\in A_j, 1\le j\le r|X_C,X_R)\nonumber\\
&\le \sum_{j=1}^r \sum_{i\in A_j}H(X_{W_{i}}|X_C,X_R )-H(X_{W_{i}}, i\in A_j,1\le j\le r|X_C,_R).\label{96a}
\end{align}
If $K_1=(C,\langle Q_i,\ 1\le i\le k\rangle)$ is valid, then by Theorem \ref{newlemma1}, $K_2(R)=((C,R),\langle {W_i}, i\in A_j,1\le j\le r\rangle)$, where $R\subseteq Q\backslash (\cup_{j=1}^r G_j)= Q\backslash(\cup_{j=1}^{r}\cup_{i\in A_j}W_i)$, is valid.
Thus, we have 
\begin{align}\label{IV.7-1}
\sum_{j=1}^r \sum_{i\in A_j}H(X_{W_{i}}|X_C,X_R )-H(X_{W_{i}}, i\in A_j,1\le j\le r|X_C,X_R)=0.
\end{align}
Combining \eqref{96a} and \eqref{IV.7-1}, we obtain 
\begin{align*}
\sum_{j=1}^r H(X_{G_j}|X_C,X_R)-H(X_{G_j},1\le j\le r|X_C,X_R)=0,
\end{align*}
which implies that $K_2(R)=((C,R),\langle {G_j}, 1\le j\le r\rangle)$ is valid.
\end{proof}

\begin{theorem}
	\label{SubCMI}
	%Let $K=(C, \langle Q_i, 1\le i\le k\rangle)$. 
	Let $K$ be a CMI, and $K'$ be a sub-CMI of $K$.
	If $K$ is valid, then $K'$ is valid.
\end{theorem}

\begin{proof}
If $K'$ is degenerate, then the theorem is trivial. If $K$ is degenerate, then by Proposition \ref{PropIV4}, $K'=(\cdot,\langle\ \rangle)$ is the only sub-CMI, and the theorem follows.

Next, we assume that both $K$ and $K'$ are non-degenerate.
Let ${\rm can}({\rm pur}(K))=(C, \langle \mathbb{I}_K, \mathbb{I}_K, P_i, 1\le i\le t\rangle)$, 
${\rm can}({\rm pur}(K'))=(C',\langle \mathbb{I}_{K'}, \mathbb{I}_{K'}, P'_j, 1\le j\le s\rangle)$, $K''=R_{K}^{K'}$,
 and ${\rm can}({\rm pur}(K''))=(C'',\langle \mathbb{I}_{K''}, \mathbb{I}_{K''}, P''_j, 1\le j\le r\rangle)$. Further let $P=\cup_{i=1}^{t}P_i$, $S=C\cup P$, and $P''=\cup_{j=1}^{r}P''_j$. 
By Proposition \ref{K-K'=K-KK'}, we obtain
%$K$ implies $K'$ if and only if one of the following conditions holds:

% \textbf{Condition i)} $K$ implies $R_{K}^{K'}$, where $R_{K}^{K'}= (\cdot,\langle\ \rangle)$ and $C\subseteq C'$,

% \textbf{Condition ii)} $K$ implies $R_{K}^{K'}$, where $R_{K}^{K'}\neq (\cdot,\langle\ \rangle)$.

\begin{list}%
{\arabic{cond})}{\usecounter{cond}}
\item
If $K''=R_{K}^{K'}= (\cdot,\langle\ \rangle)$, then $K$ implies $K'$ if and only if $C\subseteq C'$.
\item 
If $K''=R_{K}^{K'} \ne (\cdot,\langle\ \rangle)$, then $K$ implies $K'$ if and only if $K$ implies $R_{K}^{K'}$.
\end{list}

Since $K'$ is a sub-CMI of $K$ and $K'$ is non-degenerate, by Definition~\ref{sub-collection-def}, either one of the following cases holds:
\begin{list}%
{\alph{cond})}{\usecounter{cond}}
\item 
${\rm can}({\rm pur}(K''))=(\cdot,\langle\ \rangle)$ and $C\subseteq C'$.
\item 
${\rm can}({\rm pur}(K''))\neq(\cdot,\langle\ \rangle); \mathbb{I}_{K''}=\emptyset$; $P''\subseteq P$; $C\subseteq C''\subseteq S\backslash P''$; if $m_1\in P''_{j_1}$ and $m_2\in P''_{j_2}$, where $1\le j_1, j_2\le r,\ j_1\neq j_2$, then $m_1\in P_{i_1}$ and $m_2\in P_{i_2}$, where $1\le i_1, i_2\le t,\ i_1\neq i_2$.
\end{list}
If a) holds,
since ${\rm can}({\rm pur}(K''))=(\cdot,\langle\ \rangle)$, we obtain by Proposition \ref{empty-can} that $K''=(\cdot,\langle\ \rangle)$. By $C\subseteq C'$, we prove by 1) of Proposition \ref{K-K'=K-KK'} that $K$ implies $K'$.
If b) holds, we first have  $K''\neq(\cdot,\langle\ \rangle)$. By 2) of Proposition \ref{K-K'=K-KK'}, we only need to prove that if $K$ is valid, then $K''$ is valid. This will be done in the rest of the proof.

Assume $K$ is valid, which by Theorem \ref{k-cank} implies that ${\rm can}({\rm pur}(K))$ is valid, and further by Theorem \ref{newlemma1} implies that $(C, \langle P_i, 1\le i\le t\rangle)$ is valid. 

Let $V_{ji}=P''_j\cap P_i$, where $1\le j\le r$ and $1\le i\le t$, and let $A_j=\{i| V_{ji}\neq\emptyset, 1\le i\le t\}$, where $1\le j\le r$.
%To see this, consider $V_{ji}$ and $V_{j'i'}$, where $(j,i) \ne (j',i')$. 
%Then 
%\[
%V_{ji} \cap V_{j'i'} = (P''_j \cap P_i) \cap  (P''_{j'} \cap P_{i'})
%= (P''_j \cap P''_{j'}) \cap (P_i \cap P_{i'}).
%\]
%Since $(j,i) \ne (j',i')$, we have $j \ne j'$ or $i \ne i'$. If $j \ne j'$, then 
%$P''_j \cap P''_{j'} = \emptyset$, implying that $V_{ji} \cap V_{j'i'} = %\emptyset$.
%If $i \ne i'$, then 
%$P_i \cap P_{i'} = \emptyset$, implying that $V_{ji} \cap V_{j'i'} = \emptyset$.
%This proves the claim that $V_{ji}$ and $V_{j'i'}$.
Since $P''\subseteq P$ and $\{P_i,1\le i\le t\}$ is a partition of $P$, we have
$$P''_j = P''_j \cap P =\bigcup\limits_{i=1}^{t}(P''_j\cap P_i)=\bigcup\limits_{i\in A_j}V_{ji}, \ 1\le j\le r.$$
Note that for a fixed $1 \le j \le r$, the sets $V_{ji}$, $1 \le i \le t$  are disjoint because $\{P_i,1\le i\le t\}$ is a partition of $P$.
Thus we see that $\{V_{ji}, i\in A_j\}$ forms a partition of $P''_j$ for $1\le j\le r$. 

Next, we prove by contradiction that $A_j, 1\le j \le r$ are disjoint. 
If $A_j, 1\le j\le r$ are not disjoint, we assume without loss of generality that $A_1\cap A_2\neq\emptyset$, where $A_1=\{i|V_{1i}\neq\emptyset,1\le i\le t\}$ and $A_2=\{i|V_{2i}\neq\emptyset, 1\le i\le t\}$. Then there exist some $1\le i\le t$ such that $V_{1i}\neq\emptyset$ and $V_{2i}\neq\emptyset$, where $V_{1i}=P''_1\cap P_i$ and $V_{2i}=P''_2\cap P_i$. Further assume without loss of generality that $i=1$. Since $P''_1$ and $P''_2$ are disjoint, there exist distinct $m_1$ and $m_2$ such that $m_1\in P''_1$, $m_2\in P''_2$, and $m_1,m_2\in P_1$, which contradicts the implication in condition iii) of Definition~\ref{sub-collection-def}. Thus we obtain that $A_j\subseteq\mathcal{N}_{t}, 1\le j\le r$ are disjoint. 
%which implies $r\le t$.

%
For $1\le i\le t$, let
\[
W_i = \left\{ \begin{array}{ll}
	V_{ji} & \mbox{if $i\in A_{j}$ for some $1 \le j \le r$} \\
	\emptyset & \mbox{otherwise}.
\end{array} \right.
\]
Note that $W_i$ is well defined because $A_j, 1\le j\le r$ are disjoint.
According to the above definition, if $i\in A_{j}$ for some $1 \le j \le r$,
then $W_i = V_{ji} = P''_j\cap P_i$, otherwise $W_i = \emptyset$. Thus
\begin{align}\label{THMIV8-1}
W_i \subseteq P_i, \ \ \ \mbox{for  $1 \le i \le t$.}
\end{align}
Let  $G_j=\cup_{i\in A_j}W_i$. Then we have
\begin{align}\label{THMIV8-2}
G_j=\bigcup\limits_{i\in A_j}W_i=\bigcup\limits_{i\in A_j}V_{ji}=P''_j.
\end{align}
Since $C\subseteq C''\subseteq S\backslash P''$, we have $C''=C\cup R$, where $R=C''\backslash C$. Then
\begin{align}\label{THMIV8-3}
R\subseteq (S\backslash P'')\backslash C=S\backslash(C\cup P'')=(C \cup P)\backslash(C\cup P'')= P\backslash(C\cup P'')\overset{(a)}{=}P\backslash P''
=P\backslash (\cup_{j=1}^r G_j).
\end{align}
%where $P=\cup_{i=1}^t P_i$. 
Note that $(a)$ in \eqref{THMIV8-3} is due to that $C\cap P=\emptyset$ by Remark \ref{remark-1}.
By \eqref{THMIV8-2} and $\mathbb{I}_{K''}=\emptyset$ in condition iii) of Definition \ref{sub-collection-def}, we have 
\begin{align}\label{THMIV8-4}
{\rm can}({\rm pur}(K''))=(C'',\langle P''_j,1\le j\le r\rangle)=((C,R),\langle P''_j,1\le j\le r\rangle)=((C,R),\langle G_j,1\le j\le r\rangle).
\end{align}
Since $(C, \langle P_i, 1\le i\le t\rangle)$ is valid, by \eqref{THMIV8-1}, \eqref{THMIV8-3}, \eqref{THMIV8-4}, and Theorem \ref{newlemma1aa}, we obtain that ${\rm can}({\rm pur}(K''))=((C,R),\langle G_j,1\le j\le r\rangle)$ is valid.
Thus, by Theorem \ref{k-cank}, $K''$ is valid.
The proof is accomplished.
\end{proof}

\begin{define}\label{setofcollections2}
	Let $\mathcal{K}(K)$ be the set of all sub-CMIs of $K$.
\end{define}

\begin{lemma}\label{emptyiff}
	If either $K$ or $K'$ is degenerate, and $K$ implies $K'$, then $K'\in \mathcal{K}(K)$.
\end{lemma}
\begin{proof}
	If $K$ implies $K'$ and $K=(\cdot,\langle\ \rangle)$, then by Lemma \ref{emptyimplyempty}, $K'=(\cdot,\langle\ \rangle)$, which by Proposition \ref{PropIV4} implies $K'\in \mathcal{K}(K)$.
	If $K$ implies $K'$, $K\neq(\cdot,\langle\ \rangle)$, and $K'=(\cdot,\langle\ \rangle)$, then since $(\cdot,\langle\ \rangle)$ is the sub-CMI of any CMI, we have $K'\in\mathcal{K}(K)$.
\end{proof}

We now prove the main theorem of this section.

\begin{theorem}\label{Kimplication}
	Let $K$ and $K'$ be two CMIs.	Then $K$ implies $K'$ if and only if $K'$ is a sub-CMI of $K$, i.e.,
    $K'\in \mathcal{K}(K)$.
\end{theorem}
\begin{proof}
	If $K'\in \mathcal{K}(K)$, then by Theorem \ref{SubCMI}, $K$ implies $K'$.
	Next, we will prove the ``only if'' part, i.e., if $K$ implies $K'$ then $K'\in\mathcal{K}(K)$.
	
\noindent {\bf Case 1.} $K=(\cdot,\langle\ \rangle)$ or $K'=(\cdot,\langle\ \rangle)$.
	
	If $K$ implies $K'$, then by Lemma \ref{emptyiff}, $K'\in\mathcal{K}(K)$.
	
\noindent {\bf Case 2.} $K\neq(\cdot,\langle\ \rangle)$, $K'\neq(\cdot,\langle\ \rangle)$, and $R_{K}^{K'}=(\cdot,\langle\ \rangle)$.
	
	We first have $C\subseteq C'$ by Theorem \ref{IVthm1}.
	Note that ${\rm pur}(R_{K}^{K'}) = (\cdot,\langle\ \rangle)$.
	By Proposition \ref{empty-can}, we have ${\rm can}({\rm pur}(R_{K}^{K'}))=(\cdot,\langle\ \rangle)$. Thus by condition i) of Definition \ref{sub-collection-def}, $K'\in\mathcal{K}(K)$.
	
\noindent {\bf Case 3.} $K\neq(\cdot,\langle\ \rangle)$, $K'\neq(\cdot,\langle\ \rangle)$, and $R_{K}^{K'}\neq(\cdot,\langle\ \rangle)$.
		
	If $K$ implies $K'$, then we obtain $\mathbb{I}_{K''}=\emptyset$ by Proposition \ref{remark-1a}, $P''\subseteq P$ by Theorem \ref{IVthm3}, $C\subseteq C''\subseteq S\backslash P''$ by Theorem \ref{IVthm3a}, and the implication in condition ii) of Definition \ref{sub-collection-def} by Theorem \ref{IVthm5}.
	Thus, condition ii) of Definition \ref{sub-collection-def} is satisfied, which implies $K'\in\mathcal{K}(K)$.
	
	The proof is accomplished by combining Cases 1 to 3.
\end{proof}

We have proved in Proposition~\ref{equivalentCMI} that $K\sim K'$ if and only if ${\rm can}({\rm pur}(K))={\rm can}({\rm pur}(K'))$, and proved in Theorem~\ref{Kimplication} that $K$ implies $K'$ if and only if $K'\in \mathcal{K}(K)$. 
Since $K\sim K'$ if and only if ``$K$ implies $K'$'' and ``$K'$ implies $K$'',
a direct consequence is that ${\rm can}({\rm pur}(K))={\rm can}({\rm pur}(K'))$ if and only if $K'\in \mathcal{K}(K)$ and $K\in \mathcal{K}(K')$.
In the following, we give a direct proof of this result, which depends only on the definitions of the canonical form (Definition~\ref{def-Pur}), the pure form (Definition~ \ref{canonical-form-def}), and sub-CMI (Definition~\ref{sub-collection-def}). This proof verifies the consistency of our previous definitions and related theorems.

%{\color{red}
\begin{theorem}
	Let $K$ and $K'$ be two CMIs. 
	Then ${\rm can}({\rm pur}(K))={\rm can}({\rm pur}(K')$ if and only if $K'\in \mathcal{K}(K)$ and $K\in \mathcal{K}(K')$.	
\end{theorem}
\begin{proof}
	Let 
	\begin{align*}
	{\rm can}({\rm pur}(K)) & =(C, \langle \mathbb{I}_K, \mathbb{I}_K, P_i, 1\le i \le t\rangle) \\
	{\rm can}({\rm pur}(K')) & =(C',\langle \mathbb{I}_{K'}, \mathbb{I}_{K'}, P'_j, 1\le j\le s\rangle) \\
	{\rm can}({\rm pur}(K'')) & =(C'',\langle \mathbb{I}_{K''},
	\mathbb{I}_{K''}, P''_j, 1\le j\le r_1\rangle) \\
	{\rm can}({\rm pur}(K''')) & =(C''',\langle \mathbb{I}_{K'''},
	\mathbb{I}_{K'''}, P'''_j, 1\le j\le r_2\rangle).
	\end{align*}
	where $K''=R^{K'}_{K}$ and $K''' =R^{K}_{K'}$.
	Let
	$P=\cup_{i=1}^{t}P_i$, $P'=\cup_{j=1}^{s}P'_j$,
	$S=C\cup P$, $S'=C'\cup P'$,
	$P''=\cup_{j=1}^{r_1}P''_j$, and $P'''=\cup_{j=1}^{r_2}P'''_j$.
	
	We first prove the ``only if'' part.
	If ${\rm can}({\rm pur}(K))={\rm can}({\rm pur}(K')$, by Lemma \ref{can=}, we have 
	\begin{align}\label{RemThm-1}
		C=C',\ \mathbb{I}_{K}=\mathbb{I}_{K'},\ \langle P_i, 1\le i\le t\rangle=\langle P'_j, 1\le j\le s\rangle, \ \mbox{and} \ t = s.
	\end{align}
	By Remark~\ref{remark-1}, $P_i,1\le i\le t$, $\mathbb{I}_K$ and $C$ are disjoint, and $P'_j,1\le j\le s$, $\mathbb{I}_{K'}$ and $C'$ are disjoint.
	By Definition~\ref{RKK'}, we obtain $K'' = R_{K}^{K'}=(C',\langle P'_j,1\le j\le s\rangle)$ and $K''' = R_{K'}^{K}=(C,\langle P_i,1\le i\le t\rangle)$. Thus, 
	\begin{align*}
	(C'',\langle P''_j,1\le j\le r_1\rangle) & = {\rm can}({\rm pur}(K'')) = (C',\langle P'_j,1\le j\le s\rangle) \\
	(C''',\langle P'''_j,1\le j\le r_2 \rangle) & = {\rm can}({\rm pur}(K''')) = (C,\langle P_i,1\le i\le t\rangle) ,
	\end{align*}
	which together with \eqref{RemThm-1} imply that 
	\begin{eqnarray} \label{RemThm-2}
	\begin{split}
		& C''=C'=C'''=C,\ \mathbb{I}_{K''}= \mathbb{I}_{K'''}=\emptyset, \\ 
		& \langle P''_j,1\le j\le r_1\rangle=\langle P'_j,1\le j\le s\rangle = \langle P_i,1\le i\le t\rangle = \langle P'''_j,1\le j\le r_2 \rangle  \\
		& P' = P'' = P = P''', \ \mbox{and} \ r_1 = s = t = r_2.
	\end{split}
	\end{eqnarray}
	If $s=0, 1$, then ${\rm can}({\rm pur}((K'')) = (\cdot,\langle\  \rangle)$. Together with $C=C'$, by ii) of Definition~\ref{sub-collection-def}, we have $K'\in \mathcal{K}(K)$.
	If $s\ge2$, by \eqref{RemThm-1} and \eqref{RemThm-2}, we have 
    ${\rm can}({\rm pur}(K''))\neq(\cdot,\langle\ \rangle)$,
	$\mathbb{I}_{K''}=\emptyset$, $P''\subseteq P$, $C\subseteq C''\subseteq$\footnote{Here, 
		$S\backslash P = (C \cup P) \backslash P = C$ $(= C'')$ because $C$ and $P$ are disjoint by Remark~\ref{remark-1}.} $S\backslash P = S\backslash P''$,
	 and if $m_1\in P''_{j_1}$ and $m_2\in P''_{j_2}$, where $1\le j_1,j_2\le r_1,\ j_1\neq j_2$, then $m_1\in P_{i_1}$ and $m_2\in P_{i_2}$, where $1\le i_1, i_2\le t,\ i_1\neq i_2$. Thus, iii) of Definition \ref{sub-collection-def} is satisfied, which implies $K'\in \mathcal{K}(K)$.
	In the same way, we can prove that $K\in \mathcal{K}(K')$.
	
	Now, we prove the ``if'' part. Assume that $K\in \mathcal{K}(K')$ and $K'\in \mathcal{K}(K)$.
	First, consider the case that either $K=(\cdot,\langle\  \rangle)$ or $K'=(\cdot,\langle\  \rangle)$. Without loss of generality, we assume that $K=(\cdot,\langle\  \rangle)$.
	Since $K'\in\mathcal{K}(K)$, by Proposition \ref{PropIV4}, $K'=(\cdot,\langle\  \rangle)$. Thus, ${\rm can}({\rm pur}(K))={\rm can}({\rm pur}(K')=(\cdot,\langle\  \rangle)$.
	
	Next, consider the case that both $K$ and $K'$ are non-degenerate.
	Since $K'\in\mathcal{K}(K)$ and $K\in\mathcal{K}(K')$, we obtain by Proposition~\ref{IIK} that $\mathbb{I}_{K''}=\emptyset$ and $\mathbb{I}_{K'''}=\emptyset$. From $\mathbb{I}_{K''}=\emptyset$, we have
	$\emptyset = \mathbb{I}_{K''} = \mathbb{I}_{K'} \backslash \mathbb{I}_{K}$, which implies $\mathbb{I}_{K'}\subseteq \mathbb{I}_{K}$. Similarly, $\mathbb{I}_{K'''}=\emptyset$ implies $\mathbb{I}_{K}\subseteq \mathbb{I}_{K'}$.
	 Thus, $\mathbb{I}_{K}= \mathbb{I}_{K'}$. By Remark \ref{remark-1}, $P_i,1\le i\le t$, $\mathbb{I}_K$, and $C$ are disjoint, and $P'_j,1\le j\le s$, $\mathbb{I}_{K'}$, and $C'$ are disjoint.
	Thus, we obtain
	\begin{eqnarray}\label{RemEqn-1}
	\begin{split}
		(C'',\langle P''_j,1\le j\le r_1\rangle) & = {\rm can}({\rm pur}(K''))
		 = {\rm can}({\rm pur}(R_K^{K'}))
		=(C',\langle P'_j,1\le j\le s\rangle) \\
		(C''',\langle P'''_j,1\le j\le r_2\rangle) & = {\rm can}({\rm pur}(K''')) = {\rm can}({\rm pur}(R_{K'}^{K}))
		=(C,\langle P_i,1\le i\le t\rangle),
	\end{split}
	\end{eqnarray}
	which implies $C''=C'$, $C'''=C$, $\langle P'_j,1\le j\le s\rangle=\langle P''_j,1\le j\le r_1\rangle$, $s = r_1$, $\langle P_i,1\le i\le t\rangle=\langle P'''_j,1\le j\le r_2\rangle$, and $t=r_2$.

	\medskip \noindent \textbf{Case i)}\ ${\rm can}({\rm pur}(K''))=(\cdot,\langle\  \rangle)$ or ${\rm can}({\rm pur}(K''')=(\cdot,\langle\  \rangle)$.
	
	%By ii) of Definition \ref{sub-collection-def}, we have $C\subseteq C'$ from $K'\in\mathcal{K}(K)$, and $C'\subseteq C$ from $K\in\mathcal{K}(K')$. 
    From $K'\in\mathcal{K}(K)$, by either ii) or iii) of Definition \ref{sub-collection-def} \big(i) of Definition \ref{sub-collection-def} does not apply because $K'$ is non-degenerate\big), we have $C\subseteq C'$. Likewise, from $K\in \mathcal{K}(K')$, we have $C'\subseteq C$.
    Thus, $C=C'$.
	Then, without loss of generality, we assume ${\rm can}({\rm pur}(K''))=(\cdot,\langle\  \rangle)$, which implies that $s=r_1 = 0, 1$. By Definition \ref{def-can}, we obtain ${\rm can}({\rm pur}(K'))=(C',\langle\mathbb{I}_{K'}, \mathbb{I}_{K'}\rangle)$ and $P'=\emptyset$.
	Now. we prove that ${\rm can}({\rm pur}(K'''))=(\cdot,\langle\  \rangle)$ by contradiction. Assume that 
	${\rm can}({\rm pur}(K''')) \ne (\cdot,\langle\  \rangle)$.
	Since $K\in\mathcal{K}(K')$, if ${\rm can}({\rm pur}(K'''))\neq(\cdot,\langle\  \rangle)$, then $t\ge 2$ and $P'''\neq \emptyset$ so that $P'''\nsubseteq P'=\emptyset$, which contradicts iii) of Definition \ref{sub-collection-def}. Therefore, we conclude that ${\rm can}({\rm pur}(K'''))=(\cdot,\langle\  \rangle)$, which implies $t=r_2 = 0, 1$, and so ${\rm can}({\rm pur}(K))=(C,\langle\mathbb{I}_{K}, \mathbb{I}_{K}\rangle)$. Since $\mathbb{I}_K=\mathbb{I}_{K'}$ and $C=C'$, we have proved that ${\rm can}({\rm pur}(K'))={\rm can}({\rm pur}(K))$.
	
	\medskip \noindent \textbf{Case ii)}\ ${\rm can}({\rm pur}(K''))\neq (\cdot,\langle\  \rangle)$ and ${\rm can}({\rm pur}(K''')\neq (\cdot,\langle\  \rangle)$.
	
	First, we have $s\ge2$ and $t\ge2$ from \eqref{RemEqn-1}.
	By iii) of Definition \ref{sub-collection-def}, we obtain $P'=P''\subseteq P$ from $K'\in \mathcal{K}(K)$, and $P=P'''\subseteq P'$ from $K\in \mathcal{K}(K')$. Thus, $P=P'=P''=P'''$.
	By iii) of Definition \ref{sub-collection-def} and $K'\in\mathcal{K}(K)$, we also have 
	\begin{align}
		C\subseteq C''\subseteq S\backslash P''=S\backslash P=(C\cup P)\backslash P=C,
	\end{align}
	which implies $C=C''=C'$.
	
	Now, we prove that
	$\langle P'_j,1\le j\le s\rangle=\langle P_i,1\le i\le t\rangle$, where $s\ge2$ and $t\ge2$.
	Since $\langle P'_j,1\le j\le s\rangle=\langle P''_j,1\le j\le r_1\rangle$, we only need to prove that $\langle P''_j,1\le j\le r_1\rangle=\langle P_i,1\le i\le t\rangle$.
	
	Assume that $\langle P''_j,1\le j\le r_1\rangle\neq\langle P_i,1\le i\le t\rangle$. Since $s\ge2$, $t\ge2$, and $P = P''$, by Lemma \ref{set-lem}, there exist $i_1$ and~$j_1$ such that $P_{i_1}\cap P''_{j_1}\neq\emptyset$, and $P_{i_1}\backslash P''_{j_1}\neq\emptyset$ or $P''_{j_1}\backslash P_{i_1}\neq\emptyset$. 
	Without loss of generality, we assume that $P_1\cap P''_1\neq \emptyset$ and $P_1\backslash P''_1\neq \emptyset$ (the case $P''_1\backslash P_1 \neq\emptyset$ can be treated in exactly the same way). Let $m_1\in P_1\cap P''_1$ and $m_2\in P_1\backslash P''_1$. Since $P=P''$ and $m_2 \not\in P''_1$, there exist $2\le k \le t$ such that $m_2\in P''_k$. Without loss of generality, let $m_2\in P''_2$. So, we have $m_1\in P_1$, $m_2\in P_1$, $m_1\in P''_1$ and $m_2\in P''_2$, which contradicts the condition in iii) of Definition~\ref{sub-collection-def} for $K'\in \mathcal{K}(K)$ that if $m_1\in P''_{j_1}$ and $m_2\in P''_{j_2}$, where $1\le j_1,j_2\le r_1,\ j_1\neq j_2$, then $m_1\in P_{i_1}$ and $m_2\in P_{i_2}$, where $1\le i_1, i_2\le t,\ i_1\neq i_2$. Thus, we have proved that $\langle P''_j,1\le j\le r_1\rangle=\langle P_i,1\le i\le t\rangle$.
	
	The proof is accomplished.
\end{proof}
%}

\begin{define}\label{setofcollections1}
	Let $\prod =\{K_i,1\le i\le l\}$
	be a set of CMIs. Then $\bar{K}$ is called a sub-CMI of $\prod$ if $\bar{K}$ is a sub-CMI of $K_i$ for some $1\le i\le l$.
	Let  $\mathcal{K}(\prod)$ be the set of all sub-CMIs of $\prod$.
\end{define}

Let $\prod =\{K_i,1\le i\le l\}$. By Definitions \ref{setofcollections1} and \ref{setofcollections2}, we can directly obtain 
\[\mathcal{K}\left(\prod\right)=\bigcup\limits_{i=1}^{l}\mathcal{K}(K_i).\]

\begin{define}
Let $\prod_1$ and $\prod_2$ be two sets of CMIs.	We say $\prod_1$ implies $\prod_2$ to mean that if all the CMIs in $\prod_1$ are vaild, then all the CMIs in $\prod_2$ are valid.
\end{define}

\begin{define}
Two sets of CMIs $\prod_1$ and $\prod_2$ on $X_1,\ldots,X_m$ are equivalent, denoted by $\prod_1\sim\prod_2$, if $\prod_1$ and $\prod_2$ are both valid or both invalid for every joint distribution of $X_1,\ldots,X_m$.
\end{define}

%is equiva
%By Theorem \ref{Kimplication}, we can obtain the following corollary directly.
\begin{cor}\label{II-cor1}
Let $\prod_1$ and $\prod_2$ be two sets of CMIs. If  $\prod_2\subseteq\mathcal{K}(\prod_1)$, then $\prod_1$ implies $\prod_2$.
\end{cor}
\begin{proof}
Consider $K'\in \prod_2$. If $\prod_2\subseteq\mathcal{K}(\prod_1)$, then $K'\in\mathcal{K}(\prod_1)$, which means that $K'\in \mathcal{K}(K)$ for some $K\in \prod_1$. If $\prod_1$ are valid, then $K$ is valid, which by Theorem \ref{Kimplication} implies that $K'$ is valid. Thus, we have proved that if $\prod_2\subseteq \mathcal{K}(\prod_1)$, then $\prod_1$ implies $\prod_2$.
\end{proof}

\begin{cor}
Let $\prod_1$ and $\prod_2$ be two sets of CMIs. If  $\prod_2\subseteq\mathcal{K}(\prod_1)$ and $\prod_1\subseteq\mathcal{K}(\prod_2)$, then $\prod_1\sim\prod_2$.	
\end{cor}
\begin{proof}
If $\prod_2\subseteq\mathcal{K}(\prod_1)$, by Corollary \ref{II-cor1}, then $\prod_1$ implies $\prod_2$.
On the other hand, if $\prod_1\subseteq\mathcal{K}(\prod_2)$, then $\prod_2$ implies $\prod_1$. The corollary is proved.
\end{proof}

\section{Conclusion and Discussion}
\label{sec-conc}

In this paper, we studied conditional mutual independence for a finite set of discrete random variables. Our focus is on how to characterize a single conditional mutual independency (CMI). Specifically, 
for two conditional mutual independencies $K$ and $K'$, we obtained complete characterizations for ``$K$ is equivalent to $K'$'' and ``$K$ implies $K'$''. To our knowledge, these are the first such characterizations in the literature.

The results in this paper belong to two types. The first type are those proved by using Shannon's information measures, and the second type are those proved by constructing the underlying probability distribution with certain required properties. In the following, we remark on the generality of these results:
\begin{itemize}
\item 
The limitation of using Shannon information measures is that the entropies of the random variables involved must be finite. Nevertheless, results of the first type can instead be proved by working directly on the underlying probability distribution, although the proofs are in general very tedious. Moreover, these results can be proved for general probability distributions, without assuming that the random variables are discrete. 
\item 
For the second type of results, with some suitable modifications, the construction of the underlying probability distribution continues to apply when the random variables are not discrete.
\end{itemize}
Hence, the results in this paper can readily be extended for general probability distributions.

\section*{Acknowledgment}
The authors would like to thank Prof.\ Cheuk Ting Li for the useful discussion.

%Previously, full conditional closed form characterizations of 
%In this paper, we obtain the new algebraic characterization of the generalized CMIs.
%Based on the new characterization, we proved the necessary and sufficient conditions for the equivalence of CMIs and the necessary and sufficient conditions for the implication of CMIs.
%The following discussion can be extended to full conditional mutual independence and Markov fields.

\renewcommand{\appendixname}{Appendix~\Alph{section}}
\appendix
\section*{Conditional Mutual Independence Versus Conditional Independence}\label{app:A}
In this appendix, we prove that a CMI is equivalent to a collection of CIs. To simply notation, we will use $X_i^{j}$ to denote $X_i,X_{i+1},\ldots,X_{j}$. If $i > j$, $X_i^{j}$ is taken to be a degenerate random variable.

\begin{proposition}
	\label{propA1}
$X_1, X_2 \ldots, X_n$ are mutually independent conditioning on $Y$ if and only if 
\begin{eqnarray}
I(X_1; X_2^n|Y) & = & 0 \label{App1a}\\
I(X_2; X_3^n|Y,X_1) & = & 0 \label{App2a}\\
I(X_3; X_4^n|Y,X_1^2) & = & 0 \label{App3a}\\
%I(X_4; X_5^n|Y,X_1^3) & = & 0 \label{App4a}\\
& \vdots & \nonumber\\
%I(X_{n-2}; X_{n-1}^n|Y,X_1^{n-3}) & = & 0\label{Appn-2a}\\
I(X_{n-1}; X_n|Y,X_1^{n-2}) & = & 0\label{Appn-1a}.
\end{eqnarray}
\end{proposition}

%{\color{blue}
\begin{remark}
Proposition~\ref{propA1} asserts that the CMI
$\perp(X_1, X_2, \ldots, X_n) | Y$
is equivalent to the collection of CMIs
\begin{eqnarray*}
&	X_1 \perp X_2^n | Y & \\
&	X_2 \perp X_3^n | (Y, X_1)  \\
&   X_3 \perp X_4^n | (Y, X_1^2) \\
& \vdots & \\
& X_{n-1} \perp X_n | (Y, X_1^{n-2}) .
\end{eqnarray*}
If we let ${\cal A} = \{ X_1, X_2, \ldots, X_n, Y\}$, then $\perp(X_1, X_2, \ldots, X_n) | Y$ is full (with respect to $\cal A$), and so are all the CIs above. In other words, an FCMI is equivalent to a collection of FCIs. For a given $\cal A$, the set of all FCIs is evidently a subset of the set of all FCMIs. On the other hand, from the above discussion, the set of all FCMIs is also a subset of the set of all FCIs. Therefore, the set of all FCMIs is the same as the set of all FCIs.
\end{remark}
%}

\begin{proof}
We first prove the ``only if'' part.  Rewrite (\ref{App1a}) through (\ref{Appn-1a}) as
\begin{equation}
I(X_i; X_{i+1}^n | Y, X_1^{i-1}) = 0
\label{eqn*}
\end{equation}
for $1 \le i \le n-1$.
Fix $1 \le i \le n-1$ and consider
    \begin{eqnarray}
        H(X_1^n|Y)
        & = & H(X_1^i|Y) + H(X_{i+1}^n|Y,X_1^i)\nonumber   \\
        & = & H(X_1^i|Y) + H(X_{i+1}^n|Y)  -I(X_1^i;X_{i+1}^n|Y) \nonumber \\
        & = & \sum_{j=1}^i H(X_j | Y, X_1^{j-1}) + \sum_{j=i+1}^n   H(X_j| Y, X_{i+1}^{j-1})  \nonumber\\ 
        &  & ~~~- \big[ I(X_1^{i-1};X_{i+1}^n|Y) + I(X_i;X_{i+1}^n|Y,X_1^{i-1}) \big] \nonumber \\
        & = & \sum_{j=1}^i \big[H(X_j|Y) - I(X_j;X_1^{j-1}|Y)\big] + \sum_{j=i+1}^n \big[H(X_j|Y) - I(X_j;X_{i+1}^{j-1}|Y)\big]\nonumber\\ 
        &  & ~~~- I(X_1^{i-1};X_{i+1}^n|Y) - I(X_i;X_{i+1}^n|Y,X_1^{i-1}) \nonumber \\
        & = & \sum_{j=1}^n H(X_j|Y) - \sum_{j=1}^{i}I(X_j;X_1^{j-1}|Y) - \sum_{j=i+1}^n I(X_j;X_{i+1}^{j-1}|Y) \nonumber\\
        & & ~~~- I(X_1^{i-1};X_{i+1}^n|Y) - I(X_i;X_{i+1}^n|Y,X_1^{i-1}).\label{App-1}
    \end{eqnarray}
If $X_1, X_2 \ldots, X_n$ are mutually independent conditioning on $Y$, then 
$H(X_1^n|Y)=\sum_{j=1}^n H(X_j|Y)$.
Since conditional mutual information is nonnegative, from \eqref{App-1}, we obtain that $I(X_i;X_{i+1}^n|Y,X_1^{i-1})=0$, i.e., (\ref{eqn*}), where $1\le i\le n-1$. This proves the ``only if'' part.

We now prove the ``if'' part. 
Assume that (\ref{eqn*}) holds for $1 \le i \le n-1$.
Toward proving the ``if'' part, we will first prove that 
\begin{equation}
	I(X_1^i;X_{i+1}^n|Y)=0
\label{eqn**}
\end{equation}
for $1\le i\le n-1$ by induction on $i$. For $i=1$, (\ref{eqn**}) is equivalent to (\ref{eqn*}) for $i = 1$. Assume that (\ref{eqn**}) is true for $i = j-1$ for some $1 \le j \le n-2$. We now prove (\ref{eqn**}) for $i = j$. Consider
\begin{equation}
I(X_1^j; X_{j+1}^n | Y ) = 	I(X_1^{j-1}; X_{j+1}^n | Y ) + I(X_j; X_{j+1}^n | Y, X_1^{j-1} ).
\label{qr89uhf}
\end{equation}
In the above, for the first term on the right hand side, we have
\[
0 \le I(X_1^{j-1}; X_{j+1}^n | Y ) \le I(X_1^{j-1}; X_j^n | Y ) = 0, 
\]
where the equality above follows from the induction hypothesis. This implies that 
$I(X_1^{j-1}; X_{j+1}^n | Y ) = 0.$
i.e., the first term on the right hand side of (\ref{qr89uhf}) vanishes. Since we assume that (\ref{eqn*}) holds for all $1 \le i \le n-1$, in particular for $i = j$, the second term on the right hand side of (\ref{qr89uhf}) also vanishes. Therefore, 
$I(X_1^j; X_{j+1}^n | Y ) = 0,$
i.e., we have proved (\ref{eqn**}) for $i = j$. Hence, 
(\ref{eqn**}) is true for all $1 \le i \le n-1$.

Next, for $1 \le i \le n-1$, consider 
\[
0 \le I(X_i; X_{i+1}^n | Y ) \le I(X_1^i; X_{i+1}^n | Y ) = 0,
\]
where the equality above follows from (\ref{eqn**}). Then we conclude that
\begin{equation}
I(X_i; X_{i+1}^n | Y ) = 0 \ \ \ \mbox{for} \ 1 \le i \le n-1.
\label{lajtnr}
\end{equation}

%When $i=1$, by \eqref{App1a}, we directly obtain 
%\begin{align}
%I(X_1;X_2^n|Y)=0.\label{Ap5} 
%\end{align}
%
%Next, we will consider the cases for $2\le i\le n-1$.
%We first show by induction that if $I(X_i;X_{i+1}^n|Y,X_1^{i-1})=0$, then $I(X_1^{i};X_{i+1}^n|Y)=0$. 
%%$\Rightarrow I(X_i;X_1^{i-1}|Y)=0.$ 
%Assume 
%    \begin{align}
%    I(X_1^{i-1};X_i^n|Y)=0,\label{Ap1}
%    \end{align}
%    it suffices to prove $I(X_1^{i};X_{i+1}^n|Y)=0$. 
%    
%Since
%\begin{equation}
%I(X_1^{i-1};X_{i}^n|Y)=
%I(X_1^{i-1};X_{i+1}^n|Y)
%+I(X_1^{i-1};X_{i}^n|Y,X_{i+1}^n),  
%\end{equation}
%by \eqref{Ap1}, we obtain 
%\begin{align}
%I(X_1^{i-1};X_{i+1}^n|Y)=0.\label{Ap2}
%\end{align}
%
%Thus, by \eqref{Ap3}\eqref{Ap2}, we obtain 
%\begin{align}
%I(X_1^{i};X_{i+1}^n|Y)= I(X_1^{i-1};X_{i+1}^n|Y) + I(X_i;X_{i+1}^n|Y,X_1^{i-1}) =0.\label{Ap4}
%\end{align}
%
%Again since 
%\begin{align}
%I(X_1^{i};X_{i+1}^n|Y)
%=I(X_i;X_{i+1}^n|Y)+
%I(X_1^{i-1};X_{i+1}^n|Y,X_i),
%\end{align}
%and by \eqref{Ap4}, we obtain 
%\begin{align}
%I(X_i;X_{i+1}^n|Y)=0,\ 2\le i\le n-1.\label{Ap6}
%\end{align}

Finally, to complete the proof, consider
\begin{eqnarray*}
        H(X_1^n|Y)&=& H(X_2^n|Y,X_1) + H(X_1|Y)  \\
        &=& H(X_2^n|Y) -I(X_1;X_2^n|Y)+ H(X_1|Y) \\
        &=& H(X_3^n|Y,X_2)+H(X_2|Y)-I(X_1;X_2^n|Y)+H(X_1|Y)\\
        &\vdots& \\
        &=& \sum_{i=1}^n H(X_i|Y) -\sum_{i=1}^{n-1} I(X_i;X_{i+1}^{n}|Y), 
\end{eqnarray*}
which by \eqref{lajtnr} implies that 
\[
H(X_1^n|Y)=\sum_{i=1}^n H(X_i|Y),
\]
i.e., $X_1, X_2, \ldots, X_n$ are mutually independent conditioning on $Y$.
\end{proof}

\newpage

\end{document}